\documentclass[12pt]{article}

\usepackage{amsthm,color}
\usepackage{amsfonts}
\usepackage{graphicx}
\usepackage{mathrsfs}
\usepackage{amsmath}
\usepackage{amssymb}
\usepackage{float}
\usepackage{natbib}
\usepackage{booktabs}
\usepackage{url}
\usepackage{multirow}
\usepackage{arydshln}
\usepackage{threeparttable}
\usepackage{courier}
\usepackage{authblk} 

\usepackage{multicol}
\usepackage{latexsym}
\usepackage{psfrag}
\usepackage[usenames,dvipsnames]{xcolor}

\usepackage{breakcites}
\usepackage{bbm}
\usepackage{epsfig,epstopdf}
\usepackage[colorlinks, linkcolor=black, citecolor=black]{hyperref}

\usepackage[utf8]{inputenc}
\usepackage[english]{babel}
\usepackage{caption}
\usepackage{setspace}

\usepackage{sectsty}
\sectionfont{\fontsize{12}{15}\selectfont}
\subsectionfont{\fontsize{12}{15}\selectfont}

\newtheorem{theorem}{Theorem}%[section]
\newtheorem{corollary}{Corollary}
\newtheorem{lemma}[theorem]{Lemma}

\newtheorem{example}{Example}

\def\A{\boldsymbol{A}}
\def\a{\boldsymbol{a}}

\def\B{\boldsymbol{B}}
\def\C{\boldsymbol{C}}
\def\e{\boldsymbol{e}}
\def\g{\boldsymbol{g}}
\def\G{\boldsymbol{G}}
\def\h{\boldsymbol{h}}
\def\H{\boldsymbol{H}}
\def\I{\boldsymbol{I}}
\def\J{\boldsymbol{J}}

\def\M{\mathcal{M}}
\def\q{\boldsymbol{q}}
\def\Q{\boldsymbol{Q}}
\def\S{\boldsymbol{S}}
\def\t{\boldsymbol{t}}
\def\T{\boldsymbol{T}}
\def\bu{\boldsymbol{u}}
\def\U{\boldsymbol{U}}
\def\v{\boldsymbol{v}}
\def\V{\boldsymbol{V}}
\def\W{\boldsymbol{W}}
\def\x{\boldsymbol{x}}
\def\X{\boldsymbol{X}}

\def\bbeta{\boldsymbol{\beta}}

\def\btheta{\boldsymbol{\theta}}
\def\bstheta{\mbox{\scriptsize \boldmath $\theta$}}

\def\etab{\boldsymbol{\eta}}
\def\hetab{\hat{\etab}}

\def\bOmega{\boldsymbol{\Omega}}

\def\bSigma{\boldsymbol{\Sigma}}
\def\bGamma{\boldsymbol{\Gamma}}
\def\bgamma{\boldsymbol{\gamma}}
\def\bpsi{\boldsymbol{\psi}}
\def\bspsi{\mbox{\scriptsize \boldmath $\psi$}}

\def\bnu{\boldsymbol{\nu}}

\def\Sumij{\sum_{i=0}^1 \sum_{j=1}^{n_i}}
\def\Sumj{\sum_{j =1}^{n_1}}
\def\ba{\left( \begin{array}}
\def\ea{\end{array} \right)}
\def\ls{\left(}
\def\rs{\right)}
\def\lb{\left\{ }
\def\rb {\right\} }
\def\Lm{\left[ }
\def\Rm{\right] }

\begin{document}

{\centering {\large {\bf Semiparametric empirical likelihood inference with estimating equations under density ratio models}}}
\bigskip

\centerline{Meng Yuan, \ Pengfei Li \ and \ Changbao Wu\footnote{Meng Yuan is  doctoral student, Pengfei Li is Professor and Changbao Wu is Professor, Department of Statistics and Actuarial Science, University of Waterloo, Waterloo ON N2L 3G1,Canada (E-mails: {\em m33yuan@uwaterloo.ca}, \ {\em pengfei.li@uwaterloo.ca} \ and \ {\em cbwu@uwaterloo.ca}).}}

\bigskip

\bigskip

\hrule

{\small
\begin{quotation}
\noindent
The density ratio model (DRM)  provides a flexible and useful platform for combining information from multiple sources.
In this paper, we consider statistical inference under two-sample DRMs with 
additional parameters defined through and/or additional  auxiliary information expressed as  estimating equations. 
We examine the asymptotic properties of the maximum empirical likelihood estimators (MELEs) of the unknown parameters in the DRMs and/or defined through estimating equations, and establish the chi-square limiting distributions for the empirical likelihood ratio (ELR) statistics.
We show that the asymptotic variance  of the MELEs of the unknown parameters
does not decrease if one estimating equation is dropped. 
Similar properties are obtained for inferences on the cumulative distribution function and quantiles of each of the populations involved. 
We also propose an ELR test for the validity and usefulness of the auxiliary information.  
Simulation studies show that correctly specified estimating equations for the auxiliary information result in more efficient estimators and shorter confidence intervals. Two real-data examples are used for illustrations. 

\vspace{0.3cm}

\noindent
{\bf Keywords:}  Auxiliary information, density ratio model, empirical likelihood, estimating equations.
\end{quotation}
}

\hrule

\bigskip

\bigskip

\section{Introduction}
\label{gee_intro}
\subsection{Problem setup}

Suppose we have two independent random samples $\{X_{01},\ldots,X_{0n_0}\}$ and $\{X_{11},\ldots,X_{1n_1}\}$ from two populations with cumulative distribution functions (CDFs) $F_0$ and $F_1$, respectively. 
The dimension of $X_{ij}$ can be one or greater than one. 
We assume that the CDFs $F_0$ and $F_1$ are linked through a semiparametric density ratio model (DRM) \citep{Anderson1979,Qin2017},
\begin{equation}
\label{drm_def}
    dF_1(x) = \exp\{\alpha+\bbeta^\top\q(x)\}dF_0(x)= \exp\{\btheta^\top\Q(x)\}dF_0(x),
\end{equation}
where $dF_i(x)$ denotes the density of $F_i (x)$ for $i = 0$ and $1$; 
$\btheta=(\alpha,\bbeta^\top)^\top$ are the unknown parameters for the DRM; 
 $\Q(x) = (1,\q(x)^\top)^\top$ with $\q(x)$ being a prespecified, nontrivial function of dimension $d$; 
 and the baseline distribution $F_0$ is unspecified.
We further assume that  
certain auxiliary information about   $F_0$, $F_1$,  and  $\btheta$ is available in the form of 
functionally independent unbiased estimating equations (EEs): 
\begin{equation}
\label{ee}
    E_0\{ \g(X;\bpsi,\btheta)\} = {\bf 0},
\end{equation}
where $E_0(\cdot)$ refers to the expectation operator with respect to $F_0$, $\bpsi$ consists of additional parameters of interest and has dimension $p$,   $\g(\cdot;\cdot)$ is $r$-dimensional, and $r \geq p$. 
In this paper, our goal is twofold:
\begin{enumerate}
\item[(1)] we develop new and general semiparametric inference procedures for $(\bpsi,\btheta)$ and $(F_0,F_1)$ along with their quantiles under Model (\ref{drm_def}) with unbiased EEs in (\ref{ee}); 
\item[(2)] we propose a new testing procedure on the validity of (\ref{ee}) under Model (\ref{drm_def}), which leads to a practical validation  method on the usefulness of the auxiliary information. 
\end{enumerate}

 The semiparametric DRM in \eqref{drm_def}  provides a flexible and useful platform for combining information from multiple sources \citep{Qin2017}. 
 It enables us to utilize information from both $F_0$ and $F_1$
 to improve inferences on the unknown model parameters and the summary population quantities of interest \citep{chen2013quantile,cai2017hypothesis,zhuang2019semiparametric}. 
With the unspecified $F_0$, 
the DRM embraces many commonly used statistical models including distributions of exponential families \citep{kay1987transformations}. 
For example, 
when $\q(x)=\log x$, the DRM includes two log-normal distributions with the same variance with respect to the log-scale, as well as two gamma distributions with the same scale parameter; 
when $\q(x)=x$, it includes two normal distributions with different means but a common variance and  two exponential distributions. 
 \cite{jiang2012inference} observed that the DRM is actually broader than Cox proportional hazard models.
Moreover, it has a natural connection to the well-studied logistic regression if one treats $D$ = 0 and 1 as 
indicators for the observations from $F_0$ and $F_1$, respectively.
Among others,  \cite{Anderson1979} and \cite{qin1997goodness} 
noticed that the DRM is equivalent to the logistic regression model via the
fact that
\begin{equation}
\label{drm_logis}
P(D=1|x)=\frac{\exp\{\alpha^*+\bbeta^\top\q(x)\} }{ 1+\exp\{\alpha^*+\bbeta^\top\q(x)\}  },
\end{equation}
where $\alpha^*=\alpha+\log\{P(D=1)/P(D=0)\}$. 

The EEs in (\ref{ee}) play two important roles. First, they can be used to define many important summary population quantities such as the ratio of the two population means,
the centered and uncentered moments, the generalized entropy class of inequality measures, the CDFs, and the quantiles of each population.
See Example \ref{example1} below and Section 1 of the Supplementary Material for more examples. 
Second, they provide a unified platform for the use of auxiliary information.  With many data sources being increasingly available, it becomes more feasible to access auxiliary information, and using such information to enhance statistical inference is an important and active research topic in many fields. Calibration estimators, which are widely used in survey sampling, missing data problems and causal inference, rely heavily on the use of auxiliary information; 
see \cite{Wu2020book} and the references therein. 
%
%\citep{deville1992calibration,chen1993empirical,chen1999pseudo}. 
Many economics problems can be addressed using similar methodology. For instance, knowledge of the moments of the marginal distributions of economic variables from census reports can be used in combination with microdata to improve the parameter estimates of microeconomic models \citep{imbens1994combining}. 
Examples \ref{example2} and \ref{example3} below illustrate the use of auxiliary information through EEs in the form of (\ref{ee}).

\begin{example} (The mean ratio of two populations)
\label{example1}
The ratio of the means of two positive skewed distributions is often of interest in biomedical research
\citep{Zhou1997,wu2002likelihood}. 
Let $\mu_0$ and $\mu_1$ be the means with respect to $F_0$ and $F_1$, respectively. 
Further, let $\delta=\mu_1/\mu_0$ denote the mean ratio of the two populations. 
For inference on  $\delta$, 
a common assumption  is that both distributions are lognormal. 
To alleviate the risk of parametric assumptions,  
we could use the DRM in (\ref{drm_def}) with $\q(x)=\log x$ or  $\q(x)=(\log x,\log ^2 x)^\top$
depending on whether or not  the variances with respect to the log-scale are  the same. 
Then, under the DRM \eqref{drm_def}, $\delta$ can be defined through the following EE: 
\begin{equation*}
%\label{ee.mr}
g(x;\bpsi,\btheta)= \delta x- x \exp\{\btheta^\top\Q(x)\}, 
\end{equation*}
with $\bpsi=\delta$. 
When additional information is available, we may add more EEs to improve the estimation efficiency; 
see Section \ref{gee_sim1} for further detail.
\end{example}

\begin{example} (Retrospective case-control studies with auxiliary information)
\label{example2}
Consider a retrospective case-control study with
 $D=1$ or 0 representing diseased or disease-free status, 
 and $X$ representing the collection of risk factors. 
 Note that the two samples are collected retrospectively, given the diseased status. 
Let $F_0$ and $F_1$ denote the CDF of $X$ given $D=0$ and $D=1$, respectively.
Assume that the relationship between $D$ and $X$
can be modeled by the logistic regression specified in \eqref{drm_logis}.  
Then, using the equivalence between the DRM and the logistic regression discussed above, 
$F_0$ and $F_1$ satisfy the DRM   \eqref{drm_def}. 

\cite{qin2014using} used covariate-specific disease prevalence information to improve the power of case-control studies. Specifically, let 
$X=(Y,Z)^\top$ with $Y$ and $Z$ being two risk factors. Assume that we know the disease prevalence at various levels of $Y$:  $\phi(a_{l-1},a_l)=P(D=1|a_{l-1}<Y\leq a_l)$ for $l=1,\ldots,k$. 
Let $\pi=P(D=1)$ be the overall disease prevalence. 
Using Bayes' formula, the information in the $\phi(a_{l-1},a_l)$'s can be summarized as 
$E_0\{\g(X;\bpsi,\btheta)\}={\bf 0}$, where  $\bpsi=\pi$ and the $l$th component of $\g(x;\bpsi,\btheta)$ is 
\begin{equation} 
\label{ee.qin}
g_l(x;\bpsi,\btheta) = I(a_{l-1}< x \leq a_l)\Lm \frac{\pi}{1-\pi}\exp\{\btheta^\top\Q(x)\}  - \frac{\phi(a_{l-1},a_l)}{1-\phi(a_{l-1},a_l)} \Rm. 
\end{equation}

\cite{chatterjee2016constrained} improved the internal study by using summary-level information from an external study. Suppose $X=(Y^\top,Z^\top)^\top$, where $Y$ is available for both the internal and external studies, while $Z$ is available for only the internal study. 
Assume that the external study provides the true coefficients $(\alpha^*_Y,\bbeta_{Y}^*)$
for the following logistic regression model, which may not be the true model:  
$$
h(Y;\alpha_Y,\bbeta_Y)=P(D=1|Y)
=\frac{\exp(\alpha+\bbeta_{Y}^\top Y)}{1+\exp(\alpha+\bbeta_{Y}^\top Y)}.
$$  
This assumption is reasonable when  the total sample size $n=n_0+n_1$ satisfies 
$n/n_E\to 0$, where $n_E$ is the total sample size in the external study. 
Further, assume that the joint distribution of $(D,X)$ is the same for both the internal and external studies. 
Let $h(y)= h(y;\alpha_Y^*,\bbeta_Y^*)$.
In Section 2 of the Supplementary Material, we argue that if the external study is a prospective case-control study, then 
$E_0\{\g(X;\bpsi,\btheta)\}={\bf 0}$, 
where 
\begin{equation} 
\label{ee.pros}
\g(x;\bpsi,\btheta) =[ -(1-\pi) h(y)  +\pi  \exp\{\btheta^\top\Q(x)\}\{1-h(y)\} ](1,y^\top)^\top
\end{equation}
with $\bpsi=\pi$; 
if the external study is a retrospective case-control study, then  $E_0\{\g(X;\btheta)\}={\bf 0}$, where
\begin{equation} 
\label{ee.retro}
\g(x;\btheta) =[ -(1-\pi_E) h(y)  +\pi_E  \exp\{\btheta^\top\Q(x)\}\{1-h(y)\} ](1,y^\top)^\top
\end{equation}
with $\pi_E$ being the proportion of diseased individuals in the external study.

%
%
%Further, assume that the information from the external study can be summarized as 
%$E\{{\bf h}(D,Y)\}={\bf 0}$, where the expectation is taken with respect to the joint distribution of $(D,Y)$. 
%Then equivalently, we have $E_0\{\g(X;\bpsi,\btheta)\}={\bf 0}$ with 
%\begin{equation} 
%\label{ee.chatt}
%\g(X;\bpsi,\btheta) =\frac{1 - \pi}{\pi} {\bf h} (0 ,Y) +  \exp\{\btheta^\top\Q(x)\} {\bf h}(1 ,Y),
%\end{equation}
%where $\bpsi=\pi$. 

%Auxiliary information on disease prevalence with varying levels of smoking intensity is available to use. Let $a_0 < a_1 < \cdots < a_L$ be a partition of the X space, we assume the disease prevalences are 
%$\phi(a_{l-1},a_l)$ for $i = 1,\cdots,L$. Our proposed method deals with this problem by taking $\g(\t;\bxi)  = (g_1(\t;\bxi),\cdots,g_L(\t;\bxi))^\top$ and 
%$$g_l(\t;\bxi) = I(a_{l-1}< x \leq a_l)\Lm \frac{\pi}{1-\pi}\exp\{\bxi^\top\Q(\t)\}  - \frac{\phi(a_{l-1},a_l)}{1-\phi(a_{l-1},a_l)} \Rm,$$ 
%where the parameter $\pi = P(Y=1) = 1- P(Y=0)$ could be known or unknown.

%
%\begin{equation*}
%    \g(\t;\bxi,\pi) = \frac{1 - \pi}{\pi}\u(Y = 0 ,\x;\btheta^{\ast}) +  \exp\{\bxi^\top\Q(\t)\} \u(Y = 1 ,\x;\btheta^{\ast}). 
%\end{equation*}

\end{example}

\begin{example} (A two-sample problem with common mean)
\label{example3}
\cite{tsao2006empirical} considered two populations with a common mean. 
This type of problems occurs when  two ``instruments"  are used to collect data on a common response variable, and these two  instruments  are believed to have no systematic biases but differ in precision. 
The observations from the two instruments then form two samples with a common population mean. 
In the literature, there has been much interest in using the pooled sample to improve inferences. 
A common assumption is that the two samples follow normal distributions with a common mean but different variances. 
To gain robustness with respect to the parametric assumption, we may use the DRM \eqref{drm_def} with $\q(x)=(x,x^2)^\top$. 
Under this model, the common-mean assumption can be incorporated via the EE: 
\begin{equation}
\label{example3.ee}
E_0\{ X\exp\{\btheta^\top\Q(X)\} - X\}=0. 
\end{equation}
% 
%
%The assumption of common mean can also been achieved when one fails to reject the null hypothesis that means of precise measurements and `inaccurate' measurements are equal.
%The function $\g(t;\etab) = t\exp\{\etab^\top\Q(t)\} - t$ is chosen to express the common mean assumption. 
%More examples of auxiliary information in the presence of measurement error are discussed in \cite{zhong2000empirical}.

\end{example}

%Note that $\Q(x)^\top = (1,\q(x)^\top).$ %Meanwhile, the baseline distribution $F_0(x)$ is completely unspecified. 
%For notational simplicity, we denote the pooled sample as $\{\t_1,\cdots, \t_{n}\} = \{X_{01},\cdots,X_{0n_0},X_{11},\cdots,X_{1n_1}\}$, where the total sample size $n = n_0+n_1$. 
\subsection{Literature review}
The DRM has been investigated extensively because of its flexibility and efficiency. 
For example, 
it has been applied to  multiple-sample 
hypothesis-testing problems \citep{Fokianos2001, cai2017hypothesis,WANG2017, WANG2018} and quantile and quantile-function estimation \citep{zhang2000quantile, chen2013quantile}. 
These inference problems can be viewed as special cases, without auxiliary information, of the first goal to be achieved in this paper. 
%
%the receiver operating characteristic analysis \citep{chen2016using,qin2003using}. 
%It has been used to multiple sample 
%hypothesis testing problems \citep{Fokianos2001, cai2017hypothesis,WANG2017, WANG2018} and quantile and quantile-function estimation \citep{zhang2000quantile, chen2013quantile}. 
%These inference problems can be viewed as special cases of our first aim without auxiliary information. 
Other applications of the DRM include 
receiver operating characteristic (ROC) analysis \citep{qin2003using,chen2016using,yuan2021semiparametric}, 
inference under semiparametric mixture models \citep{qin1999empirical, Zou2002, li2017semiparametric},
the modeling of multivariate extremal distributions \citep{Davison2014}, 
and
dominance index estimation \citep{zhuang2019semiparametric}.
Recently, \cite{li2018comparison} studied maximum empirical likelihood 
estimation (MELE) and empirical likelihood ratio (ELR) based confidence intervals (CIs)
for a parameter defined as 
$\bpsi=\int  u(x;\btheta) dF_0(x)$, where $u(\cdot;\cdot)$ is a one-dimensional function. 
They did not consider auxiliary information, and because of
the specific form of $\bpsi$, their results do not apply to the mean ratio  discussed in Example \ref{example1}. 
\cite{zhang2020empirical} investigated the ELR statistic for quantiles under the DRM
and showed that the ELR-based confidence region of the quantiles is preferable to 
the Wald-type confidence region. 
Again, they did not consider auxiliary information. 
In summary, the existing literature on DRMs focuses on cases where there is no auxiliary information, and furthermore, there is no general theory available to handle parameters defined through the EEs in \eqref{ee}. 
% 
%In this paper, our general theory will cover both cases $r=p$ and $r>p$. 
 
Using  the connection of the DRM to the logistic regression model, 
 \cite{qin2014using} studied the MELE of $\btheta$ and the ELR statistic for testing a parameter in $\btheta$
under Model \eqref{drm_def} with the unbiased EEs in  \eqref{ee.qin}. 
\cite{chatterjee2016constrained} proposed constrained maximum 
likelihood estimation for the unknown parameters in the internal study
using summary-level information from an external study. 
In Section 2 of the Supplementary Material, we argue that
their results  are applicable to the MELE of $\btheta$
under Model \eqref{drm_def} with the unbiased EEs in \eqref{ee.pros} 
but not to the MELE of $\btheta$
under Model \eqref{drm_def} with the unbiased EEs in \eqref{ee.retro}.
Furthermore, they did not consider the ELR statistic for the unknown parameters. 
\cite{qin2014using} and \cite{chatterjee2016constrained} focused on how to use auxiliary information to improve inference on the unknown parameters, and they did not check the validity of that information 
or explore inferences on the CDFs $(F_0,F_1)$  and their quantiles.

\subsection{Our contributions}
%
%In order to generalize the EL method under DRMs, we assume that the information about distribution $F_0,F_1$ and interested parameter $\bpsi$ are available in a known form of EEs 
%\begin{equation}
%\label{ee}
%    \int \g(X;\bpsi,\btheta)dF_0(x) = {\bf 0},
%\end{equation}
%Parameter $\bpsi$ is of $p$ dimension while the function $\g(\cdot)$ is $r$-dimensional.  
%Note that $r \geq p$. 
With two-sample observations from the DRM \eqref{drm_def}, 
we use 
the empirical likelihood (EL) of \cite{owen1988empirical, owen2001empirical}  to incorporate the unbiased EEs in \eqref{ee}. 
We show that the MELE of $(\bpsi,\btheta)$ is asymptotically normal, and its asymptotic variance will not decrease when an EE in \eqref{ee} is dropped. 
We also develop an ELR statistic for testing a general hypothesis about $(\bpsi,\btheta)$, and show that it has a $\chi^2$ limiting distribution under the null hypothesis. 
The result can be used to construct the ELR-based confidence region for $(\bpsi,\btheta)$.
Similar results are obtained for inferences on $(F_0,F_1)$ and their quantiles. 
Finally,  we construct an ELR statistic with the $\chi^2$ null limiting distribution 
to test the validity of some or all of the EEs in \eqref{ee}.

We make the following observations: 
\begin{enumerate}
\item[(1)] Our results on the two-sample DRMs contain more advanced development than those in  \cite{qin1994empirical} for the one-sample case.
\item[(2)] Our inferential framework and theoretical results are very general. 
The results in \cite{qin2014using} and \cite{chatterjee2016constrained} for case-control studies are special cases of our theory for an appropriate choice of $\g(x;\bpsi,\btheta)$ in \eqref{ee}. 
Our results are also applicable to cases that are not covered by these two earlier studies, e.g., Example \ref{example2} with the EEs in \eqref{ee.retro} and Example \ref{example3}. 
\item[(3)] Our proposed ELR statistic, to our best knowledge, is the first formal procedure to test the validity of auxiliary information under the DRM or for case-control studies.
\item[(4)] Our proposed inference procedures for $(F_0,F_1)$ and their quantiles in the presence of auxiliary information are new to the literature. 
\end{enumerate}
 
The rest of this paper is organized as follows. In Section \ref{gee_result}, we develop the EL inferential procedures and study the asymptotic properties of the MELE of $(\bpsi,\btheta)$. 
We also investigate the ELR statistics for $(\bpsi,\btheta)$ and for testing the validity of the EEs in \eqref{ee}.
In Section \ref{gee_cdf.quantile}, we discuss inference procedures for $(F_0,F_1)$ and their quantiles. 
Simulation results are reported in Section \ref{gee_sim}, and two real-data examples are presented in Section \ref{gee_real}. We conclude the paper with a discussion in Section \ref{gee_dis}.  For convenience of presentation, all the technical details are given in the Supplementary Material.

\section{Empirical likelihood and inference on $(\bpsi,\btheta)$}
\label{gee_result}

In this section, we first develop the EL formulation under the DRM \eqref{drm_def} with the unbiased EEs in  \eqref{ee}. 	
With two samples $\{X_{01},\ldots,X_{0n_0}\}$ and $\{X_{11},\ldots,X_{1n_1}\}$
from $F_0$ and $F_1$, respectively, the full likelihood is 
$$
\prod_{i=0}^1\prod_{j=1}^{n_i} dF_i(X_{ij}). 
$$
Under the one-sample EL formulation of \cite{owen2001empirical}, the baseline distribution function $F_0(x)$ would have been modeled as 
$F_0(x) = \sum_{j = 1}^{n_{0}}p_{j}I(X_{0j}\le x)$, where $p_{j}=dF_0(X_{0j})$ for $j=1,\ldots,n_{0}$. Under the two-sample DRM \eqref{drm_def}, 
we use the combined sample to model the baseline function $F_0(x)$ as 
\begin{equation} \label{el.F0}
F_0(x)= \sum_{i=0}^1\sum_{j = 1}^{n_{i}}p_{ij}I(X_{ij}\le x),
\end{equation}
where $p_{ij}=dF_0(X_{ij})$ for $i=0,1$ and $j=1,\ldots,n_{i}$.
Note that the size of the combined sample is $n=n_0+n_1$. 
With (\ref{el.F0}) and under the DRM \eqref{drm_def}, the EL function is given by 
\begin{equation}
\label{L}
    \mathcal{L}_n = \left\{\prod_{i=0}^1\prod_{j=1}^{n_i} p_{ij}\right\}\left\{ \prod_{j=1}^{n_1}\btheta^\top\Q(X_{1j})\right\}.
\end{equation}
The feasible $p_{ij}$'s satisfy two sets of constraints given by 
\begin{equation}
\label{cdf_cons}
\mathcal{C}_1 = \lb (F_0, \btheta): p_{ij} > 0,~
\Sumij p_{ij}=1,~
\Sumij p_{ij}\exp\{\btheta^\top\Q(X_{ij})\}=1 \rb
\end{equation}
and 
\begin{equation}
    \label{ee_cons}
    \mathcal{C}_2 = \lb (F_0, \bpsi, \btheta): \Sumij p_{ij} \g(X_{ij};\bpsi,\btheta)= {\bf 0} \rb,
\end{equation}
where the set of constraints $\mathcal{C}_1$  ensures that  $F_0$ and $F_1$ are CDFs
and the set of constraints $  \mathcal{C}_2$ is induced by the EEs  in  \eqref{ee}. 

Using the Lagrange multiplier method and for the given $\bpsi$ and $\btheta$, it can be shown that the  maximizer of the EL function is given by  
\begin{equation*}
%\label{pij}
    p_{ij} = \frac{1}{n}\frac{1}{1 + \lambda\left[\exp\{\btheta^\top\Q(X_{ij})\} -1\right] + \bnu^\top \g(X_{ij};\bpsi,\btheta)},
\end{equation*}
where the Lagrange multipliers $\lambda$ and $\bnu= (\nu_1,\cdots,\nu_r)^\top$ are the solutions to the following set of $r+1$ equations:
\begin{eqnarray}
\label{solution_lambda}
\Sumij \frac{\exp\{\btheta^\top\Q(X_{ij})\}-1}{1 + \lambda\left[\exp\{\btheta^\top\Q(X_{ij})\} -1\right] + \bnu^\top \g(X_{ij};\bpsi,\btheta)}  = 0,\\
\label{solution_z}
\Sumij \frac{\g(X_{ij};\bpsi,\btheta)}{1 + \lambda\left[\exp\{\btheta^\top\Q(X_{ij})\} -1\right] + \bnu^\top \g(X_{ij};\bpsi,\btheta)}  = {\bf 0}.
\end{eqnarray}
The profile empirical log-likelihood of $(\bpsi,\btheta)$ is given by 
\begin{equation*}
    %\label{ell_1}
    \ell_n(\bpsi,\btheta) = -\Sumij \log \left\{1 + \lambda\left[\exp\{\btheta^\top\Q(X_{ij})\} -1\right] + \bnu^\top \g(X_{ij};\bpsi,\btheta)\right\} + \Sumj \btheta^\top\Q(X_{1j}).
\end{equation*}
%
%In order to solve MELEs $ (\hat{\boldsymbol{P}},\htheta,\hbeta)$, We first follow a similar profile procedure of \cite{qin1994empirical} and use the method of Lagrange multipliers to obtain the profile log-EL (PEL) functions of $\bpsi$ and $\btheta$. 
%When fixing parameters $\bpsi$ and $\btheta$, the maximum of $\mathcal{L}(\boldsymbol{P},\btheta)$ subject to constraints $\mathcal{C}_1$ and $\mathcal{C}_1$  is achieved at 
%\begin{equation*}
%    p_{ij} = \frac{1}{n}\frac{1}{1 + \lambda\left[\exp\{\btheta^\top\Q(X_{ij})\} -1\right] + \bnu^\top \g(X_{ij};\bpsi,\btheta)},
%\end{equation*}
%where Lagrange multipliers $\lambda$ and $\bnu= (\nu_1,\cdots,\nu_r)^\top$ are the solutions to the following equations:
%\begin{eqnarray}
%\label{solution_lambda}
%\Sumij \frac{\exp\{\btheta^\top\Q(X_{ij})\}-1}{1 + \lambda\left[\exp\{\btheta^\top\Q(X_{ij})\} -1\right] + \bnu^\top \g(X_{ij};\bpsi,\btheta)}  = 0,\\
%\label{solution_z}
%\Sumij \frac{\g(X_{ij};\bpsi,\btheta)}{1 + \lambda\left[\exp\{\btheta^\top\Q(X_{ij})\} -1\right] + \bnu^\top \g(X_{ij};\bpsi,\btheta)}  = {\bf 0}.
%\end{eqnarray}
%Hence, the PEL of $\bpsi$ and $\btheta$ is given by
%\begin{equation}
%    \label{ell_1}
%    \ell(\bpsi,\btheta) = -\Sumij \log \left\{1 + \lambda\left[\exp\{\btheta^\top\Q(X_{ij})\} -1\right] + \bnu^\top \g(X_{ij};\bpsi,\btheta)\right\} + \Sumj \btheta^\top\Q(X_{1j}).
%\end{equation}
The MELEs of $\bpsi$ and $\btheta$ are then defined as 
$(\hat{\bpsi},\hat{\btheta})= \arg\max_{\bpsi,\btheta} \ell_n(\bpsi,\btheta)$.
%After obtaining MELEs $\htheta$ and $\hbeta$, the MELE of $p_{ij}$ is given by 
%\begin{equation*}
%    \hat{p}_{ij} = n^{-1}\left\{1 + \hat\lambda\left[\exp\{\hbeta^\top\Q(X_{ij})\} -1\right] + \hnu^\top \g(X_{ij};\htheta,\hbeta)\right\}^{-1}.
%\end{equation*}

We now establish the asymptotic distribution of $(\hat{\bpsi},\hat{\btheta})$.
Let $\etab = (\bpsi^\top,\btheta^\top)^\top$ be the vector of parameters and $\etab^*$ be the true value of $\etab$. Let $\lambda^*=n_1/n$.  
We further define
\begin{eqnarray*}
&& \omega(x; \btheta) = \exp\left\{ \btheta^\top \Q(x)\right\},~~\omega(x) = \omega(x; \btheta^*),~~
~~h(x) = 1 + \lambda^* \left\{ \omega(x) - 1 \right\},
\\
&&h_1(x) = \frac{\lambda_0\omega(x)}{h(x)},~~\G(x;\etab) = (\omega(x; \btheta) - 1, \g(x;\btheta,\bbeta)^\top)^\top,~~~\G(x)=\G(x;\etab^*),\\
&&
\A_{\btheta\btheta} = (1 - \lambda_0) E_0 \lb h_1(X)\Q(x)\Q(x)^\top \rb,
\\
&&\A_{\btheta\bu} = \A_{\bu\btheta}^\top = E_0 \lb \frac{\partial \G(X;\etab_0)}{ \partial \btheta} \rb^\top - E_0 \lb h_1(X) \Q(x) \G(X)^\top\rb,\\
&&\A_{\bpsi\bu} = \A_{\bu\bpsi}^\top = E_0 \lb \frac{\partial \G(X;\etab_0)}{ \partial \bpsi} \rb^\top ,~~~
\A_{\bu\bu} = E_0 \lb \frac{\G(X)\G(X)^\top}{h(X)} \rb.
\end{eqnarray*}
Noting that $\omega(\cdot)$, $h(\cdot)$,  $h_1(\cdot)$ and $G(\cdot)$ depend on $\bpsi^*$ and/or $\btheta^*$, 
we drop these redundant parameters for notational simplicity.
\begin{theorem}
\label{theorem_eta}
Assume that the regularity conditions in the Appendix are satisfied. As the total sample size $n = n_0 + n_1$ goes to infinity, we have
 \begin{equation*}
     n^{1/2} (\hetab - \etab^*) \to 
     N \ls {\bf 0}, \J^{-1} \rs
 \end{equation*} in distribution, where
 \begin{equation*}
    \J = \U\V^{-1}\U^\top,~~  \U = \ba{cc} {\bf 0}& \A_{\bpsi\bu}\\
    \A_{\btheta\btheta}&\A_{\btheta\bu}\ea,~~ \text{and}~~\V = \ba{cc}\A_{\btheta\btheta}&{\bf 0}\\ {\bf 0}& \A_{\bu\bu} \ea.
 \end{equation*}
\end{theorem}

In the absence of the constraints $\mathcal{C}_2$ in \eqref{ee_cons}, we can 
maximize the EL function in \eqref{L} with respect only to the CDF constraints $\mathcal{C}_1$ in \eqref{cdf_cons}
to obtain the MELE $\tilde\btheta$ of $\btheta$.  
\cite{qin1997goodness} and \cite{keziou2008empirical} noticed that 
$\tilde\btheta$ equivalently maximizes the following dual likelihood: 
\begin{equation}
\label{del}
   \ell_{nd}(\btheta) = -\Sumij \log \left\{1+\lambda^* \left[ \exp \left\{ \btheta^\top\Q(X_{ij})\right\} -1\right]\right\} + \Sumj \left\{\btheta^\top\Q(X_{1j})\right\}.
\end{equation}
That is,  $\Tilde{\btheta} = \arg\max_{\btheta} \ell_{nd}(\btheta) $. 

\begin{corollary}
\label{corol_1}
Under the conditions of Theorem 1, 
\begin{itemize}
    \item [(a)] if $ r =p $, the asymptotic variance of $n^{1/2}(\hat{\btheta} - \btheta^*)$ is the same as that of $n^{1/2}(\Tilde{\btheta} - \btheta^*)$; 
    \item [(b)] if $ r>p $, the asymptotic variance matrix of $n^{1/2} (\hetab - \etab^*)$ cannot decrease if an EE in (\ref{ee}) is dropped. 
\end{itemize}
\end{corollary}

We provide some further comments on the results presented in Corollary \ref{corol_1}. First, when the dimensions of the parameters $\bpsi$ and the EEs are equal, 
we can solve 
$$\int \g(X; \bpsi,\Tilde{\btheta}) d\tilde F_0(x) = {\bf 0}$$
to get the estimator $\Tilde{\bpsi}$ of $\bpsi$, 
where $\tilde F_0(x)$ is the MELE of $F_0$ without the constraints $\mathcal{C}_2$ in \eqref{ee_cons}. 
Because of the result in Corollary \ref{corol_1}(a), the estimators $\Tilde{\bpsi}$ and $\hat{\bpsi}$ share the same asymptotic property. 
Second, Corollary \ref{corol_1}(b) indicates that additional auxiliary information leads to more efficient estimation of $\etab$. 

The proposed semiparametric method provides a way to find the point estimator of the unknown parameters, 
which has the asymptotic normality analogue to the parametric estimator. 
The semiparametric framework also creates a natural platform for hypothesis tests using the ELR statistic.
We consider a general null hypothesis
\begin{equation*}
%\label{H}
   H_0: \H(\etab) = {\bf 0},
\end{equation*}
where the function $\H(\cdot)$ is $q\times1$ with $q\leq p+d+1$, and the derivative of this function is of rank $q$. 
%For a compact presentation, we denote $\h(\etab_0) = \partial \H(\etab_0)/\partial \etab^\top$ . 
This null hypothesis forms a third set of constraints
\begin{equation*}
    %\label{hypo_cons}
    \mathcal{C}_3 = \lb \etab=(\bpsi^\top,\btheta^\top)^\top: \H(\etab) = {\bf 0} \rb.
\end{equation*}
The ELR statistic for testing $H_0$ is then defined as
\begin{equation*}
    %\label{LR_def} 
    R_n = 2\lb \sup_{\bpsi,\btheta} \ell_n(\bpsi,\btheta) - \sup \limits_{\etab \in \mathcal{C}_3}  \ell_n(\bpsi,\btheta)  \rb.
\end{equation*}
%
%To solve for the maximum EL under constraints $\mathcal{C}_1$, $\mathcal{C}_2$, and $\mathcal{C}_3$, a EL is defined based on the PEL \eqref{ell_1} as 
%\begin{equation}
%\label{ell_hypo}
%   \check{\ell}(\etab,\v) = \ell(\etab) + n \v^\top \H(\etab),
%\end{equation}
%where $\v$ is a $q \times 1$ vector of Lagrange multipliers. By maximizing $\check{\ell}(\etab,\v)$, we have the MELE of parameters $(\etab,\v)$ as $(\check{\etab},\check{\v} = \arg\sup_{\etab,\v} \check{\ell}(\etab,\v)$. We examine the asymptotic normality of the MELE $\check{\etab}$ and the result shows in the following theorem. 
%\begin{theorem}
%\label{theorem_eta_cons}
% Under the same conditions of Theorem \ref{theorem_eta}, as n goes to infinity, 
% \begin{equation*}
%     n^{1/2} (\hetab^{\ast} - \etab_0) \to {\rm N} \left( {\bf 0}, \J^{-1} - \J^{-1}\h_0^\top\{\h_0\J^{-1}\h_0^\top)^{-1}\}\h_0J^{-1}\right)
% \end{equation*}
%in distribution, where $\h_0 = \partial \H(\etab_0)/\partial \etab^\top$. In terms of the asymptotic variance, $\check{\etab}$ can be more efficient than $\hetab$. 
%\end{theorem}
%
%Theorem \ref{theorem_eta_cons} confirms that the more information we have, the more efficient estimator we obtain. 
The next theorem establishes the asymptotic property of the ELR statistic $R_n$ under the null hypothesis $H_0$. 
\begin{theorem}
\label{theorem_lr_eta}
Assume  that the conditions  of Theorem \ref{theorem_eta} hold. Under $H_0$,  as $n\to\infty$,  
the ELR statistic $R_n \to \chi^2_q$ in distribution.
\end{theorem}

The result of Theorem \ref{theorem_lr_eta} is very general due to the general form of the function $\H(\cdot)$. 
First, it is applicable to testing problems that focus on some of the parameters in  $\etab$. 
For example, if we wish to test $H_0: \bpsi = \bpsi_0$, we can choose $\H(\etab) =\bpsi - \bpsi_0 $. 
Let $R_n^*(\bpsi)$ be the  ELR function of $\bpsi$. That is, 
$$
R_n^*(\bpsi)=2\lb \sup \limits_{\bpsi,\btheta} \ell_n(\bpsi,\btheta) - \sup \limits_{\btheta}  \ell_n(\bpsi,\btheta)  \rb.
$$ 
Then $R_n^*(\bpsi_0)$  has a  chi-squared null limiting distribution with $p$ degrees of freedom. 
Second, the result can be used to construct confidence regions for some of the parameters in  $\etab$. 
For example, we can construct an ELR-based confidence region for the parameter $\bpsi$ at the  nominal level $1-a$ as
\begin{equation}
\label{lr_CI}
    \{\bpsi: R_n^*(\bpsi)\leq \chi^2_{q,1-a} \},
\end{equation}
where $ \chi^2_{q,1-a}$ is the $100(1-a)$th quantile of the $\chi^2_q$ distribution.

The use of valid auxiliary information leads to improved inference on $\etab$. 
However, if the information is not properly specified in terms of unbiased estimating functions, the resulting estimator of $\etab$ may be biased \citep{qin2014using}. Our last major theoretical result is to construct an ELR statistic for testing the validity and usefulness of the auxiliary information. Let 
\begin{equation}
\label{LRv_def}
    W_n = 2\lb \sup \limits_{(\etab,F_0) \in \mathcal{C}_1} \log \mathcal{L}_n - \sup \limits_{(\etab,F_0) \in \mathcal{C}_1 \cap \mathcal{C}_2} \log \mathcal{L}_n \rb
    = 2\lb \ell_{nd}(\Tilde{\btheta}) - \ell_n(\hat\bpsi,\hat\btheta) \rb.
\end{equation}

\begin{theorem}
\label{theorem_valid}
 Under the conditions of Theorem \ref{theorem_eta} and as $n\to\infty$, we have 
 $W_n \to \chi^2_{r-p}$
 in distribution  if  \eqref{ee} is correctly specified. 
\end{theorem}

We can also test the validity of some but not all of the EEs in  \eqref{ee}. 
To do so,  we partition the EEs in  \eqref{ee}  into two parts:
\begin{equation*}
    \g(x;\bpsi,\btheta) = \ba{c} \g_1(x;\bpsi,\btheta)\\\g_2(x;\bpsi,\btheta)\ea,
\end{equation*}
where $\g_1(\cdot)$ and $\g_2(\cdot)$ are of dimension $r-m$ and $m$ with $r-m\geq p$. 
 We are interested in testing  $H_0: E_0\{ \g_2(X;\bpsi,\btheta)\}={\bf 0}$.
 Let $\ell_{n1}(\bpsi,\btheta)$ be the profile empirical log-likelihood of $(\bpsi,\btheta)$ that uses 
 the auxiliary information only through $E_0\{ \g_1(x;\bpsi,\btheta)\}={\bf 0}$. 
That is, 
\begin{equation*}
    %\label{ell_1}
    \ell_{n1}(\bpsi,\btheta) = -\Sumij \log \left\{1 + \lambda\left[\exp\{\btheta^\top\Q(X_{ij})\} -1\right] + \bnu^\top_1 \g_1(X_{ij};\bpsi,\btheta)\right\} + \Sumj \btheta^\top\Q(X_{1j}),
\end{equation*}
where $\lambda$ and $\bnu_1$ are the solution to 
\begin{eqnarray*}
%\label{solution_lambda1}
\Sumij \frac{\exp\{\btheta^\top\Q(X_{ij})\}-1}{1 + \lambda\left[\exp\{\btheta^\top\Q(X_{ij})\} -1\right] + \bnu^\top_1 \g_1(X_{ij};\bpsi,\btheta)}  = 0,\\
%\label{solution_z1}
\Sumij \frac{\g(X_{ij};\bpsi,\btheta)}{1 + \lambda\left[\exp\{\btheta^\top\Q(X_{ij})\} -1\right] + \bnu^\top_1 \g_1(X_{ij};\bpsi,\btheta)}  = {\bf 0}.
\end{eqnarray*}
Then the ELR statistic for  testing $H_0: E_0\{ \g_2(X;\bpsi,\btheta)\}={\bf 0}$ can be constructed similar to \eqref{LRv_def} as
\begin{equation*}
%\label{LRv1_def}
   W_{n}^* = 2\lb \sup \limits_{\bpsi,\btheta}   \ell_{n1}(\bpsi,\btheta)   -  \sup \limits_{\bpsi,\btheta}   \ell_{n}(\bpsi,\btheta)  \rb.
\end{equation*}
\begin{corollary}
\label{corol_valid}
 Under the conditions of Theorem \ref{theorem_eta} and as $n\to\infty$, we have 
 $
 W_{n}^*\to\chi^2_m
 $
  if $E_0\{ \g_2(X;\bpsi,\btheta)\}={\bf 0}$ is true. 
\end{corollary}

\section{Inferences on CDFs and quantiles}
\label{gee_cdf.quantile}
In this section, we discuss inferences on the CDFs $F_0$ and $F_1$ and their quantiles. 
For convenience of presentation, we assume that the dimension of $X_{ij}$ is one. 

We first construct point estimators of $F_0$ and $F_1$. 
Let $\hat\lambda$ and $\hat\bnu$ be the solutions to \eqref{solution_lambda} and \eqref{solution_z}
with $(\bpsi,\btheta)$ replaced by $(\hat\bpsi,\hat\btheta)$. 
The MELEs of $p_{ij}$ are then given as 
$$
\hat p_{ij}=     \frac{1}{n}\frac{1}{1 +\hat \lambda\left[\exp\{\hat \btheta^\top\Q(X_{ij})\} -1\right] + \hat \bnu^\top \g(X_{ij};\hat\bpsi,\hat\btheta)}.
$$
The MELEs of $F_0$ and $F_1$ are then defined as 
$$
\hat F_0(x)=\Sumij\hat p_{ij} I(X_{ij}\leq x) \;\;\;\;
{\rm and} \;\;\;\;
\hat F_1(x)=\Sumij\hat p_{ij}\exp\{\hat \btheta^\top\Q(X_{ij})\} I(X_{ij}\leq x).
$$

We now present results on the asymptotic properties of the MELEs $\hat F_0(x)$ and $\hat F_1(x)$ of the two population CDFs $F_0(x)$ and $F_1(x)$.  Let 
%%\begin{equation*}
%%   \W = \V^{-1} \U^\top \J^{-1}\U\V^{-1} - \ba{cc} {\bf 0}&{\bf 0}\\ {\bf 0} & \A_{\bu\bu}^{-1}\ea,
%%\end{equation*}
$$
 \W = \V^{-1} \U^\top \J^{-1}\U\V^{-1} - \ba{cc} {\bf 0}&{\bf 0}\\ {\bf 0} & \A_{\bu\bu}^{-1}\ea, ~~
 \B_0^*(x) = \ba{c}\B_{0\btheta}(x)\\\B_{0\bu}(x) \ea, ~~
 \B_1^*(x) = \ba{c}  \B_{1\btheta}(x)\\\B_{1\bu}(x) \ea ,
$$
where 
\begin{eqnarray*}
 &&\B_{0\btheta}(x) =E_0\lb h_1(X)\Q(X) I(X \leq x) \rb , \hspace{1.6cm}
 \B_{0\bu}(x) = E_0\lb \frac{\G(X)}{h(X)}I(X\leq x) \rb, \\
 &&\B_{1\btheta}(x) = \frac{\lambda^*-1}{\lambda^*}E_0\lb h_1(X)\Q(X) I(X \leq x) \rb,~~~
 \B_{1\bu}(x) = E_0\lb \frac{\omega(X)\G(X)}{h(X)}I(X\leq x) \rb.
\end{eqnarray*}
Furthermore, let $\tilde F_0(x)$ and $\tilde F_1(x)$ be the MELEs of $F_0$ and $F_1$ under the DRM when there is no auxiliary information. 
We refer to \cite{qin1997goodness} for the forms of $\tilde F_0(x)$ and $\tilde F_1(x)$ and their asymptotic properties. 
Denote $x \wedge y=\min(x,y)$.

\begin{theorem}
\label{theoremMELE.cdf}
Assume that the conditions of Theorem \ref{theorem_eta} are satisfied. 
\begin{enumerate}
\item[(a)]
 For any $l,s \in \{0,1\}$ and real numbers $x$ and $y$ in the support of $F_0$,
 as $n\to\infty$, 
$$
\sqrt{n}
\left(
\begin{array}{c}
\hat{F}_l(x) - F_l(x)\\
\hat{F}_s(y) - F_s(y)\\
\end{array}
\right)
\to 
N\Big({\bf 0}, \bSigma_{ls}(x,y) \Big), 
$$ 
where 
\begin{equation*}
   \bSigma_{ls}(x,y)  = \ba{cc} \sigma_{ll}(x,x) &\sigma_{ls}(x,y) \\ \sigma_{sl}(y,x) &\sigma_{ss}(y,y) \ea
\end{equation*}
with 
\begin{equation*}
%\label{sigmaij}
         \sigma_{ij}(x,y) = E_0\left\{ \frac{  \omega^{i+j}(X) I(X\leq x \wedge y)}{h(X)} \right\} - F_i(x)F_j(y)+ \B^*_i(x)^\top \W \B^*_j(y)
     \end{equation*}
for any $i,j \in \{l,s\}$. 
\item[(b)]  If $r = p$, the asymptotic variance-covariance matrix $\bSigma_{ls}(x,y)$ reduces to the same one of $\sqrt{n}\big(\tilde{F}_l(x) - F_l(x), \tilde{F}_s(x) - F_s(x) \big)^\top$. 
\item[(c)] If $r>p$, the asymptotic variance  matrix $\bSigma_{ls}(x,y)$  cannot decrease if an EE in \eqref{ee} is dropped. 
\end{enumerate} 
\end{theorem}

Theorem \ref{theoremMELE.cdf} indicates that the MELEs $\hat F_0(x)$ and $\hat F_1(x) $
have asymptotic properties similar to those of $\hat\etab$. 
That is, they are asymptotically normally distributed; they are asymptotically equivalent to $\tilde F_0(x)$ and $\tilde F_1(x) $ when $r=p$; and they become more efficient when $r>p$.

In the second half of this section we discuss the estimation of the quantiles of $F_i(x)$ for $i=0$ and 1. 
For any $\tau \in (0,1)$, we define the $\tau$th-quantile of $F_i$ as 
$\xi_{i,\tau} = \inf \{x:F_i(x) \geq \tau\}$ and its MELE as 
\begin{equation}
\label{mele.hatXi}
    \hat{\xi}_{i,\tau} = \inf \{x:\hat{F}_i(x) \geq \tau\}.
\end{equation}
Similarly, the estimator of $\xi_{i,\tau}$ based on $\tilde{F}_i(x)$ is defined as 
\begin{equation}
\label{mele.tildeXi}
    \tilde{\xi}_{i,\tau} = \inf \{x:\tilde{F}_i(x) \geq \tau\}.
\end{equation}
See \cite{zhang2000quantile} and \cite{chen2013quantile} for 
the asymptotic properties of  $    \tilde{\xi}_{i,\tau}$.
We refer to $\hat{\xi}_{i,\tau}$ as the ``DRM-EE" quantile estimators and $\tilde{\xi}_{i,\tau}$ as the ``DRM" quantile estimators.

The Bahadur representation is a useful tool for studying the asymptotic properties of quantile estimators. 
In the following theorem, we show that the DRM-EE quantile estimators are Bahadur representable. 
Let $f_i(x)$ be the probability density function of $F_i(x)$ for $i=0$ and 1. 

\begin{theorem}
\label{theoremBahadur}
Assume that the conditions of Theorem \ref{theorem_eta} are satisfied. 
Further, for $i= 0,1$ and any $\tau \in (0,1)$,  assume that $f_i(x)$ is continuous and positive at 
$x = \xi_{i,\tau}$. Then   $\hat{\xi}_{i,\tau}$ admits the Barhadur representation
\begin{equation*}
%\label{bahadur}
  \hat{\xi}_{i,\tau} = \xi_{i,\tau}+ \frac{\tau - \hat{F}_i(\xi_{i,\tau})}{f_i(\xi_{i,\tau})} + O_p\big(n^{-3/4}(\log n)^{1/2}\big).
\end{equation*}
\end{theorem}

%%With Theorems \ref{theoremMELE.cdf} and \ref{theoremBahadur}, 

The following theorem shows that the DRM-EE quantile estimators have asymptotic properties similar to those of the MELEs of $\etab$, $F_0(x)$, and $F_1(x)$. 
\begin{theorem}
\label{theoremQuantile}
Assume that the conditions in Theorem \ref{theoremBahadur} hold for $\xi_{l,\tau_l}$ and $\xi_{s,\tau_s}$. 
\begin{enumerate}
\item[(a)] 
As $n\to\infty$, 
\begin{equation*}
   \sqrt{n} \ba{c}\hat{\xi}_{l,\tau_l}-\xi_{l,\tau_l}\\\hat{\xi}_{s,\tau_s}-\xi_{s,\tau_s}\ea 
   \to N({\bf 0},  \bOmega_{ls}),
\end{equation*}
where
\begin{equation*}
    \bOmega_{ls}=  
    \ba{cc} \sigma_{ll}(\xi_{l,\tau_l},\xi_{s,\tau_s})/f_l^2(\xi_{l,\tau_l}) &\sigma_{ls}(\xi_{l,\tau_l},\xi_{s,\tau_s})/f_l(\xi_{l,\tau_l})f_s(\xi_{s,\tau_s})  \\ \sigma_{sl}(\xi_{s,\tau_s},x)/f_s(\xi_{s,\tau_s}) f_l(\xi_{l,\tau_l}) &\sigma_{ss}(\xi_{s,\tau_s},\xi_{s,\tau_s})/f_s^2(\xi_{s,\tau_s}) \ea.
\end{equation*}
\item[(b)]   If $r = p$, the asymptotic variance  matrix $\bOmega_{ls}$ of the DRM-EE quantile estimators is the same as that for the DRM quantile estimators;
    \item [(c)] if $r > p$, the asymptotic variance  matrix $\Omega_{ls}$ of the DRM-EE quantile estimators cannot decrease if an EE in \eqref{ee} is dropped. 
    \end{enumerate}
\end{theorem}

Using the results of Theorems \ref{theoremMELE.cdf} and  \ref{theoremQuantile}, we may construct confidence regions and/or  test hypotheses on the CDFs at some fixed points and for quantiles through the Wald-type statistics. However, methods based on the Wald-type statistics require a consistent estimator of the corresponding asymptotic variance. 
%It is noteworthy that the expression of elements in $\Omega_{ls}$ contains density functions, which are always unknown. 
%Usually kernel approach is used to obtain the smoothed estimator of the density function \citep{chen2013quantile}.
%In addition, inference can also be made on some smooth functions of quantiles, such as the difference between quantiles.
%
%\cite{zhang2000quantile} and \cite{chen2013quantile} proposed quantile estimators under DRMs, denoted as the DRM quantiles hereafter. They demonstrated that DRMs help exploit information from related samples so that DRM quantiles  achieve better efficiency than the empirical quantile obtained through a single sample. We write the empirical quantile as EMP quantile for simplicity. 
%With the additional information embedded in the EEs, we expect that our DRM-EE quantile would be even more efficient than the DRM quantile.
%This intuition leads to the following theorem. 
It is more attractive to use the results in Corollary \ref{corol_valid} to construct the ELR-based confidence region 
for the CDFs at some fixed points and for quantiles.

Suppose we are interested in constructing a $(1-a)$-level CI  for a CDF at some fixed point $x_0$ for $i = 0$ or $1$. Denote the parameter of interest as $\zeta=F_i(x_0)$.
%where $x_0$ is a fixed point. 
Let 
$$
g^*_1(x;\btheta,\zeta)= 
\left\{
\begin{array}{cl}
I(x\leq x_0)- \zeta,&i=0\\
 \exp\{\btheta^\top\Q(x)\} I(x\leq x_0)-\zeta,&i=1\\
\end{array}
\right..
$$
We further define $\ell_{n1}^*(\bpsi,\bbeta,\zeta)$ to be the profile empirical log-likelihood of $(\bpsi,\bbeta,\zeta)$
under Model \eqref{drm_def}
with the unbiased EEs in \eqref{ee} and $E_0\{g^*_1(X;\btheta,\zeta) \}$=0. 
Then the ELR function of $\zeta$ is defined as 
$$
R_{n1}(\zeta)=2\{\ell_n(\hat\bpsi,\hat\btheta)-\sup_{ \bpsi, \btheta}\ell_{n1}^*(\bpsi,\btheta,\zeta)\}.
$$

We can similarly define the ELR function for a quantile $\xi$  at the quantile level $\tau$ for $i=0$ or $1$, i.e., $\xi=\xi_{i,\tau}$. 
Let 
$$
g^*_2(x;\btheta,\xi)= 
\left\{
\begin{array}{cl}
I(x\leq \xi)- \tau,&i=0\\
 \exp\{\btheta^\top\Q(x)\} I(x\leq \xi)-\tau,&i=1\\
\end{array}
\right..
$$
We further  define $\ell_{n2}^*(\bpsi,\bbeta,\xi)$ to be the profile empirical log-likelihood of $(\bpsi,\bbeta,\xi)$
under Model \eqref{drm_def}
with the unbiased EEs in \eqref{ee} and $E_0\{g^*_2(X;\btheta,\xi)\}=0$. 
Then the ELR function of $\xi$ is defined as 
$$
R_{n2}(\xi)=2\{\ell_n(\hat\bpsi,\hat\btheta)-\sup_{ \bpsi, \btheta}\ell_{n2}^*(\bpsi,\btheta,\xi)\}.
$$

Using Corollary \ref{corol_valid}, we have the following results for $R_{n1}(\zeta^*)$ and 
$R_{n2}(\xi^*)$, where $\zeta^*$ and $\xi^*$ are the true values of $\zeta$ and $\xi$. 

\begin{corollary}
\label{corol_cdf.quantile}
 Under the conditions of Theorem \ref{theorem_eta}, as $n\to\infty$,
 both $R_{n1}(\zeta^*)$  and $R_{n2}(\xi^*)$  converge in distribution to $\chi^2_1$. 
\end{corollary}

Corollary \ref{corol_cdf.quantile} enables us to construct the ELR-based CI for $\zeta$ and $\xi$. 
For example, the ELR-based CI for $\xi$ with level $1-a$ can be constructed  as 
$
\{\xi: R_{n2}(\xi)\leq \chi^2_{1,1-a} \}. 
$
%Its performance will be  evaluated in next section. 

\section{Simulation Studies}
\label{gee_sim}
We conducted simulation studies to investigate three aspects of the proposed semiparametric inference procedures: 
\begin{enumerate}
\item[(1)] The performance of the inference procedures for $\bpsi$; 
\item[(2)] The power of the ELR test for the validity and usefulness of the auxiliary information;
\item[(3)] The performance of the inference procedures for the population quantiles. 
\end{enumerate}
We consider four combinations of sample sizes $(n_0,n_1)$:  $(50, 50)$, $(50, 150)$, $(100, 100)$, and $(200, 200)$.
For each simulation setting, the number of simulation runs is 2,000. 

\subsection{Simulation studies for inferences on $\bpsi$}
\label{gee_sim1}
%We conducted simulations to study: (1) the performance of the proposed point estimators (2) the performance of the ELR confidence intervals of parameters, and (3) the power of the ELR test for the validity and usefulness of the auxiliary information. The sensitivity of the point estimators and the ELR confidence intervals using biased auxiliary information is also discussed. The number of simulation runs is set to 10,000. 

\subsubsection{Simulation setup}
We start by exploring the first  aspect of the proposed semiparametric inference procedures. 
In the simulations, $F_0$ and $F_1$ are the CDFs of  ${\rm LN}(0,1)$ and ${\rm LN}(0.5,1)$, respectively, where ${\rm LN}(a, b)$ denotes the lognormal distribution with mean $a$ and variance $b$, both with respect to the log scale. 
It is easy to show that $F_0$ and $F_1$ satisfy the DRM in \eqref{drm_def} with 
$\Q(x)=(1,\log x)^\top$.  
The  parameter of interest is the mean ratio $\delta = \mu_1/\mu_0$ which was discussed in Example \ref{example1}.
%Our method deals with inference procedure on $\theta$ by taking 
%$$g(x;\theta,\btheta) = x\theta - x \exp\{\btheta^\top\Q(x)\}.$$

To examine the usefulness of auxiliary information, we construct another variable $Z$ using the following model:
\begin{equation}
\label{x and z}
   Z = 1 + 0.5X + \epsilon
~~{\rm and}~~\epsilon \sim N(0,1). 
\end{equation}
That is, given $X_{ij}$,  $Z_{ij}$ is generated from \eqref{x and z}, for $i = 0,1,j = 1,\cdots,n_i$.
Hence,  the two-sample data consist of $\T_{ij}=(X_{ij},Z_{ij})^\top$ for $i = 0,1,j = 1,\cdots,n_i$. 
We treat $\mu_{z0} = E(Z|D=0)$, the population mean of covariate $Z$ for the first group (i.e., the $D=0$ group), as the known auxiliary information. 
Let the CDFs of $\T$ given $D=0$ and $D=1$ be $\boldsymbol{F}_0$ and $\boldsymbol{F}_1$, respectively. It can be checked that  $\boldsymbol{F}_0$ and $\boldsymbol{F}_1$ satisfy the DRM with $\Q(x,z) = (1,\log x)^\top$. 

To explore the effect of misspecified estimating equations for the auxiliary information, we introduce a bias by using $\kappa \mu_{z0}$ instead of the true value $\mu_{z0}$ for $E(Z|D=0)$. We consider $\kappa = 0.90,0.95,1.00,1.05,1.10$. Note that $\kappa = 1.00$ corresponds to correctly specified auxiliary information. 
We incorporate the biased/unbiased auxiliary information into our problem by setting $\bpsi=\delta$ and
%\begin{equation*}
$\g(\t;\bpsi,\btheta) = \left(\delta x- x \exp\{\btheta^\top\Q(x)\},  z - \kappa \mu_{z0}\right)^\top$
%\end{equation*}
in \eqref{ee}. 

%Two existing competitive method are adopted here. The first is the empirical method,  The other one is the two-sample EL method \citep{owen2001empirical}, which ignores the relationship between two populations and assumes them independent.

\subsubsection{Performance of point estimators\label{gee_sim1.point}}
We compare three point estimators:
\begin{itemize}
    \item[(i)] EMP: $\bar\delta=\bar{\mu}_1/\bar\mu_0$, where $\bar{\mu}_i=n_i^{-1}\sum_{j =1}^{n_i}x_{ij}$ for $i=0$ and 1;
    \item[(ii)] DRM: $\tilde\delta=\tilde{\mu}_1/\tilde\mu_0$, where $\tilde\mu_i=\int x d \tilde F_i(x)$ for $i=0$ and 1;
    \item[(iii)] DRM-EE: $\hat\delta=\hat{\mu}_1/\hat\mu_0$, where $\hat\mu_i=\int x d \hat F_i(x)$ for $i=0$ and 1.
\end{itemize}
Note that the asymptotic properties of $\tilde \delta$ and $\hat\delta$ are covered in Theorem \ref{theorem_eta}. 
The performance of each estimator is evaluated by the relative bias (RB) and the mean squared error (MSE). Simulation results on the three point estimators are presented in Table \ref{tab_point}.

\begin{table}[htbp]
  \centering
  \footnotesize
  \caption{RB (\%) and MSE ($\times 100$) of  three point estimators of the mean ratio}
  \label{tab_point}
    \begin{tabular}{ccccccccc}
    \toprule
          \multicolumn{1}{c}{\multirow{2}[2]{*}{$(n_0,n_1)$}}&       & \multicolumn{1}{c}{\multirow{2}[2]{*}{EMP}} & \multicolumn{1}{c}{\multirow{2}[2]{*}{DRM}} & \multicolumn{5}{c}{DRM-EE} \\
          &       &       &       & \multicolumn{1}{c}{$\kappa =1$} & \multicolumn{1}{c}{$\kappa=0.9$} & \multicolumn{1}{c}{$\kappa=0.95$} & \multicolumn{1}{c}{$\kappa=1.05$} & \multicolumn{1}{c}{$\kappa=1.1$} \\
    \midrule
    \multirow{2}[1]{*}{$(50,50)$} & RB    & 3.37  & 1.46  & 1.15  & 12.73 & 6.83  & -4.24 & -9.32 \\
          & MSE   & 20.03 & 12.50 & 9.61  & 16.59 & 12.07 & 9.00  & 9.96 \\ [1mm]
    \multirow{2}[0]{*}{$(50,150)$} & RB    & 3.70  & 1.75  & 0.89  & 16.61 & 8.50  & -6.11 & -12.41 \\
          & MSE   & 12.91 & 8.07  & 4.67  & 13.94 & 7.46  & 4.92  & 7.50 \\[1mm]
    \multirow{2}[0]{*}{$(100,100)$} & RB    & 1.86  & 1.21  & 0.89  & 12.32 & 6.48  & -4.35 & -9.20 \\
          & MSE   & 9.35  & 6.17  & 5.08  & 10.46 & 6.78  & 5.11  & 6.56 \\[1mm]
    \multirow{2}[1]{*}{$(200,200)$} & RB    & 0.90  & 0.46  & 0.53  & 11.87 & 6.06  & -4.62 & -9.27 \\
          & MSE   & 4.88  & 3.15  & 2.56  & 7.03  & 3.84  & 2.92  & 4.60 \\
    \bottomrule
    \end{tabular}%
\end{table}%

%Major observations from Table \ref{tab_point} can be summarized as follows. 
We first compare the results reported in the third to fifth columns, i.e., EMP, DRM, and DRM-EE with correctly specified auxiliary information (DRM-EE with $\kappa=1$). 
We see that the EMP estimator has the largest RBs and MSEs in all cases. 
The estimator of DRM-EE with $\kappa=1$ has the best performance, 
followed by the DRM estimator.  
This suggests that using correctly specified auxiliary information improves the estimation efficiency, 
which  agrees with Corollary \ref{corol_1} in Section \ref{gee_result}. 
We also note that as the sample size increases, all three estimators have improved performance and the gaps between the three estimators become less pronounced, especially between DRM and DRM-EE.
%%Their consistency explains this phenomenon. 
%This phenomenon can be explained by the consistency of these estimators. 

The sensitivity of the DRM-EE estimator with respect to misspecified auxiliary information 
can be observed from the last four columns of Table~\ref{tab_point}. 
The DRM-EE estimator for $\kappa\neq 1$ are clearly not as good as the estimator for $\kappa=1$.   
The absolute value of the RB increases as $\kappa$ moves further away from 1. 

\subsubsection{Performance of confidence intervals\label{gee_sim1.ci}}

%The asymptotic normality of the naive estimator $\hat{\theta}_{na}$ is used to construct a Wald-type confidence interval for $\theta$. Note that, as $n \to \infty$, we have 
%\begin{equation*}
%    \sqrt{n_k}\{\log(\Bar{x}_k) - \log(\mu_k)\} \to {\rm N}
 %   \ls 0, \frac{\sigma^2_k}{\mu_k^2} \rs~~~ k = 0,1,
%\end{equation*}
%in distribution, where $\mu_k$ and $\sigma^2_k$ are the population mean and variance of $X$ for $k$th group. Since the sample mean $\hat{\mu}_k$ and the sample variance $\hat{\sigma}^2_k$ are consistent estimators of $\mu_k$ and $\sigma^2_k$. This suggests a confidence interval for $\theta$ at nominal level $1-\tau$ based on naive method:
We compare four   CIs for $\delta$:
\begin{itemize}
    \item[(i)] EMP-NA: Wald-type CI for $\delta$ based on the asymptotic normality of  $\log\bar\delta$;
    \item[(ii)] EMP-EL: \cite{owen2001empirical}'s ELR-based CI for $\delta$; 
    
    \item[(iii)] DRM: the ELR-based CI for $\delta$ in \eqref{lr_CI} without auxiliary information;
    \item[(iv)] DRM-EE: the ELR-based CI for $\delta$  in \eqref{lr_CI} with auxiliary information.
\end{itemize}
%\citep{owen2001empirical,zhang2020empirical}
%For two-sample EL method, an ELR confidence interval is built and named as $\mathcal{I}_{two}$ \citep{owen2001empirical}. Because of the chi-squared limiting distribution of the ELR statistics presented in Theorem \ref{theorem_lr_eta}, ELR confidence intervals for $\theta$ can be constructed with corresponded $R(\etab)$'s. We denote $\mathcal{I}_{pro}$ as the ELR confidence intervals when the auxiliary information is not available, while the ELR confidence intervals is represented by $\mathcal{I}_{pro}$ when utilizing the auxiliary information $\mu_{z0}$.
The performance of a CI is  evaluated in terms of coverage probability (CP) and average length (AL). The simulation results for the four CIs at the  95\% nominal level  are shown in Table \ref{tab_CI}.

\begin{table}[!htbp]
  \centering
  \footnotesize
  \caption{CP (\%) and AL of four CIs for the mean ratio at 95\%  nominal level}
    \label{tab_CI}%
    \begin{tabular}{cccccccccc}
    \toprule
         \multicolumn{1}{c}{\multirow{2}[2]{*}{$(n_0,n_1)$}} &       & \multicolumn{1}{c}{\multirow{2}[2]{*}{EMP-NA}} & \multicolumn{1}{c}{\multirow{2}[2]{*}{EMP-EL}} & \multicolumn{1}{c}{\multirow{2}[2]{*}{DRM}} & \multicolumn{5}{c}{DRM-EE} \\
          &       &       &       &       & \multicolumn{1}{c}{$\kappa =1$} & \multicolumn{1}{c}{$\kappa=0.9$} & \multicolumn{1}{c}{$\kappa=0.95$} & \multicolumn{1}{c}{$\kappa=1.05$} & \multicolumn{1}{c}{$\kappa=1.1$} \\
    \midrule
    \multirow{2}[1]{*}{$(50,50)$} & CP    & 92.6  & 91.6  & 94.5  & 94.2  & 90.7  & 93.9  & 92.1  & 88.1 \\
          & AL    & 1.65  & 1.65  & 1.41  & 1.23  & 1.38  & 1.30  & 1.16  & 1.10 \\[1mm]
    \multirow{2}[0]{*}{$(50,150)$} & CP    & 92.9  & 92.2  & 95.6  & 94.3  & 78.1  & 91.4  & 88.5  & 75.9 \\
          & AL    & 1.33  & 1.31  & 1.15  & 0.84  & 1.00  & 0.92  & 0.77  & 0.71 \\[1mm]
    \multirow{2}[0]{*}{$(100,100)$} & CP    & 94.9  & 93.9  & 95.3  & 94.3  & 85.6  & 92.5  & 92.0  & 85.1 \\
          & AL    & 1.18  & 1.20  & 1.00  & 0.88  & 0.98  & 0.93  & 0.84  & 0.80 \\[1mm]
    \multirow{2}[1]{*}{$(200,200)$} & CP    & 93.8  & 93.3  & 94.6  & 94.7  & 75.3  & 89.0  & 90.4  & 78.4 \\
          & AL    & 0.84  & 0.86  & 0.70  & 0.62  & 0.69  & 0.66  & 0.60  & 0.58 \\
    \bottomrule
    \end{tabular}
\end{table}%

As we can see in the third to sixth columns, EMP-NA and EMP-EL are comparable but are clearly inferior to DRM and DRM-EE ($\kappa=1$) in terms of CP and AL.  The CPs of  the CIs for DRM and DRM-EE with $\kappa=1$ are close to the nominal level for all sample size combinations.
This suggests that the limiting distributions provide accurate approximations to the finite-sample distributions of the ELR statistics. 
The ALs of the CIs for DRM-EE with $\kappa=1$ are always shorter than other CIs, a strong evidence that using correctly specified auxiliary information improves the performance of a CI. 
%four estimators are all satisfactory in term of both CPs and ALs. In large sample sizes, the CPs of these confidence intervals are closer to the 95\% confidence level and the ALs are shorter. Among these four estimators, $\mathcal{I}_{aux}$ always gives the shortest AL, which may causes a worse performance on the CPs in some cases.
On the other hand, misspecified auxiliary information results in inaccurate CIs. 
As $\kappa$ moves further away from $1$, the CP of the ELR-based CI shifts away from the nominal value. 
%It is interesting to note that the CPs of CIs seem to be more sensitive in large sample sizes.  

\subsubsection{Power of the validity test}
In this section, we explore the second aspect of the proposed semiparametric inference procedures on the power of the ELR test for  the validity of the auxiliary information. The null hypothesis for the ELR test is 
$H_0: E_0(z - \kappa \mu_{z0})=0.$
According to Theorem \ref{theorem_valid} and Corollary \ref{corol_valid}, 
the ELR statistic has a $\chi^2_1$ limiting distribution  under the null hypothesis.
We consider misspecified auxiliary information with  $\kappa = 0.90,0.95,1.05,1.10$ as the alternatives. 
Table~\ref{tab_power} gives the simulated power ($\kappa \neq 1$) and type I error rate ($\kappa = 1$) of the ELR test at the 5\% significance level.

%\begin{figure}[ht!]
%\centering
%\includegraphics[scale = 0.6]{power.eps} 
%\caption{Power curve of the ELR statistic under different sample sizes}       
%\label{fig_power}
%\end{figure}

 \begin{table}[!htbp]
  \centering
  \footnotesize
  \caption{Power and  type I error rate of the ELR test  (\%)  at 5\% significance level}
  \label{tab_power}
    \begin{tabular}{cccccc}
    \toprule
    \multicolumn{1}{c}{$(n_0,n_1)$} &  \multicolumn{1}{c}{$\kappa=0.9$} & \multicolumn{1}{c}{$\kappa=0.95$} &
    \multicolumn{1}{c}{$\kappa =1$} &
    \multicolumn{1}{c}{$\kappa=1.05$} & \multicolumn{1}{c}{$\kappa=1.1$} \\
    \midrule
    $(50,50)$ & 21.43 & 8.76  & 5.36  & 9.41  & 20.48 \\
    $(50,150)$ & 27.33 & 10.08 & 5.37  & 10.13 & 22.97 \\
    $(100,100)$ & 36.44 & 11.26 & 5.51  & 13.61 & 32.48 \\
    $(200,200)$ & 62.98 & 20.66 & 5.15  & 19.16 & 55.23 \\
    \bottomrule
    \end{tabular}%
  
\end{table}%

We observe from Table \ref{tab_power} that the type I error rates of the ELR tests are close to the 5\% nominal level in all cases, which suggests that the limiting distribution for the ELR test works very well. 
As $\kappa$ deviates from 1 and  the sample size increases, the power of the test increases, as expected. 
%The power curves of the ELR statistics share the same trend. The power of the tests grows with increased sample sizes. The power of the tests grows with increased sample sizes. 

\subsection{Simulation studies for inference on quantiles}
\label{gee_sim2}
\subsubsection{Simulation setup}

%We conducted a small simulation study to investigate the finite sample performance of the estimators for distribution functions and quantiles as well as the performance of CIs for them. 
The third aspect of the proposed semiparametric inference procedures is inference on population quantiles with auxiliary information. 
In the simulations, we consider two distributional settings: 
\begin{itemize}
    \item [(1)] $f_0 \sim N(18,4)$ and $f_1 \sim N(18,9)$;
    \item [(2)] $f_0 \sim Gam (6,1.5)$ and $f_1 \sim Gam (8,1.125)$.
\end{itemize}
Here $N(a,b)$ denotes the normal distribution with mean $a$ and variance $b$ and $Gam(a,b)$ is the gamma distribution with shape parameter $a$ and scale parameter $b$. 
We are interested in estimating and constructing CIs for the  quantiles of $F_0$ and $F_1$ at the levels $\tau = 0.10,0.25,0.5,0.75,0.90$. 
%Two combinations of sample sizes $(n_0,n_1)$ are considered: (100,100) and (100,300). For each simulation scenario, results are based on 2,000 repeated simulation runs.  

\subsubsection{Performance of quantile estimators}
We compare four quantile estimators:
\begin{itemize}
    \item[(i)]  EMP: the  quantile estimator based on the empirical CDFs;
    \item[(ii)]  EL:    the quantile estimator based on the MELEs of the CDFs in \cite{tsao2006empirical}, in which a common mean is assumed;
    \item[(iii)]  DRM: the DRM  based quantile estimator in \eqref{mele.tildeXi};
    \item[(iv)]  DRM-EE: our proposed quantile estimator in \eqref{mele.hatXi} with the common-mean assumption or the EE \eqref{example3.ee} in Example \ref{example3}.
\end{itemize}

The DRM and DRM-EE methods are calculated with the correctly specified $\q(x)$, where $\q(x) = (x,x^2)^\top$ for the normal distributional setting and $\q(x) = (x,\log x)^\top$ for the gamma distributional setting.
The performance of an estimator is evaluated by the RB and MSE. 
The general patterns of the simulation results for the four methods are similar in the two settings. 
Hence, Table~\ref{tab_quant_point_norm} presented here is only for the normal setting; the results under gamma distributions are included in Section~3 of the Supplementary Material.

% Table generated by Excel2LaTeX from sheet 'point'
\begin{table}[!htbp]
  \centering
  \footnotesize
  \caption{RB (\%) and MSE ($\times 100$) for quantile estimators (normal distributions)}
  \label{tab_quant_point_norm}
    \begin{tabular}{ccccccccccc}
    \toprule
    \multirow{2}[2]{*}{$(n_0,n_1)$} & \multirow{2}[2]{*}{$\tau$} &       & \multicolumn{4}{c}{$N(18,4)$}   & \multicolumn{4}{c}{$N(18,9)$} \\
          &       &       & EMP    & EL    & DRM   & DRM-EE & EMP    & EL    & DRM   & DRM-EE \\
    \midrule
    \multirow{10}[10]{*}{$(50,50)$} & \multirow{2}[2]{*}{0.10} & RB    & -0.58 & 0.08  & 0.25  & 0.19  & -1.07 & -0.10 & 0.17  & -0.07 \\
          &       & MSE   & 23.87 & 19.88 & 18.85 & 16.32 & 59.74 & 44.17 & 46.26 & 37.35 \\[1mm]
           & \multirow{2}[2]{*}{0.25} & RB    & 0.04  & 0.02  & 0.15  & 0.14  & 0.01  & -0.06 & -0.14 & -0.25 \\
          &       & MSE   & 14.73 & 12.25 & 12.23 & 9.57  & 33.32 & 22.42 & 29.22 & 18.11 \\[1mm]
           & \multirow{2}[2]{*}{0.50} & RB    & -0.21 & 0.03  & 0.04  & 0.03  & -0.43 & 0.03  & 0.00  & 0.03 \\
          &       & MSE   & 12.47 & 9.93  & 10.06 & 7.76  & 29.21 & 16.25 & 25.08 & 11.10 \\[1mm]
           & \multirow{2}[2]{*}{0.75} & RB    & -0.01 & -0.01 & -0.08 & -0.07 & -0.05 & 0.02  & 0.03  & 0.14 \\
          &       & MSE   & 13.92 & 11.81 & 11.97 & 9.64  & 34.86 & 21.55 & 29.68 & 16.95 \\[1mm]
           & \multirow{2}[2]{*}{0.90} & RB    & -0.62 & -0.08 & -0.21 & -0.18 & -0.87 & 0.08  & -0.08 & 0.10 \\
          &       & MSE   & 23.36 & 21.36 & 19.51 & 17.66 & 53.89 & 43.03 & 46.50 & 37.61 \\
    \midrule
    \multirow{10}[10]{*}{$(50,150)$} & \multirow{2}[2]{*}{0.10} & RB    & -0.60 & 0.01  & 0.26  & 0.17  & -0.28 & 0.09  & 0.17  & 0.13 \\
          &       & MSE   & 23.91 & 18.16 & 16.36 & 11.49 & 17.62 & 14.72 & 16.05 & 13.34 \\[1mm]
           & \multirow{2}[2]{*}{0.25} & RB    & 0.04  & 0.02  & 0.14  & 0.12  & 0.06  & 0.03  & -0.01 & -0.03 \\
          &       & MSE   & 14.81 & 10.08 & 11.22 & 6.64  & 11.00 & 8.67  & 10.20 & 7.88 \\[1mm]
           & \multirow{2}[2]{*}{0.50} & RB    & -0.21 & 0.07  & 0.04  & 0.04  & -0.10 & 0.05  & 0.05  & 0.06 \\
          &       & MSE   & 12.39 & 7.69  & 9.09  & 4.59  & 8.97  & 6.92  & 8.15  & 5.84 \\[1mm]
           & \multirow{2}[2]{*}{0.75} & RB    & -0.01 & 0.02  & -0.10 & -0.05 & -0.06 & -0.04 & -0.01 & 0.00 \\
          &       & MSE   & 13.90 & 10.24 & 10.87 & 6.49  & 10.49 & 8.18  & 9.89  & 7.71 \\[1mm]
           & \multirow{2}[2]{*}{0.90} & RB    & -0.61 & -0.03 & -0.20 & -0.12 & -0.30 & -0.02 & -0.04 & -0.02 \\
          &       & MSE   & 23.32 & 19.87 & 17.26 & 12.94 & 17.04 & 14.93 & 16.25 & 14.40 \\
    \midrule
    \multirow{10}[10]{*}{$(100,100)$} & \multirow{2}[2]{*}{0.10} & RB    & -0.35 & 0.03  & 0.10  & 0.09  & -0.34 & 0.15  & 0.23  & 0.11 \\
          &       & MSE   & 11.82 & 10.05 & 9.13  & 7.86  & 25.71 & 19.44 & 22.01 & 16.62 \\[1mm]
           & \multirow{2}[2]{*}{0.25} & RB    & -0.17 & 0.03  & 0.04  & 0.04  & -0.18 & 0.02  & 0.03  & -0.06 \\
          &       & MSE   & 7.42  & 6.20  & 6.33  & 5.04  & 15.56 & 9.84  & 13.54 & 8.01 \\[1mm]
           & \multirow{2}[2]{*}{0.50} & RB    & -0.11 & 0.03  & 0.01  & 0.03  & -0.15 & 0.03  & 0.07  & 0.05 \\
          &       & MSE   & 6.07  & 4.81  & 5.21  & 3.88  & 13.53 & 7.87  & 11.53 & 5.41 \\[1mm]
           & \multirow{2}[2]{*}{0.75} & RB    & -0.17 & 0.01  & -0.05 & -0.02 & -0.30 & -0.05 & 0.01  & 0.02 \\
          &       & MSE   & 7.37  & 6.20  & 6.10  & 5.02  & 15.95 & 9.94  & 13.60 & 7.94 \\[1mm]
           & \multirow{2}[2]{*}{0.90} & RB    & -0.35 & -0.02 & -0.11 & -0.08 & -0.45 & -0.05 & -0.05 & -0.02 \\
          &       & MSE   & 11.82 & 10.83 & 9.40  & 8.24  & 25.37 & 19.69 & 22.77 & 17.23 \\
    \midrule
    \multirow{10}[10]{*}{$(200,200)$} & \multirow{2}[2]{*}{0.10} & RB    & -0.12 & 0.04  & 0.13  & 0.10  & -0.29 & -0.05 & 0.01  & -0.02 \\
          &       & MSE   & 5.77  & 5.01  & 4.50  & 3.91  & 13.65 & 10.89 & 11.81 & 8.94 \\[1mm]
            & \multirow{2}[2]{*}{0.25} & RB    & -0.06 & 0.02  & 0.05  & 0.04  & -0.12 & 0.02  & -0.04 & -0.04 \\
          &       & MSE   & 3.58  & 3.00  & 3.03  & 2.41  & 8.37  & 5.04  & 7.30  & 4.18 \\[1mm]
            & \multirow{2}[2]{*}{0.50} & RB    & -0.04 & 0.03  & 0.02  & 0.01  & -0.15 & -0.03 & -0.02 & 0.00 \\
          &       & MSE   & 3.02  & 2.40  & 2.57  & 1.99  & 7.07  & 3.99  & 6.04  & 2.80 \\[1mm]
            & \multirow{2}[2]{*}{0.75} & RB    & -0.10 & -0.03 & -0.03 & -0.04 & -0.16 & 0.00  & 0.00  & 0.03 \\
          &       & MSE   & 3.60  & 3.04  & 3.06  & 2.49  & 8.39  & 5.05  & 7.26  & 4.03 \\[1mm]
            & \multirow{2}[2]{*}{0.90} & RB    & -0.18 & -0.02 & -0.05 & -0.04 & -0.18 & 0.06  & 0.01  & 0.06 \\
          &       & MSE   & 5.90  & 5.24  & 4.68  & 4.10  & 12.78 & 10.16 & 11.75 & 8.75 \\
    \bottomrule
    \end{tabular}%
\end{table}%

Table  \ref{tab_quant_point_norm} shows that the RBs are negligibly small for all methods under all scenarios. The EMP estimator has the largest MSEs. The DRM-EE quantile estimators have the smallest MSEs due to its use of additional information, and   
the results agree with Theorem \ref{theoremQuantile}. 
We also notice that the  EL and DRM quantile estimators are comparable.

\subsubsection{Performance of confidence intervals}
We compare three CIs:
\begin{itemize}
    \item[(i)]  EMP: \cite{owen2001empirical}'s ELR-based CI for quantiles;
    \item[(ii)]  DRM: the ELR-based CI under the DRM without the common-mean assumption \citep{zhang2020empirical};
    \item[(iii)]  DRM-EE:  the proposed  ELR-based CI.
\end{itemize}
The construction of CIs for the quantiles under the two-sample EL method with the common-mean assumption has not been discussed in the literature, and hence is not included in the simulation. 
The CP and AL are used to compare CIs. 
We present the simulation results for the normal case in Table~\ref{tab_quant_ci_norm}.
The results for the gamma distributions display similar patterns and are included in Section~3 of the Supplementary Material.

% Table generated by Excel2LaTeX from sheet 'CI'
\begin{table}[!htbp]
  \centering
  \footnotesize
  \caption{CP (\%) and AL for 95\% CIs of $100\tau\%$-quantiles (normal distributions)}
  \label{tab_quant_ci_norm}
    \begin{tabular}{ccccccccc}
    \toprule
    \multirow{2}[2]{*}{$(n_0,n_1)$} & \multirow{2}[2]{*}{$\tau$} &       & \multicolumn{3}{c}{$N(18,4)$} & \multicolumn{3}{c}{$N(18,9)$} \\
          &       &       & EMP   & DRM   & DRM-EE & EMP   & DRM   & DRM-EE \\
    \midrule
    \multirow{10}[2]{*}{(50,50)} & \multirow{2}[1]{*}{0.10} & CP    & 94.5  & 94.3  & 94.2  & 94.4  & 94.5  & 94.8 \\
          &       & AL    & 1.96  & 1.74  & 1.61  & 2.94  & 2.89  & 2.48 \\[1mm]
          & \multirow{2}[0]{*}{0.25} & CP    & 95.9  & 95.1  & 95.2  & 95.0  & 94.8  & 94.2 \\
          &       & AL    & 1.60  & 1.40  & 1.25  & 2.36  & 2.18  & 1.64 \\[1mm]
          & \multirow{2}[0]{*}{0.50} & CP    & 94.3  & 94.6  & 95.4  & 93.8  & 94.8  & 95.4 \\
          &       & AL    & 1.32  & 1.28  & 1.11  & 1.98  & 1.98  & 1.36 \\[1mm]
          & \multirow{2}[0]{*}{0.75} & CP    & 95.2  & 94.3  & 94.8  & 95.2  & 94.8  & 95.1 \\
          &       & AL    & 1.59  & 1.39  & 1.24  & 2.36  & 2.16  & 1.63 \\[1mm]
          & \multirow{2}[1]{*}{0.90} & CP    & 94.2  & 94.5  & 93.9  & 94.3  & 95.0  & 94.9 \\
          &       & AL    & 1.97  & 1.74  & 1.62  & 2.97  & 2.92  & 2.50 \\
    \midrule
    \multirow{10}[2]{*}{(50,150)} & \multirow{2}[1]{*}{0.10} & CP    & 94.5  & 94.3  & 95.0  & 93.7  & 94.7  & 94.7 \\
          &       & AL    & 1.96  & 1.62  & 1.38  & 1.63  & 1.63  & 1.49 \\[1mm]
          & \multirow{2}[0]{*}{0.25} & CP    & 95.9  & 95.1  & 95.2  & 95.8  & 95.4  & 95.3 \\
          &       & AL    & 1.60  & 1.33  & 1.02  & 1.34  & 1.28  & 1.11 \\[1mm]
          & \multirow{2}[0]{*}{0.50} & CP    & 94.3  & 95.1  & 95.5  & 94.4  & 96.0  & 96.0 \\
          &       & AL    & 1.32  & 1.20  & 0.86  & 1.16  & 1.16  & 0.97 \\[1mm]
          & \multirow{2}[0]{*}{0.75} & CP    & 95.2  & 94.5  & 94.8  & 95.3  & 95.8  & 96.1 \\
          &       & AL    & 1.59  & 1.31  & 1.00  & 1.32  & 1.27  & 1.10 \\[1mm]
          & \multirow{2}[1]{*}{0.90} & CP    & 94.2  & 94.8  & 94.2  & 95.2  & 94.6  & 94.3 \\
          &       & AL    & 1.97  & 1.62  & 1.39  & 1.65  & 1.63  & 1.50 \\
    \midrule
    \multirow{10}[2]{*}{(100,100)} & \multirow{2}[1]{*}{0.10} & CP    & 95.6  & 95.0  & 95.2  & 95.9  & 94.3  & 95.2 \\
          &       & AL    & 1.42  & 1.20  & 1.12  & 2.12  & 1.95  & 1.67 \\[1mm]
          & \multirow{2}[0]{*}{0.25} & CP    & 95.7  & 94.4  & 95.4  & 94.9  & 95.3  & 95.1 \\
          &       & AL    & 1.10  & 1.00  & 0.89  & 1.66  & 1.51  & 1.14 \\[1mm]
          & \multirow{2}[0]{*}{0.50} & CP    & 94.8  & 94.7  & 95.2  & 95.2  & 96.1  & 95.5 \\
          &       & AL    & 0.96  & 0.90  & 0.79  & 1.45  & 1.38  & 0.94 \\[1mm]
          & \multirow{2}[0]{*}{0.75} & CP    & 95.2  & 94.7  & 95.5  & 95.3  & 95.8  & 95.5 \\
          &       & AL    & 1.09  & 0.98  & 0.87  & 1.62  & 1.51  & 1.14 \\
          & \multirow{2}[1]{*}{0.90} & CP    & 95.5  & 94.2  & 94.3  & 95.6  & 95.2  & 94.8 \\
          &       & AL    & 1.43  & 1.21  & 1.13  & 2.15  & 1.96  & 1.66 \\[1mm]
    \midrule
    \multirow{10}[2]{*}{(200,200)} & \multirow{2}[1]{*}{0.10} & CP    & 93.8  & 95.4  & 95.1  & 94.5  & 94.4  & 94.9 \\
          &       & AL    & 0.93  & 0.84  & 0.79  & 1.39  & 1.36  & 1.16 \\[1mm]
          & \multirow{2}[0]{*}{0.25} & CP    & 95.8  & 95.7  & 95.3  & 95.0  & 95.0  & 94.0 \\
          &       & AL    & 0.77  & 0.69  & 0.62  & 1.14  & 1.06  & 0.80 \\[1mm]
          & \multirow{2}[0]{*}{0.50} & CP    & 94.9  & 95.0  & 94.6  & 95.2  & 94.9  & 95.4 \\
          &       & AL    & 0.68  & 0.63  & 0.55  & 1.03  & 0.96  & 0.66 \\[1mm]
          & \multirow{2}[0]{*}{0.75} & CP    & 94.9  & 95.5  & 95.2  & 95.0  & 95.2  & 95.4 \\
          &       & AL    & 0.76  & 0.69  & 0.62  & 1.14  & 1.07  & 0.81 \\[1mm]
          & \multirow{2}[1]{*}{0.90} & CP    & 95.0  & 94.4  & 95.0  & 93.7  & 94.5  & 94.6 \\
          &       & AL    & 0.94  & 0.85  & 0.79  & 1.41  & 1.37  & 1.18 \\
    \bottomrule
    \end{tabular}%
\end{table}%

The CIs for all the methods have satisfactory performance in terms of CP. However,  the CIs using the DRM-EE method have the shortest AL. 
The results indicate that the limiting distribution of the ELR statistic in Corollary \ref{corol_cdf.quantile} works very well, and additional auxiliary information leads to shorter CIs. 
%The CIs under the EMP method do not perform well as other two methods. More precisely, it experiences under coverage problems when $\tau = 0.9$. 

\section{Two real-data applications}
\label{gee_real}

The first dataset \citep{simpson1975bayesian} is from a randomized airborne pyrotechnic seeding experiment, which is designed to test whether seeding clouds with silver iodide increase rainfall. The measurements are the amount of rainfall (in acre-feet) from 52 isolated cumulus clouds, half of which were randomly chosen and massively injected with silver iodide smoke. The rest were untreated. 
We use $D=0$ to indicate untreated clouds and $D=1$ for seeded clouds. 
We estimate the mean ratio $\delta$ of the two populations and construct CIs for $\delta$. 

To use our proposed method to analyze the dataset, we need to choose an appropriate $\q(x)$ in the DRM \eqref{drm_def}. \cite{simpson1975bayesian} and \cite{krishnamoorthy2003inferences} argued that this dataset is highly skewed. This suggests that the two-sample data can be fitted by the DRM with $\q(x) = \log x$. The goodness-of-fit test of \cite{qin1997goodness} gives a $p$-value of 0.568, which indicates that the DRM with $\q(x) = \log x$ provides an adequate fit to the two-sample data. 
Since there is no auxiliary information available, we analyze the data using DRM
and the other methods discussed in Section~\ref{gee_sim1}.
For the point estimates, the EMP method gives $2.685$, while our proposed DRM based estimate is $2.369$. 
As we have demonstrated in Section \ref{gee_sim1.point},
DRM provides smaller MSEs and RBs than EMP, so we expect that the DRM estimate is more accurate.  
We consider the three CIs at the 95\% nominal level, EMP-NA, EMP-EL, and DRM. Table~\ref{cloud} presents the lower bound (LB), the upper bound (UB), and the length of the CIs. 
The EMP-NA CI is significantly longer than the others, and DRM provides the shortest CI.
This agrees with the simulation results in Section \ref{gee_sim1.ci}.  
The LBs of all three CIs are greater than 1, indicating that the seeded clouds slightly increase rainfall. 

\begin{table}[htbp]
  \centering
  \footnotesize
  \caption{Summary of 95\% CIs for $\delta$ (cloud data)}
  \label{cloud}
    \begin{tabular}{c ccc }
    \toprule
          &  \multicolumn{1}{l}{LB} & \multicolumn{1}{l}{UB} & \multicolumn{1}{l}{Length} \\
    %\midrule
    EMP-NA    &   1.13  & 6.36  & 5.23 \\
    EMP-EL    &  1.41  & 5.24  & 3.83 \\
    DRM    & 1.21  & 4.89  & 3.68 \\
    \bottomrule
    \end{tabular}%%
\end{table}%

The second dataset   \citep{hawkins2002diagnostics} is from a  clinical study of cyclosporine measurements in blood samples of organ transplant recipients.
In total,  
56 assay pairs for cyclosporine are obtained by a standard approved method, high-performance liquid chromatography (HPLC),  and an alternative radio-immunoassay (RIA) method. 
We would like to investigate whether the RIA assay is essentially equivalent to the HPLC assay.
The results in \cite{hawkins2002diagnostics} and \cite{bebu2008comparing} indicate 
 that the measurements from the two methods can be modeled by lognormal distributions and 
 have a common mean. 
Since the quantiles are important characteristics of  the population, 
we consider inference on these quantities at  $\tau = 0,25,0.50,0.75$.

Our methods and theory are applicable to two independent samples, but 
in this dataset, two methods are used to measure the same blood sample, 
so the two measurements may be correlated. 
To demonstrate the value of auxiliary information, 
we randomly split the  56 blood samples into two equal groups. 
We use $D= 0$ to indicate the  HPLC method for the first group and $D= 1$ to indicate
the RIA method for the second group.
This gives two independent samples, shown in 
Table~\ref{indep_sample}.
We set $\q(x)$ in the DRM \eqref{drm_def} to $\q(x) = (\log x,\log^2 x)^\top$.
For this choice, the goodness-of-fit test of \cite{qin1997goodness} gives a $p$-value of $0.839$. 
An ELR test to check the validity of the common-mean assumption gives a $p$-value of $0.530$. 
This preliminary analysis indicates that the DRM with the common-mean assumption is reasonable.

\begin{table}[htbp]
\centering
   \footnotesize
  \caption{Measurements from HPLC and RIA methods in two  independent samples}
  \label{indep_sample}
    \begin{tabular}{ccccccc|ccccccc}
    \toprule
    \multicolumn{7}{c|}{HPLC ($D=0$) }                             & \multicolumn{7}{c}{RIA ($D=1$)} \\
    77    & 87    & 93    & 109   & 109   & 129   & 130   & 38    & 98    & 108   & 109   & 111   & 118   & 125 \\
    153   & 156   & 159   & 185   & 198   & 203   & 227   & 130   & 144   & 149   & 162   & 165   & 169   & 172 \\
    244   & 245   & 271   & 280   & 285   & 318   & 336   & 204   & 218   & 234   & 235   & 293   & 294   & 303 \\
    339   & 340   & 440   & 498   & 521   & 556   & 578   & 311   & 341   & 376   & 404   & 406   & 477   & 679 \\
    \bottomrule
    \end{tabular}%
\end{table}%

%We apply the empirical method to each sample and get the EMP quantiles and 95\% EMP-EL CIs for quantiles.We present the EMP quantiles and the LB, UB, and length of CIs in Table \ref{cyc_emp}. 
%The independent data are analyzed under all methods explored in Section \ref{gee_sim2}. 
      
We use the methods of Section~\ref{gee_sim2} to analyze the independent samples. Table~\ref{cyc} summarizes the point estimates and 95\% CIs.
Note that the EL method does not specify how to construct CIs for quantiles with the common-mean assumption.  
We also provide the results of analyzing the original 56 pairs using the EMP method; these are recorded under ``EMP--ALL" in Table~\ref{cyc} and serve as the benchmarks.
Table \ref{cyc} shows that 
the DRM-EE CIs are always shorter than the DRM and EMP CIs. This is in line with our simulation results. 
Although each independent sample is half the size of the original sample, 
 the DRM-EE quantile estimates and CIs 
are similar to the EMP-ALL quantile estimates and CIs. 
This indicates that our method can combine information from two samples and effectively utilize available auxiliary information.

\begin{table}[htbp]
  \centering
\footnotesize
  \caption{Summary of point estimates and 95\% CIs for quantiles (cyclosporine data)}
\label{cyc}
    \begin{tabular}{cccccccccc}
    \toprule
    \multirow{2}[2]{*}{$\tau$} &       & \multicolumn{4}{c}{HPLC ($D=0$)}      & \multicolumn{4}{c}{RIA ($D=1$)} \\
          &       & Estimate & LB    & UB    & Length & Estimate & LB    & UB    & Length \\
    \midrule
    \multirow{5}[1]{*}{0.25} & EMP-ALL & 127   & 109   & 159   & 50    & 141   & 118   & 162   & 50 \\
          & EMP   & 130   & 93    & 198   & 105   & 125   & 108   & 165   & 105 \\
          & EL    & 130   & --    & --    & --    & 130   & --    & --    & -- \\
          & DRM   & 144   & 109   & 185   & 76    & 129   & 108   & 162   & 54 \\
          & DRM-EE & 130   & 109   & 165   & 56    & 130   & 109   & 162   & 53 \\
\midrule
    \multirow{5}[0]{*}{0.5} & EMP-ALL & 206   & 159   & 271   & 112   & 196   & 162   & 287   & 112 \\
          & EMP   & 227   & 156   & 318   & 162   & 172   & 144   & 294   & 162 \\
          & EL    & 227   & --    & --    & --    & 204   & --    & --    & -- \\
          & DRM   & 234   & 162   & 303   & 141   & 198   & 149   & 280   & 131 \\
          & DRM-EE & 218   & 162   & 280   & 118   & 204   & 162   & 280   & 118 \\
\midrule
    \multirow{5}[1]{*}{0.75} & EMP-ALL & 336   & 271   & 402   & 131   & 311   & 287   & 408   & 131 \\
          & EMP   & 336   & 240   & 432   & 192   & 303   & 218   & 388   & 192 \\
          & EL    & 336   & --    & --    & --    & 311   & --    & --    & -- \\
          & DRM   & 339   & 280   & 477   & 197   & 311   & 235   & 406   & 171 \\
          & DRM-EE & 318   & 280   & 404   & 124   & 336   & 280   & 406   & 126 \\
    \bottomrule
    \end{tabular}%
\end{table}%

\section{Discussion}
\label{gee_dis}
We have proposed new and general semiparametric inference procedures to utilize the combined information from two samples as well as auxiliary information formulated  through unbiased EEs. 
We have established the asymptotic normality of the MELEs of the unknown parameters in the DRMs and/or defined through EEs and  the chi-square limiting distributions for the ELR statistics on the parameters. 
We have also derived efficiency results for estimating these parameters and obtained
similar results for inference on the CDFs and population quantiles. 
We have developed an ELR test for checking the validity and usefulness of auxiliary information, and conducted 
simulation studies to evaluate the power of the test.
Our theoretical results and simulation studies demonstrated that the use of DRMs and auxiliary information leads to improved  
efficiency of statistical inferences.

We have focused on two-sample data under the DRM \eqref{drm_def} in the current paper.
This leads to many interesting potential research topics. 
First, we may generalize our results to multiple-sample DRMs \citep{chen2013quantile} with unbiased EEs. 
Second, we may study  other types of parameters, such as the ROC curve and the area under the curve. 
Third, in Example \ref{example2} (a retrospective case-control study with auxiliary information),
it is assumed that the ratio of the total sample size for the internal study to the total sample size for the external study goes to 0. This assumption ensures that the uncertainty of the regression coefficient from the external study 
is negligible.
If the sample sizes of the internal and external studies are comparable, then the variation of the regression coefficient cannot be ignored. Simply discarding the uncertainty may not guarantee efficiency with the auxiliary information \citep{zhang2020bio}. 
We may generalize \cite{zhang2020bio}'s method from the one-sample case to case-control studies
with uncertainty in the regression coefficient for the external study. 
We hope to address these problems in future research.

\section*{Appendix: Regularity conditions}

The asymptotic results in this paper depend on the following regularity conditions. 
We use $||\cdot||$ to denote the Euclidean norm, i.e., $||\cdot||^2$ is the sum of squares of the elements. 

\begin{itemize}
    \item[C1.] The total sample size $n = n_0 + n_1 \to \infty$ and $\lambda^* = n_1/n$ is a constant. 
    \item[C2.]  The two CDFs $F_0$ and $F_1$ satisfy the DRM \eqref{drm_def} with a true parameter value $\btheta^*$, 
     and  $\int \exp \{\btheta^\top \Q(x)\}dF_0(x) <\infty$ in a neighborhood of the true value $\btheta^*$.
    \item[C3.] $\int \Q(x)^\top \Q(x)  dF_0(x)$ exists and is positive definite.
    \item[C4.]  $ E_0\left\{  \g(X; \bpsi^*,\btheta^*) \right\}={\bf 0}$,  $ E_0\left\{  \partial \g(X; \bpsi^*,\btheta^*)/\partial \etab\right\}$ has rank $p$, and  $\int\G(X) \G(X)^\top dF_0(x)$ exists and is positive definite, where $\G(x)$ is defined before Theorem \ref{theorem_eta}.
    \item[C5.] $ \G(x;\etab) $ is twice differentiable with respect to $\etab$, 
    and $||  \G(x,\etab) ||^3$,  $|| \partial \G(x,\etab)/\partial \etab||^2$, and $||\partial \G(x,\etab)/\{\partial \etab\partial\etab^\tau\}|| $ are bounded by some integrable function $R(x)$ with respect to both $F_0$ and $F_1$ in the neighborhood of $\etab^*$.
\end{itemize}

Conditions C1--C3 ensure that the quadratic approximation of the dual likelihood $\ell_{nd}$ in \eqref{del} is applicable. Condition C2 guarantees the existence of finite moments of $\Q(x)$ in a neighborhood of $\btheta^*$. Condition C3 is an identifiability condition, and it ensures that the components of $\Q(x)$ are linearly independent under both $F_i$'s, and hence the elements of $\Q(x)$ except the first cannot be constant functions.  Conditions C3 and C4 together ensure that $\U$ and  $\V$ in Theorem \ref{theorem_eta} have full rank, guaranteeing that $\J$ is invertible. Conditions C1--C5 guarantee that quadratic approximations of the profile empirical log-likelihood $\ell_n(\bpsi,\btheta)$ are applicable.

%\bibliographystyle{agsm}
%\bibliography{reference}

\newpage

{\centering {\large {\bf Supplementary Material for \\ ``Semiparametric empirical likelihood inference with general estimating equations under density ratio models"}}}
\bigskip

\noindent
This document provides supplement materials to the paper entitled ``Semiparametric empirical likelihood inference with general estimating equations under density ratio models."
Section 1 contains more examples of important summary quantities. 
Section 2 presents more details of the extraction of the summary-level information from the external case-control study. 
Section 3 gives additional simulation results, and
Sections 4--11 provide technical details and proofs for the theoretical results presented in the main paper. 

\setcounter{section}{0}
\section{Examples of summary quantities}
In this section, we provide some examples to demonstrate that the estimating equations (EEs)
$E_0\{\g(X;\bpsi,\btheta)\}={\bf 0}$ can define many important summary quantities. 
Recall that the two cumulative distribution functions (CDFs) $F_0$ and $F_1$ are linked via the
density ratio model (DRM): 
\begin{equation}
\label{drm_def2}
  dF_1(x) = \exp\{\btheta^\top\Q(x)\}dF_0(x). 
\end{equation}

\begin{example}(Means and variances)
Let $\mu_i$ and $\sigma_i^2$ be the mean and variance of $F_i$ for $i=0,1$. 
Further, let $\bpsi=(\mu_0,\mu_1,\sigma_0^2,\sigma_1^2)^\top$
and 
$$
\g(x;\bpsi,\btheta)=\left(
\begin{array}{c}
x-\mu_0\\
x\exp\{\btheta^\top\Q(x)\}-\mu_1\\
x^2-\mu_0^2-\sigma_0^2\\
x^2\exp\{\btheta^\top\Q(x)\}-\mu_1^2-\sigma_1^2\\
\end{array}
\right).
$$
Then these means and variances can be defined through $E_0\{\g(X;\bpsi,\btheta)\}={\bf 0}$.
The general uncentered and centered moments can be defined similarly. 

Applying the results in Theorem 2 in the main paper, we can construct an empirical likelihood ratio (ELR) statistic for testing $H_0: \sigma_0^2=\sigma_1^2$, which to our best knowledge is new for such a testing problem. 
\end{example}

\begin{example} (Generalized entropy class of inequality measures)
Suppose the $X_{ij}$'s are positive random variables. 
Let
$$
GE_{i}^{(\xi)}
=\left\{
\begin{array}{ll}
\frac{1}{\xi^2-\xi}\left\{ \int_{0}^\infty \left( \frac{x}{\mu_i} \right )^\xi d F_i(x)-1 \right\},&\mbox{if }\xi\neq0,1,\\
-\int_{0}^\infty \log\left(\frac{x}{\mu_i}\right )dF_i(x),&\mbox{if }\xi=0,\\
\int_{0}^\infty \frac{x}{\mu_i} \log\left(\frac{x}{\mu_i}\right )dF_i(x) ,&\mbox{if }\xi=1\\
\end{array}
\right.
$$
be the generalized entropy class of inequality measures of the $i$th population, $i=0,1$.
We assume that $GE_{i}^{(\xi)}$ exists. 
In our setup, $(GE_0^{(\xi)},GE_1^{(\xi)})^\top$ together with $(\mu_0,\mu_1)$ can also be defined through the EEs. 
For illustration, we consider $\xi=1$. 

Let $\bpsi=(\mu_0,\mu_1,GE_0^{(1)},GE_0^{(1)})^\top$
and 
$$
\g(x;\bpsi,\btheta)=\left(
\begin{array}{c}
x-\mu_0\\
x\exp\{\btheta^\top\Q(x)\}-\mu_1\\
x \log(x/\mu_0) - \mu_0GE_0^{(1)}\\
x \log(x/\mu_1) \exp\{\btheta^\top\Q(x)\}-\mu_1GE_1^{(1)}\\
\end{array}
\right).
$$
Then $(GE_0^{(\xi)},GE_1^{(\xi)})^\top$ together with $(\mu_0,\mu_1)$ can be defined through $E_0\{\g(X;\bpsi,\btheta)\}={\bf 0}$.
For other values of $\xi$, we can define the corresponding EEs similarly. 

Applying the results in Theorem 2 in the main paper, we can also construct an ELR statistic for testing $H_0: GE_0^{(\xi)}=GE_1^{(\xi)}$. 
Again, to our best knowledge this ELR statistic is new for such testing problems. 
\end{example}

\begin{example} (Cumulative distribution functions)
Suppose we are interested in $\zeta_0=F_0(x_0)$ and $\zeta_1=F_1(x_1)$, 
where $x_0$ and $x_1$ are fixed points. 
Let $\bpsi=(\zeta_0,\zeta_1)^\top$
and 
$$
\g(x;\bpsi,\btheta)=\left(
\begin{array}{c}
I(x\leq x_0)-\zeta_0\\
\exp\{\btheta^\top\Q(x)\}I(x\leq x_1)-\zeta_1\\
\end{array}
\right).
$$
Then $(\zeta_0,\zeta_1)^\top$ can be defined through $E_0\{\g(X;\bpsi,\btheta)\}={\bf 0}$.

Applying the results in Theorem 2 in the main paper, we can also construct an ELR-based confidence interval (CI) for $\zeta_0$ or $\zeta_1$ or an ELR-based confidence region for $(\zeta_0,\zeta_1)^\top$. 
\end{example}

\begin{example} (Quantiles)
Suppose we are interested in $\xi_{0,\tau_0}=\inf \{x:F_0(x) \geq \tau_0\}$ and $\xi_{1,\tau_1}=\inf \{x:F_1(x) \geq \tau_1\}$, where $\tau_0,\tau_1\in(0,1)$. 
Let $\bpsi=(\zeta_0,\zeta_1)^\top$
and 
$$
\g(x;\bpsi,\btheta)=\left(
\begin{array}{c}
I(x\leq \xi_{0,\tau_0} )-\tau_0\\
\exp\{\btheta^\top\Q(x)\}I(x\leq \xi_{1,\tau_1})-\tau_1\\
\end{array}
\right).
$$
Then $(\xi_{0,\tau_0},\xi_{1,\tau_1})^\top$ can be defined through $E_0\{\g(X;\bpsi,\btheta)\}={\bf 0}$.

Applying the result of Corollary 2 or 3 in the main paper, we can also construct an ELR-based CI for $\xi_{0,\tau_0}$ or $\xi_{1,\tau_1}$ or an ELR-based confidence region for $(\xi_{0,\tau_0},\xi_{1,\tau_1})^\top$. 
\end{example}

\section{Summary-level information from external case-control studies}
Let $\{(Y_{i},D_i):i=1,\ldots,n_E\}$ be the data from an external study, 
where $D_i=0$ or 1 indicates that the individual is from a disease-free or diseased group.  
We model the relationship between $D$ and $Y$ through a logistic regression model, which may not be the true model:  
\begin{equation}
\label{external.logis}
h(Y;\alpha_Y, \bbeta_Y)= P(D=1|Y)=\frac{\exp(\alpha_Y+\bbeta_Y^\top Y)}{1+\exp(\alpha_Y+\bbeta_Y^\top Y)}. 
\end{equation}

Let 
$$
\a(\alpha_Y, \bbeta_Y)
=\frac{1}{n_E}\sum_{i=1}^{n_{E}}\{D_i-h(Y_i;\alpha_Y, \bbeta_Y) \} (1,Y^\top)^\top,
$$
which are the score functions based on the logistic regression model in \eqref{external.logis}. 
Further, let $(\alpha_Y^*,\bbeta_Y^*)$ be the solution to  $E\{ \a(\alpha_Y, \bbeta_Y)\}={\bf 0}$. 
That is,  
$$
E\{ \a(\alpha_Y^*, \bbeta_Y^*)\}={\bf 0}.
$$
Note that $(\alpha_Y^*, \bbeta_Y^*)$ may not be known exactly. 
We can solve the score equations $\a(\alpha_Y, \bbeta_Y)={\bf 0}$
to obtain the estimator $(\hat \alpha_Y, \hat\bbeta_Y)$. 
That is, $\a(\hat \alpha_Y, \hat\bbeta_Y)={\bf 0}$.
Assume that we have access to the estimator $(\hat \alpha_Y, \hat\bbeta_Y)$ but not necessarily to the individual-level data $\{(Y_{i},D_i):i=1,\ldots,n_E\}$. 

When the total sample size $n=n_0+n_1$ for the internal study 
satisfies $n/n_E\to 0$, we can use $(\hat \alpha_Y, \hat\bbeta_Y)$ for $(\alpha_Y^*, \bbeta_Y^*)$. 
This will cause a negligible error for inference for the internal study. 
In the following, we assume that $(\alpha_Y^*, \bbeta_Y^*)$ is known and we denote 
$
h(y)= h(y;\alpha_Y^*, \bbeta_Y^*). 
$

Next, we discuss how to summarize the information from $
E\{ \a(\alpha_Y^*, \bbeta_Y^*)\}={\bf 0}
$ 
into unbiased EEs with respect to $F_0$, which is the setup in the main paper. 
When the external study is a prospective case-control study, by defining the unknown overall disease prevalence $\pi = P(D =1)$, we have 
\begin{eqnarray}
E\{ \a(\alpha_Y^*, \bbeta_Y^*)\}
&=&E\left[\{D-h(Y) \} (1,Y^\top)^\top \right]\label{ee.pros1.jasa}\\
&=&E_0\Big(
[ -(1-\pi) h(Y) +\pi \exp\{\btheta^\top\Q(X)\}\{1-h(Y)\} ](1,Y^\top)^\top
\Big),  \label{ee.pros1.our}
\end{eqnarray}
where we have used the law of total expectation and the DRM  \eqref{drm_def2} in the last step. 

When the external study is a retrospective case-control study, we have 
\begin{eqnarray}
\label{ee.retro1.jasa}
E\{ \a(\alpha_Y^*, \bbeta_Y^*)\} 
&=&-(1-\pi_E) E_0 \{ h(Y)  (1,Y^\top)^\top\} +\pi_E E_1 [ \{ 1- h(Y)\}  (1,Y^\top)^\top\}],\end{eqnarray}
where $E_1$ represents the expectation operators with respect to $F_1$, and  $\pi_E$ is the proportion of diseased individuals in the external case-control study. Note that $\pi_E$ is a known and fixed value.

Using the DRM \eqref{drm_def2}, we further get 
\begin{eqnarray}
E\{ \a(\alpha_Y^*, \bbeta_Y^*)\} 
&=&E_0\Big(
[ -(1-\pi_E) h(Y) +\pi_E \exp\{\btheta^\top\Q(X)\}\{1-h(Y)\} ](1,Y^\top)^\top
\Big). \label{ee.retro1.our}
\end{eqnarray}

Summarizing \eqref{ee.pros1.our} and \eqref{ee.retro1.our}, we have that 
 if the external study is a prospective case-control study, then 
$E_0\{\g(X;\bpsi,\btheta)\}={\bf 0}$, 
where 
\begin{equation*} 
\g(x;\bpsi,\btheta) =[ -(1-\pi) h(y) +\pi \exp\{\btheta^\top\Q(x)\}\{1-h(y)\} ](1,y^\top)^\top
\end{equation*}
with $\bpsi=\pi$; 
if the external study is a retrospective case-control study, then $E_0\{\g(X;\btheta)\}={\bf 0}$, where 
\begin{equation*} 
\g(x;\btheta) =[ -(1-\pi_E) h(y) +\pi_E \exp\{\btheta^\top\Q(x)\}\{1-h(y)\} ](1,y^\top)^\top.
\end{equation*}

Similarly, we summarize the information from $
E\{ \a(\alpha_Y^*, \bbeta_Y^*)\}={\bf 0}
$ 
into unbiased EEs with respect to the joint distribution of $(D,Y)$, which is the setup in \cite{chatterjee2016constrained}.
Note that when the external study is a retrospective case-control study, \eqref{ee.retro1.jasa} can be equivalently written as 
\begin{eqnarray}
\label{ee.retro2.jasa}
E\{ \a(\alpha_Y^*, \bbeta_Y^*)\} 
=E \left[ \frac{1-\pi_E}{1-\pi} (1-D) \{D-h(Y)\}(1,Y^\top)^\top +  \frac{\pi_E}{\pi} D \{D-h(Y)\}  (1,Y^\top)^\top \right ].\end{eqnarray}
 
 Summarizing \eqref{ee.pros1.jasa} and \eqref{ee.retro2.jasa}, we have that if 
the external study is a prospective case-control study, then 
$E\{\bu(D, Y)\}={\bf 0}$, 
where 
\begin{equation*} 
%\label{ee.pros}
\bu(D, Y) =\{D-h(Y) \} (1,Y^\top)^\top;
\end{equation*}
if the external study is a retrospective case-control study, then $E\{\bu(D, Y;\pi)\}={\bf 0}$, where
\begin{equation*} 
%\label{ee.retro}
\bu(D, Y;\pi) =\frac{1-\pi_E}{1-\pi} (1-D) \{D-h(Y)\}(1,Y^\top)^\top +  \frac{\pi_E}{\pi} D \{D-h(Y)\}  (1,Y^\top)^\top.
\end{equation*}

Note that the method and theory in \cite{chatterjee2016constrained} are applicable when there is no unknown parameter in the functions $\bu(\cdot)$. 
Hence, their general results do not apply when the external study is a retrospective case-control study.

\section{Additional simulation under the gamma distributional setting}
Table \ref{tab_quant_point_gamma} presents the four quantile estimates under gamma distributions. 
Table \ref{tab_quant_ci_gamma}  presents the three CIs for quantiles under gamma distributions. 
The general summary statements are similar to those for normal distributions, and hence are omitted. 
\begin{table}[!htbp]
  \centering
  \footnotesize
  \caption{RB (\%) and MSE ($\times 100$) for quantile estimators (gamma distributions)}
  \label{tab_quant_point_gamma}
    \begin{tabular}{ccccccccccc}
    \toprule
    \multirow{2}[2]{*}{$(n_0,n_1)$} & \multirow{2}[2]{*}{$\tau$} &       & \multicolumn{4}{c}{$Gam(8,1.125)$} & \multicolumn{4}{c}{$Gam(6,1,5)$} \\
          &       &       & \multicolumn{1}{c}{EMP} & \multicolumn{1}{c}{EL} & \multicolumn{1}{c}{DRM} & \multicolumn{1}{c}{DRM-EE} & \multicolumn{1}{c}{EMP} & \multicolumn{1}{c}{EL} & \multicolumn{1}{c}{DRM} & \multicolumn{1}{c}{DRM-EE} \\
    \midrule
    \multirow{10}[2]{*}{$(50,50)$} & \multirow{2}[1]{*}{0.10} & RB    & -2.25 & -0.05 & 0.25  & 0.16  & -1.40 & 0.71  & 1.26  & 0.65 \\
          &       & MSE   & 29.71 & 25.04 & 23.26 & 20.29 & 31.70 & 26.96 & 26.66 & 22.88 \\[1mm]
          & \multirow{2}[0]{*}{0.25} & RB    & 0.01  & -0.04 & 0.08  & 0.03  & 0.75  & 0.30  & 0.47  & -0.06 \\
          &       & MSE   & 25.02 & 19.93 & 21.38 & 16.39 & 32.91 & 24.71 & 27.78 & 20.32 \\[1mm]
          & \multirow{2}[0]{*}{0.50} & RB    & -1.03 & -0.04 & -0.15 & -0.02 & -0.74 & -0.07 & 0.28  & -0.08 \\
          &       & MSE   & 30.99 & 23.20 & 25.91 & 17.32 & 40.46 & 25.74 & 35.52 & 19.68 \\[1mm]
          & \multirow{2}[0]{*}{0.75} & RB    & -0.13 & -0.02 & -0.33 & -0.13 & -0.02 & -0.20 & 0.15  & 0.12 \\
          &       & MSE   & 48.41 & 35.85 & 42.11 & 28.23 & 65.70 & 43.10 & 57.48 & 33.81 \\[1mm]
          & \multirow{2}[1]{*}{0.90} & RB    & -1.85 & 0.15  & -0.47 & -0.20 & -1.93 & 0.01  & 0.12  & 0.14 \\
          &       & MSE   & 99.19 & 86.91 & 83.12 & 62.28 & 133.79 & 110.01 & 120.28 & 86.79 \\
    \midrule
    \multirow{10}[2]{*}{$(50,150)$} & \multirow{2}[1]{*}{0.10} & RB    & -2.25 & 0.05  & 0.41  & 0.32  & -0.36 & 0.36  & 0.42  & 0.33 \\
          &       & MSE   & 29.98 & 23.32 & 20.31 & 15.18 & 10.40 & 9.74  & 9.86  & 9.10 \\[1mm]
          & \multirow{2}[0]{*}{0.25} & RB    & -0.02 & 0.01  & 0.03  & -0.02 & 0.19  & 0.09  & 0.12  & -0.03 \\
          &       & MSE   & 25.11 & 17.45 & 19.28 & 11.05 & 10.58 & 9.27  & 9.89  & 8.61 \\[1mm]
          & \multirow{2}[0]{*}{0.50} & RB    & -1.03 & 0.02  & -0.18 & -0.01 & -0.21 & 0.01  & 0.12  & -0.03 \\
          &       & MSE   & 31.26 & 17.31 & 22.92 & 9.55  & 14.17 & 11.46 & 12.98 & 10.15 \\[1mm]
          & \multirow{2}[0]{*}{0.75} & RB    & -0.15 & 0.04  & -0.45 & -0.16 & -0.06 & -0.18 & 0.02  & -0.06 \\
          &       & MSE   & 48.19 & 27.80 & 36.99 & 15.98 & 21.18 & 17.52 & 19.94 & 15.74 \\[1mm]
          & \multirow{2}[1]{*}{0.90} & RB    & -1.83 & 0.42  & -0.56 & -0.09 & -0.62 & -0.05 & 0.11  & 0.03 \\
          &       & MSE   & 99.26 & 74.83 & 74.58 & 43.00 & 44.60 & 40.83 & 40.68 & 36.31 \\
    \midrule
    \multirow{10}[2]{*}{$(100,100)$} & \multirow{2}[1]{*}{0.10} & RB    & -1.03 & 0.07  & 0.41  & 0.32  & -0.92 & 0.25  & 0.35  & 0.15 \\
          &       & MSE   & 14.47 & 13.00 & 11.19 & 9.91  & 16.95 & 14.43 & 14.18 & 11.95 \\[1mm]
          & \multirow{2}[0]{*}{0.25} & RB    & -0.54 & 0.06  & 0.06  & 0.03  & -0.52 & -0.02 & 0.10  & -0.12 \\
          &       & MSE   & 12.76 & 10.64 & 10.81 & 8.35  & 15.41 & 11.85 & 13.73 & 9.82 \\[1mm]
          & \multirow{2}[0]{*}{0.50} & RB    & -0.48 & 0.03  & -0.03 & -0.02 & -0.41 & 0.02  & 0.14  & -0.03 \\
          &       & MSE   & 15.70 & 11.67 & 12.92 & 8.89  & 20.57 & 13.41 & 17.58 & 9.84 \\[1mm]
          & \multirow{2}[0]{*}{0.75} & RB    & -0.61 & -0.04 & -0.19 & -0.14 & -0.71 & -0.17 & 0.04  & -0.06 \\
          &       & MSE   & 24.94 & 18.73 & 19.98 & 13.73 & 32.29 & 20.67 & 27.94 & 16.02 \\[1mm]
          & \multirow{2}[1]{*}{0.90} & RB    & -0.94 & 0.05  & -0.20 & -0.09 & -1.11 & 0.03  & 0.01  & 0.09 \\
          &       & MSE   & 48.17 & 42.30 & 41.07 & 31.30 & 70.72 & 54.02 & 57.47 & 40.26 \\
    \midrule
    \multirow{10}[2]{*}{$(200,200)$} & \multirow{2}[1]{*}{0.10} & RB    & -0.44 & 0.04  & 0.24  & 0.16  & -0.50 & 0.15  & 0.15  & 0.07 \\
          &       & MSE   & 7.03  & 6.34  & 5.54  & 4.80  & 8.17  & 7.06  & 6.80  & 5.81 \\[1mm]
          & \multirow{2}[0]{*}{0.25} & RB    & -0.29 & 0.01  & 0.08  & 0.05  & -0.31 & -0.04 & -0.01 & -0.10 \\
          &       & MSE   & 6.53  & 5.24  & 5.19  & 3.92  & 7.59  & 5.89  & 6.52  & 4.79 \\[1mm]
          & \multirow{2}[0]{*}{0.50} & RB    & -0.23 & 0.02  & -0.03 & -0.03 & -0.31 & -0.11 & -0.02 & -0.07 \\
          &       & MSE   & 7.83  & 5.84  & 6.15  & 4.25  & 9.90  & 6.03  & 8.39  & 4.76 \\[1mm]
          & \multirow{2}[0]{*}{0.75} & RB    & -0.38 & -0.12 & -0.11 & -0.10 & -0.29 & 0.05  & 0.02  & 0.03 \\
          &       & MSE   & 11.98 & 9.21  & 10.19 & 7.24  & 17.41 & 11.09 & 14.98 & 8.33 \\[1mm]
          & \multirow{2}[1]{*}{0.90} & RB    & -0.48 & 0.00  & -0.09 & -0.07 & -0.42 & 0.09  & 0.08  & 0.13 \\
          &       & MSE   & 23.81 & 20.31 & 19.73 & 15.34 & 36.06 & 26.76 & 31.15 & 20.87 \\
    \bottomrule
    \end{tabular}%
\end{table}%

% Table generated by Excel2LaTeX from sheet 'CI'
\begin{table}[htbp]
  \centering
  \footnotesize
  \caption{CP (\%) and AL for  three 95\% CIs of $100\tau\%$-quantile (gamma distributions)}
  \label{tab_quant_ci_gamma}
    \begin{tabular}{ccccccccc}
    \toprule
    \multirow{2}[2]{*}{$(n_0,n_1)$} & \multirow{2}[2]{*}{$\tau$} &       & \multicolumn{3}{c}{$Gam(8,1.125)$} & \multicolumn{3}{c}{$Gam(6,1.5)$} \\
          &       &       & EMP    & DRM   & DRM-EE & EMP    & DRM   & DRM-EE \\
    \midrule
    \multirow{10}[2]{*}{(50,50)} & \multirow{2}[1]{*}{0.10} & CP    & 94.7  & 95.1  & 95.5  & 93.7  & 94.5  & 94.9 \\
          &       & AL    & 2.10  & 1.89  & 1.77  & 2.24  & 2.10  & 1.93 \\
          & \multirow{2}[0]{*}{0.25} & CP    & 94.9  & 94.7  & 94.5  & 95.4  & 94.5  & 94.8 \\
          &       & AL    & 2.03  & 1.82  & 1.60  & 2.25  & 2.04  & 1.73 \\
          & \multirow{2}[0]{*}{0.50} & CP    & 93.2  & 94.4  & 94.3  & 94.2  & 95.1  & 94.9 \\
          &       & AL    & 2.06  & 1.99  & 1.62  & 2.33  & 2.31  & 1.74 \\
          & \multirow{2}[0]{*}{0.75} & CP    & 94.2  & 94.2  & 94.0  & 95.8  & 93.7  & 93.7 \\
          &       & AL    & 2.86  & 2.55  & 2.10  & 3.46  & 3.04  & 2.29 \\
          & \multirow{2}[1]{*}{0.90} & CP    & 94.8  & 94.7  & 94.9  & 94.7  & 94.3  & 94.9 \\
          &       & AL    & 4.17  & 3.73  & 3.27  & 5.03  & 4.68  & 3.80 \\
    \midrule
    \multirow{10}[2]{*}{(50,150)} & \multirow{2}[1]{*}{0.10} & CP    & 94.7  & 95.2  & 95.4  & 94.1  & 95.2  & 95.5 \\
          &       & AL    & 2.10  & 1.77  & 1.56  & 1.29  & 1.24  & 1.20 \\
          & \multirow{2}[0]{*}{0.25} & CP    & 94.9  & 94.8  & 94.7  & 94.8  & 94.5  & 94.5 \\
          &       & AL    & 2.03  & 1.72  & 1.33  & 1.28  & 1.22  & 1.14 \\
          & \multirow{2}[0]{*}{0.50} & CP    & 93.2  & 94.7  & 94.7  & 94.2  & 94.1  & 94.0 \\
          &       & AL    & 2.06  & 1.86  & 1.20  & 1.37  & 1.37  & 1.22 \\
          & \multirow{2}[0]{*}{0.75} & CP    & 94.2  & 94.4  & 94.9  & 95.9  & 95.6  & 95.4 \\
          &       & AL    & 2.86  & 2.41  & 1.58  & 1.88  & 1.79  & 1.60 \\
          & \multirow{2}[1]{*}{0.90} & CP    & 94.8  & 94.8  & 95.0  & 94.4  & 95.4  & 95.1 \\
          &       & AL    & 4.17  & 3.44  & 2.60  & 2.72  & 2.68  & 2.47 \\
    \midrule
    \multirow{10}[2]{*}{(100,100)} & \multirow{2}[1]{*}{0.10} & CP    & 95.0  & 95.2  & 94.5  & 95.0  & 94.0  & 93.7 \\
          &       & AL    & 1.53  & 1.33  & 1.24  & 1.66  & 1.47  & 1.35 \\
          & \multirow{2}[0]{*}{0.25} & CP    & 94.8  & 94.4  & 93.9  & 95.2  & 95.1  & 95.1 \\
          &       & AL    & 1.42  & 1.28  & 1.12  & 1.58  & 1.44  & 1.22 \\
          & \multirow{2}[0]{*}{0.50} & CP    & 93.8  & 94.9  & 94.2  & 94.3  & 94.3  & 94.3 \\
          &       & AL    & 1.48  & 1.39  & 1.14  & 1.73  & 1.61  & 1.22 \\
          & \multirow{2}[0]{*}{0.75} & CP    & 95.0  & 94.5  & 95.4  & 95.2  & 95.5  & 94.9 \\
          &       & AL    & 1.99  & 1.78  & 1.46  & 2.33  & 2.11  & 1.60 \\
          & \multirow{2}[1]{*}{0.90} & CP    & 96.2  & 95.5  & 94.8  & 94.9  & 94.9  & 95.3 \\
          &       & AL    & 3.01  & 2.57  & 2.25  & 3.58  & 3.15  & 2.60 \\
    \midrule
    \multirow{10}[2]{*}{(200,200)} & \multirow{2}[1]{*}{0.10} & CP    & 93.8  & 95.2  & 94.7  & 93.8  & 95.2  & 95.4 \\
          &       & AL    & 1.02  & 0.94  & 0.87  & 1.10  & 1.03  & 0.95 \\
          & \multirow{2}[0]{*}{0.25} & CP    & 95.4  & 95.4  & 95.2  & 94.2  & 95.1  & 94.8 \\
          &       & AL    & 0.99  & 0.90  & 0.79  & 1.10  & 1.01  & 0.85 \\
          & \multirow{2}[0]{*}{0.50} & CP    & 94.4  & 95.0  & 94.8  & 94.2  & 94.8  & 94.8 \\
          &       & AL    & 1.05  & 0.98  & 0.81  & 1.21  & 1.13  & 0.85 \\
          & \multirow{2}[0]{*}{0.75} & CP    & 95.1  & 95.0  & 95.0  & 95.5  & 94.8  & 94.9 \\
          &       & AL    & 1.37  & 1.26  & 1.04  & 1.61  & 1.49  & 1.13 \\
          & \multirow{2}[1]{*}{0.90} & CP    & 93.7  & 94.9  & 94.7  & 94.6  & 94.1  & 95.0 \\
          &       & AL    & 1.94  & 1.78  & 1.55  & 2.30  & 2.17  & 1.80 \\
    \bottomrule
    \end{tabular}%
\end{table}%

\section{Some Preliminary Results}

Recall that the profile empirical log-likelihood of $(\bpsi, \btheta)$ is 
\begin{equation*}
 % \label{ell_1}
  \ell_n(\bpsi, \btheta) = -\Sumij \log \left\{1 + \lambda\left[\exp\{\btheta^\top\Q(X_{ij})\} -1\right] + \bnu^\top \g(X_{ij};\bpsi, \btheta)\right\} + \Sumj \btheta^\top\Q(X_{1j}),
\end{equation*}
where the Lagrange multipliers satisfy 
\begin{eqnarray*}
%\label{solution_lambda}
\Sumij \frac{\exp\{\btheta^\top\Q(X_{ij})\}-1}{1 + \lambda\left[\exp\{\btheta^\top\Q(X_{ij})\} -1\right] + \bnu^\top \g(X_{ij};\bpsi, \btheta)} = 0,\\
%\label{solution_z}
\Sumij \frac{\g(X_{ij};\bpsi, \btheta)}{1 + \lambda\left[\exp\{\btheta^\top\Q(X_{ij})\} -1\right] + \bnu^\top \g(X_{ij};\bpsi, \btheta)} = {\bf 0}.
\end{eqnarray*}
Then $ \ell_n(\bpsi, \btheta)$ can be rewritten as 
$$ 
\ell_n(\bpsi, \btheta)=\inf \limits_{\lambda,\bnu}l_n(\bpsi, \btheta,\lambda,\bnu),
$$ where 
\begin{eqnarray*}
%\label{ell_n}
 && l_n(\bpsi, \btheta,\lambda,\bnu) \\
 &=& -\Sumij \log \lb 1+\lambda \Lm \exp\lb \btheta^\top \Q(X_{ij})\rb-1\Rm+\bnu^\top\g(X_{ij};\bpsi,\btheta) \rb + \Sumj \{\btheta^\top\Q(X_{1j})\}.
\end{eqnarray*}
Equivalently, $\ell_n(\bpsi, \btheta)=l_n(\bpsi, \btheta,\lambda,\bnu)$ with $\lambda$ and $\bnu$ being the solution to 
$$
\frac{\partial l_n(\bpsi, \btheta,\lambda,\bnu)}{\partial \lambda}=0~~\mbox{and}~~ 
\frac{\partial l_n(\bpsi, \btheta,\lambda,\bnu)}{\partial \bnu}={\bf 0}. 
$$

With the above preparation, 
it can be verified that the maximum empirical likelihood estimate (MELE) $(\hat\bpsi, \hat\btheta)$ of $(\bpsi, \btheta)$ 
and the corresponding Lagrange multipliers $(\hat\lambda,\hat\bnu)$ satisfy
\begin{eqnarray*}
\frac{\partial l_n(\hat\bpsi, \hat\btheta,\hat\lambda,\hat\bnu)}{\partial \btheta}={\bf 0},~~
\frac{\partial l_n(\hat\bpsi, \hat\btheta,\hat\lambda,\hat\bnu)}{\partial \bbeta}={\bf 0},~~
\frac{\partial l_n(\hat\bpsi, \hat\btheta,\hat\lambda,\hat\bnu)}{\partial \lambda}=0,~~
\frac{\partial l_n(\hat\bpsi, \hat\btheta,\hat\lambda,\hat\bnu)}{\partial \bnu}={\bf 0}. 
\end{eqnarray*}
To investigate the asymptotic properties of $\hat\bpsi$ and $\hat\btheta$, we need their approximations. 
We first find the first and second derivatives of $l_n(\bpsi, \btheta,\lambda,\bnu)$. 

For convenience of presentation, we recall and define some notation. 
%For convenience of presentation, we use $l_n(\bgamma)$ to denote $l_n(\bpsi, \btheta,\lambda,\bnu)$ and $\g(x ;\etab)$
%to denote $\g(x;\bpsi, \btheta)$. 
We use $\bpsi^*$ and $\btheta^*$ to denote the true values of $\bpsi$ and $\btheta$. 
%We denote $\etab^* = (\bpsi$ be the true value of $\etab$, $\lambda^*=n_1/n$, 
%and $\bgamma^* = (\etab^\top_0,\lambda^*,{\bf 0}^\top)^\top$.
Let 
\begin{eqnarray*}
&& \omega(x; \btheta) = \exp\left\{ \btheta^\top \Q(x)\right\},
~~\omega(x) = \omega(x; \btheta^*),
~~\lambda^*=n_1/n,\\
&&h(x) = 1 + \lambda^* \left\{ \omega(x) - 1 \right\},
~~h_1(x) = \lambda^*\omega(x)/h(x),
~~h_0(x)=(1-\lambda^*)/h(x),\\
&&\G(x;\bpsi,\btheta) = (\omega(x; \btheta) - 1, \g(x;\bpsi,\btheta)^\top)^\top,
~~\G(x) = \G(x;\bpsi^*,\btheta^*).
\end{eqnarray*}
Note that $\omega(\cdot)$, $h(\cdot)$, $h_0(\cdot)$, $h_1(\cdot)$, and $\G(\cdot)$ depend on $\bpsi^*$ and/or $\btheta^*$, and $h_0(x) + h_1(x) =1$.
By Condition C1, $\lambda^*$ is a fixed value and does not depend on the total sample size $n$. 

%where $E_0(\cdot)$ is the expectation operator with respect to $F_0$.

Recall that $\etab=(\bpsi^\top,\btheta^\top)^\top$ and $\bu =(\lambda, \bnu^\top)^\top $. 
Let  $\bgamma = (\etab^\top, \bu^\top)$. We further define
\begin{eqnarray*}
&&\hat\etab=({\hat\bpsi}^\top,{\hat\btheta}^\top)^\top,
~~\hat\bu =(\hat\lambda, \hat\bnu^\top)^\top,
~~ \hat\bgamma = (\hat\etab^\top, \hat\bu^\top),\\
&&\etab^*=({\bpsi^*}^\top,{\btheta^*}^\top)^\top,
~~\bu^* =(\lambda^*, {\bf 0}_{1\times r})^\top,
~~\bgamma^* = ({\etab^*}^\top, {\bu^*}^\top).
\end{eqnarray*}
In the following, 
we use $l_n(\bgamma)$, $\g(x;\etab)$, and $\G(x;\etab) $
to denote $l_n(\bpsi, \btheta,\lambda,\bnu)$, $\g(x;\bpsi,\btheta)$, and $\G(x;\bpsi,\btheta) $.

\subsection{First and second derivatives of $l_n(\bgamma)$}
After some algebra, the first derivatives of $l_n(\bgamma)$ are found to be: 
\begin{eqnarray*}
 \frac{\partial l_n(\bgamma)}{\partial \bpsi}&=&
 -\Sumij \frac{ \{\partial \g(X_{ij};\etab)/\partial \bpsi\}^\top \bnu}{1+\lambda \lb \omega(X_{ij};\btheta)-1\rb +\bnu^\top\g(X_{ij};\etab)},\\
  \frac{\partial l_n(\bgamma)}{\partial \btheta}&=&
  -\Sumij \frac{\lambda \omega(X_{ij};\btheta)\Q(X_{ij})+\{ \partial \g(X_{ij};\etab)/\partial \btheta\}^\top \bnu}{1+\lambda \lb \omega(X_{ij};\btheta)-1\rb +\bnu^\top\g(X_{ij};\etab)} + \Sumj \Q(X_{1j}), \\
  \frac{\partial l_n(\bgamma)}{\partial \bu}&=&
  -\Sumij \frac{\G(X_{ij};\etab)}{1+\lambda \lb \omega(X_{ij};\btheta)-1\rb +\bnu^\top\g(X_{ij};\etab)}.
\end{eqnarray*}
Then the first derivatives at the true values $\etab^*$ and $\bu^*$ are 
\begin{equation*}
\S_n = 
\frac{\partial l_n(\bgamma^*)}{\partial \bgamma} = 
\ba{c}
 \frac{\partial l_n(\bgamma^*)}{\partial \bpsi}\\
 \frac{\partial l_n(\bgamma^*)}{\partial \btheta}\\
 \frac{\partial l_n(\bgamma^*)}{\partial \bu}
 \ea
 = \ba{c} {\bf 0} \\ \S_{n\bstheta}\\\S_{n\bu}\ea ,
\end{equation*}
where 
\begin{eqnarray*}
\S_{n\bstheta} = \Sumj \Q(X_{1j}) - \Sumij h_1(X_{ij})\Q(X_{ij}),~~
\S_{n\bu} = -\Sumij \frac{\G(X_{ij})}{h(X_{ij})}.
\end{eqnarray*}

Similarly, we calculate the second derivatives of $l_n(\bgamma)$. Evaluating them at $\bgamma^*$ gives: 
\begin{equation}
  \label{2nd_deriv}
 \frac{\partial^2 l_n(\bgamma^*)}{\partial \bgamma \partial \bgamma^\top} =
 \ba{ccccc}
 {\bf 0} &{\bf 0}&\frac{\partial^2 l_n(\bgamma^*)}{\partial \btheta \partial \bu^\top}\\
 {\bf 0} & \frac{\partial^2 l_n(\bgamma^*)}{\partial \bbeta \partial \bbeta^\top} & \frac{\partial^2 l_n(\bgamma^*)}{\partial \bbeta \partial \bu^\top}\\
  
 \frac{\partial^2 l_n(\bgamma^*)}{\partial \bu \partial \btheta^\top} & \frac{\partial^2 l_n(\bgamma^*)}{\partial \bu \partial \bbeta^\top}& \frac{\partial^2 l_n(\bgamma^*)}{\partial \bu \partial \bu^\top}
 \ea,
\end{equation}
where 
\begin{eqnarray*}
\frac{\partial^2 l_n(\bgamma^*)}{ \partial \bpsi \partial \bu^\top} &=& 
\ls\frac{\partial^2 l_n(\bgamma^*)}{ \partial \bu \partial \bpsi^\top} \rs^\top= - \Sumij \frac{ \{\partial \G(X_{ij};\etab^*)/ \partial \bpsi\}^\top}{h(X_{ij})};\\
\frac{\partial^2 l_n(\bgamma^*)}{ \partial \btheta \partial \btheta^\top} &=&
- \Sumij h_0(X_{ij}) h_1(X_{ij})\Q(X_{ij})\Q(X_{ij})^\top;\\
\frac{\partial^2 l_n(\bgamma^*)}{ \partial \btheta \partial \bu^\top} &=& 
\ls\frac{\partial^2 l_n(\bgamma^*)}{ \partial \bu \partial \btheta^\top } \rs^\top \\
&=&
 \Sumij \frac{h_1(X_{ij}) \Q(X_{ij}) \G(X_{ij})^\top}{h(X_{ij})} -\Sumij \frac{\{\partial \G(X_{ij};\etab^*)/\partial \btheta\}^\top }{h(X_{ij})};\\
\frac{\partial^2 l_n(\bgamma^*)}{\partial \bu \partial \bu^\top} &=& \Sumij\frac{\G(X_{ij})\G(X_{ij})^\top}{h(X_{ij})^2}.
\end{eqnarray*}

\subsection{Some useful lemmas}
We first review a lemma from the supplementary material of \cite{qin2014using}, which helps to ease the calculation in our proofs. 
In the following, we assume that the DRM \eqref{drm_def2} is satisfied as required in Condition C2.  
%In the following, we use $E_1$ to denote the expectation operator with respect to $F_1$.
\begin{lemma}
\label{expectation}
Suppose that $\mathcal{S}$ is an arbitrary vector-valued function.
Let $E_0(\cdot)$ represent the expectation operator with respect to $F_0$ and $X$ refer to a random variable from $F_0$.
Then we have for $j =1,\cdots, n_1$,
\begin{equation*}
  E\lb \mathcal{S}(X_{1j})\rb = E_0\lb \omega(X) \mathcal{S}(X)\rb~~~{\rm and}~~~E\lb \Sumij \mathcal{S}(X_{ij}) \rb = nE_0\lb \mathcal{S}(X)h(X)\rb. 
\end{equation*}
\end{lemma}

\begin{proof}
Under the DRM with true parameter $\btheta^*$, we have 
\begin{eqnarray*}
E\lb \mathcal{S}(X_{1j})\rb = \int \mathcal{S}(x) dF_1(x)= \int \mathcal{S}(x)\omega(x) dF_0(x)= E_0\lb \omega(X) \mathcal{S}(X)\rb.
\end{eqnarray*}

Using the fact that $\lambda^* = n_1/n$ and the definition of the function $h(\cdot)$, we further have 
\begin{eqnarray*}
E\lb \Sumij \mathcal{S}(X_{ij}) \rb &=& n_0E_0\lb \mathcal{S}(X)\rb + n_1 E_0\lb \omega(X) \mathcal{S}(X)\rb\\
&=& n\Lm (1-\lambda^*)E_0\lb \mathcal{S}(X)\rb + \lambda^* E_0\lb \omega(X) \mathcal{S}(X)\rb \Rm\\
&=& n E_0\Lm \{(1- \lambda^*) +\lambda^*\omega(X)\} \mathcal{S}(X)\Rm\\
&=& nE_0\lb \omega(X) \mathcal{S}(X)\rb.
\end{eqnarray*}

This completes the proof.
\end{proof}

Recall that 
\begin{eqnarray*}
&&
\A_{\bstheta\bstheta} = (1 - \lambda^*) E_0 \lb h_1(X)\Q(X)\Q(X)^\top \rb,
\\
&&\A_{\bstheta\bu} = \A_{\bu\bstheta}^\top = E_0 \lb \frac{\partial \G(X;\etab^*)}{ \partial \btheta} \rb^\top - E_0 \lb h_1(X) \Q(X) \G(X)^\top\rb,\\
&&\A_{\bspsi\bu} = \A_{\bu\bspsi}^\top = E_0 \lb \frac{\partial \G(X;\etab^*)}{ \partial \bpsi} \rb^\top ,~~~
\A_{\bu\bu} = E_0 \lb \frac{\G(X)\G(X)^\top}{h(X)} \rb.
\end{eqnarray*}
Applying Lemma \ref{expectation}, after some algebra, we have the following Lemma.
\begin{lemma}
\label{2nd_expect}
\begin{enumerate}
\item[(a)]
With the form of $\partial^2 l_n(\bgamma^*)/(\partial \bgamma \partial \bgamma^\top)$ defined in \eqref{2nd_deriv}, we have
\begin{equation*}
  -\frac{1}{n}E\lb\frac{\partial^2 l_n(\bgamma^*)}{\partial \bgamma \partial \bgamma^\top}\rb =\A = \ba{cccc}
  {\bf 0}&{\bf 0}&\A_{\bspsi\bu}\\
  {\bf 0}&\A_{\bstheta\bstheta}& \A_{\bstheta\bu}\\
  \A_{\bu\bspsi}&\A_{\bu\bstheta}&-\A_{\bu\bu}\\
  \ea.
\end{equation*}
 \item[(b)] Let $\S^*_n = (\S_{n\bstheta}^\top, \S_{n\bu}^\top)^\top$. Then as $n\to\infty$,
\begin{equation*}
  n^{-1/2}\S^*_n \to N({\bf 0}, \bGamma)
\end{equation*}
in distribution with 
\begin{eqnarray*}
 &&\e_{\bstheta} = \ba{c} 1\\{\bf 0}_{d\times 1} \ea,~~~
   \e_{\bu} = \ba{c} 1\\{\bf 0}_{r\times 1} \ea,~~~
   \C = \ba{c}
  \A_{\bstheta\bstheta}\e_{\bstheta}\\ 
  -\lambda^*(1-\lambda^*) \A_{\bu\bu}\e_{\bu}
  \ea, \\
 &&\text{and}~~\bGamma =\ba{cc}\A_{\bstheta\bstheta}&{\bf 0}\\ {\bf 0}& \A_{\bu\bu} \ea - \frac{1}{\lambda^*(1-\lambda^*)} \C\C^\top.\\
% &&\V = \ba{cc}\A_{\bstheta\bstheta}&{\bf 0}\\ {\bf 0}& \A_{\bu\bu} \ea,~
\end{eqnarray*}

 \end{enumerate}

%with 
%\begin{eqnarray*}
%\A_{\bstheta\bu} &=& \A_{\bu\bstheta}^\top = E_0 \lb \frac{\partial \G(\X;\etab^*)}{ \partial \btheta} \rb^\top ;\\\A_{\bstheta\bstheta} &=& (1 - \lambda^*) E_0 \lb h_1(X)\Q(X)\Q(X)^\top \rb ;\\\A_{\bsbeta\bu} &=& \A_{\bu\bsbeta}^\top = E_0 \lb \frac{\partial \G(\X;\etab^*)}{ \partial \bbeta} \rb^\top - E_0 \lb h_1(X) \Q(X) \G(\X;\etab^*)^\top\rb ;\\\A_{\bu\bu} &=& E_0 \lb \frac{\G(\X;\etab^*)\G(\X;\etab^*)^\top}{h(X)} \rb.
%\end{eqnarray*}
\end{lemma}

\begin{proof}
For (a): Note that Conditions C3 and C4 ensure that $\A$ is well defined. The results then follow by applying Lemma \ref{expectation} to each term of $E\lb\partial^2 l_n(\bgamma^*)/(\partial \bgamma \partial \bgamma^\top) \rb$. We use $E\lb\partial^2 l_n(\bgamma^*)/(\partial \btheta \partial \btheta^\top) \rb$ as an illustration;
for the other entries, the idea is similar and we omit the details. 

With Lemma \ref{expectation} and the fact that $h_0(x)h(x) = 1-\lambda^*$, we have 
\begin{eqnarray*}
-\frac{1}{n}E\lb\frac{\partial^2 l_n(\bgamma^*)}{\partial \btheta \partial \btheta^\top}\rb 
&=& \frac{1}{n} E\lb \Sumij h_0(X_{ij}) h_1(X_{ij})\Q(X_{ij})\Q(X_{ij})^\top\rb\\
&=& (1-\lambda^*)E_0\lb h_1(X)\Q(X)\Q(X)^\top \rb\\
&=& \A_{\bstheta\bstheta}.
\end{eqnarray*}

For (b): 
Conditions C2--C4 ensure that $E(\S^*_n)$  and $Var(\S^*_n)$ are well defined. 
We first use the results in Lemma \ref{expectation} to show that $E(\S^*_n) = {\bf 0}$.
For $E(\S_{n\bstheta})$,
\begin{eqnarray*}
E(\S_{n\bstheta}) &=& n_1 E\{\Q(X_{11})\} -n E_0 \{h(X)h_1(X)\Q(X)\}\\
&=& n_1 E_0\{\omega(X)\Q(X)\}-n E_0 \{\lambda^*\omega(X)\Q(X)\}\\
&=& {\bf 0}.
\end{eqnarray*}
The last step follows from the fact that $\lambda^*=n_1/n$. 

The unbiasedness of the EEs leads to 
\begin{eqnarray*}
E(\S_{n\bu}) &=& -nE_0\{\G(\X;\etab^*)\} = {\bf 0}.
\end{eqnarray*}
Hence, we have $E(\S^*_n) = {\bf 0}$.

Since $\S^*_n$ is a summation of independent random vectors, by the central limit theorem, 
 $$
 n^{-1/2}\S^*_n
 \to N({\bf 0}, \bGamma)
 $$ for some $\bGamma$. 
Next, we show that $\bGamma$ has the form claimed in the lemma. 

We start with the variances of $n^{-1/2}\S_{n\bstheta}$ and $n^{-1/2}\S_{n\bu}$. 
Note that
\begin{equation*}
  \S_{n\bstheta} = \Sumj h_0(X_{1j})\Q(X_{1j}) - \sum_{j=1}^{n_0}h_1(X_{0j})\Q(X_{0j}).
\end{equation*}
With the help of Lemma \ref{expectation}, we have 
\begin{eqnarray*}
Var( n^{-1/2}\S_{n\bstheta})
&=&\frac{1}{n} Var\ls \Sumj h_0(X_{1j})\Q(X_{1j}) - \sum_{j=1}^{n_0}h_1(X_{0j})\Q(X_{0j}) \rs\\
&=& \lambda^* E_0\lb h_0(X)^2\omega(X)\Q(X)\Q(X)^\top \rb \\
&&
+ (1-\lambda^*) E_0\lb h_1(X)^2\Q(X)\Q(X)^\top \rb\\
&&-\lambda^* E_0\lb h_0(X)\omega(X)\Q(X) \rb E_0\lb h_0(X)\omega(X)\Q(X)^\top \rb \\
&&- (1-\lambda^*) E_0\lb h_1(X)\Q(X) \rb E_0\lb h_1(X)\Q(X)^\top \rb. 
\end{eqnarray*}
Using the definitions of functions $h_1(\cdot)$ and $h_0(\cdot)$ and the fact that $\lambda^*=n_1/n$, 
we further have 
\begin{eqnarray*}
Var( n^{-1/2}\S_{n\bstheta})
&=& (1-\lambda^*) E_0 \lb h_1(X)\Q(X)\Q(X)^\top\rb\\
&& - \frac{1-\lambda^*}{\lambda^*} E_0\lb h_1(X)\Q(X) \rb E_0\lb h_1(X)\Q(X)^\top \rb\\
&=& \A_{\bstheta\bstheta} - \{\lambda^*(1-\lambda^*)\}^{-1}\A_{\bstheta\bstheta}\e_{\bstheta}\ls \A_{\bstheta\bstheta}\e_{\bstheta}\rs^\top.
\end{eqnarray*}

Similarly, we calculate the variance of $n^{-1/2}\S_{n\bu}$ as
\begin{eqnarray*}
Var(n^{-1/2}\S_{n\bu})&=&\frac{1}{n}Var\lb -\Sumij \frac{\G(X_{ij})}{h(X_{ij})} \rb \\
&=& \frac{1}{n}\Sumij E_0\lb\frac{\G(X_{ij})\G(X_{ij})^\top}{h(X_{ij})^2} \rb 
- \frac{1}{n}\sum_{j=1}^{n_0} E_0 \lb\frac{\G(X_{0j})}{h(X_{0j})} \rb E_0 \lb\frac{\G(X_{0j})^\top}{h(X_{0j})} \rb\\
&&-\frac{1}{n}\Sumj E_0 \lb\frac{\omega(X_{1j})\G(X_{1j})}{h(X_{1j})} \rb E_0 \lb\frac{\omega(X_{1j})\G(X_{1j})^\top}{h(X_{1j})} \rb \\
&=& \A_{\bu\bu}- (1-\lambda^*) E_0 \lb\frac{\G(X)}{h(X)} \rb E_0 \lb\frac{\G(X)^\top}{h(X)} \rb \\
&& -\lambda^* E_0 \lb\frac{\omega(X)\G(X)}{h(X)} \rb E_0 \lb\frac{\omega(X)\G(X)^\top}{h(X)} \rb.
\end{eqnarray*}

It can easily be verified that 
\begin{eqnarray*}
  && (1-\lambda^*) E_0 \lb\frac{\G(X)}{h(X)} \rb + \lambda^* E_0 \lb\frac{\omega(X)\G(X)}{h(X)} \rb = E_0 \lb \G(X) \rb = {\bf 0},
  \end{eqnarray*}
  which implies that 
 \begin{eqnarray*} 
  &&E_0\lb \frac{\{\omega(X) - 1\}\G(X) }{h(X)}\rb = -\frac{1}{\lambda^*}E_0\lb \frac{\G(X) }{h(X)}\rb=
  \A_{\bu\bu}\e_{\bu}. 
\end{eqnarray*}
Therefore, 
\begin{equation*}
  Var(n^{-1/2}\S_{n\bu}) = \A_{\bu\bu} - \lambda^*(1-\lambda^*) \A_{\bu\bu}\e_{\bu}(\A_{\bu\bu}\e_{\bu})^\top.
\end{equation*}
Lastly, we consider the covariance between $n^{-1/2}\S_{n\bstheta}$ and $n^{-1/2}\S_{n\bu}$:
\begin{eqnarray*}
&&Cov(n^{-1/2}\S_{n\bstheta}, n^{-1/2}\S_{n\bu})\\
 &=&-\frac{1}{n}Cov\ls \Sumj h_0(X_{1j})\Q(X_{1j}) - \sum_{j=1}^{n_0}h_1(X_{0j})\Q(X_{0j}), ~~\Sumij \frac{\G(X_{ij})^\top}{h(X_{ij})}\rs\\
 &=& -\frac{1}{n} \Sumj Cov \ls h_0(X_{1j})\Q(X_{1j}), \frac{\G(X_{1j})^\top}{h(X_{1j})} \rs
 +\frac{1}{n} \sum_{j=1}^{n_0} Cov \ls h_1(X_{0j};)\Q(X_{0j}),\frac{\G(X_{0j})^\top}{h(X_{0j})} \rs\\
 &=& \lambda^* E_0 \lb \omega(X)h_0(X)\Q(X)\rb E_0 \lb \frac{\omega(X)\G(X)^\top}{h(X)} \rb 
 - (1-\lambda^*) E_0 \lb h_1(X)\Q(X)\rb E_0 \lb \frac{\G(X)^\top}{h(X)} \rb\\
 &=& (1-\lambda^*) E_0 \lb h_1(X)\Q(X)\rb E_0 \lb \frac{\{\omega(X)-1\}\G(X)^\top}{h(X)} \rb\\
 &=& \A_{\bstheta\bstheta}\e_{\bstheta}(\A_{\bu\bu}\e_{\bu})^\top. 
\end{eqnarray*}
Then $ \bGamma=Var(n^{-1/2}\S^*_n)$ has the form claimed in the lemma. 
This completes the proof. 
\end{proof}

\section{Proof of Theorem 1}

Recall that $\hat\bgamma =(\hetab^\top,\hat{\bu}^\top)^\top$ is the MELE of $\bgamma$. Using an argument similar to that in \cite{qin1994empirical} and \cite{qin2014using}, we have that $\hetab = \etab^* +O_p(n^{-1/2})$ and $\hat{\bu} = \bu^*+O_p(n^{-1/2})$. 
To develop the asymptotic approximation of $\hetab$, we apply the first-order Taylor expansion to $\partial l_n(\hat\bgamma)/\partial \bgamma$ at the true value $\bgamma^*$. 
This, together with Condition C5,  gives 
\begin{equation*}
 {\bf 0}=  \S_n + \frac{\partial^2 l_n(\bgamma^*)}{\partial \bgamma \partial \bgamma^\top}(\hat\bgamma - \bgamma^*) + o_p(n^{1/2}).
\end{equation*}
With the law of large numbers and Lemma \ref{2nd_expect}, we have 
\begin{equation}
 \label{approx_expec_2nd}
  \frac{1}{n}\frac{\partial^2 l_n(\bgamma^*)}{\partial \bgamma \partial \bgamma^\top} = \frac{1}{n}E\left\{\frac{\partial^2 l_n(\bgamma^*)}{\partial \bgamma \partial \bgamma^\top}\right\} +o_p(1) = -\A +o_p(1) .
\end{equation}
Hence, we can write
\begin{eqnarray}
\label{taylor1}
&&\ba{cc}
  {\bf 0}&{\bf 0}\\
  {\bf 0}&\A_{\bstheta\bstheta}
  \ea (\hetab-\etab^*)  
 + \ba{c}\A_{\bspsi\bu}\\ \A_{\bstheta\bu} \ea (\hat{\bu} - \bu_0)
  = \frac{1}{n}\ba{c} {\bf 0}\\ \S_{n\bstheta}\ea + o_p(n^{-\frac{1}{2}});\\
  \label{taylor2}
  && \ba{cc}
  \A_{\bu\bspsi}&\A_{\bu\bstheta}\\
  \ea(\hetab-\etab^*) - \A_{\bu\bu}(\hat{\bu} - \bu_0)
   = \frac{1}{n}\S_{n\bu} + o_p(n^{-\frac{1}{2}}).
\end{eqnarray}
Recall that 
\begin{equation}
\label{def.uvj}
  \U = \ba{cc} {\bf 0}& \A_{\bspsi\bu}\\
  \A_{\bstheta\bstheta}&\A_{\bstheta\bu}\ea,~~\V = \ba{cc}\A_{\bstheta\bstheta}&{\bf 0}\\ {\bf 0}& \A_{\bu\bu} \ea,~~\text{and}~~\J = \U\V^{-1}\U^\top.
\end{equation}
Conditions C3 and C4 ensure that $\U$, $\V$, and $\J$ have full rank. 
Then \eqref{taylor1} and \eqref{taylor2} together imply that 
\begin{equation*}
%\label{eta_expres}
   n^{1/2} (\hetab - \etab^*) = \J^{-1}\U\V^{-1}(n^{-1/2}\S^*_n) + o_p(1).
\end{equation*}
Applying Lemma \ref{2nd_expect} and Slusky's theorem, we have as $n \to \infty$
\begin{equation*}
  n^{1/2} (\hetab - \etab^*) \to N({\bf 0 }, \bSigma)
\end{equation*}
in distribution with 
$\bSigma = \J^{-1}\U\V^{-1}Var(n^{-1/2}\S^*_n)\V^{-1}\U^\top\J^{-1}$.

Recall that
$$
Var(n^{-1/2}\S^*_n)= \bGamma=\V
- \frac{1}{\lambda^*(1-\lambda^*)} \C\C^\top~~\text{and}~~\C = \ba{c}
  \A_{\bstheta\bstheta}\e_{\bstheta}\\ 
  -\lambda^*(1-\lambda^*) \A_{\bu\bu}\e_{\bu}
  \ea.
$$
Since
\begin{equation*}
  \A_{\bspsi\bu}\e_{\bu} = {\bf 0}~~\text{and}~~ \A_{\bstheta\bu}\e_{\bu} = \frac{1}{\lambda^*}E_0\lb h_1(X)\Q(X)\rb = \frac{1}{\lambda^*(1-\lambda^*)}\A_{\bstheta\bstheta}\e_{\bstheta},
\end{equation*}
 we have
\begin{equation*}
  \U\V^{-1}\C = \U\V^{-1}\ba{c}
  \A_{\bstheta\bstheta}\e_{\bstheta}\\ 
  -\lambda^*(1-\lambda^*) \A_{\bu\bu}\e_{\bu}
  \ea
  =\ba{c} -\lambda^*(1-\lambda^*)\A_{\bspsi\bu}\e_{\bu}\\
 \A_{\bstheta\bstheta}\e_{\bstheta} - \lambda^*(1-\lambda^*) \A_{\bstheta\bu}\e_{\bu}\ea 
 = {\bf 0}.
\end{equation*}
This leads to $\bSigma = \J^{-1}$ and completes the proof. 

\section{Proof of Corollary 1}

\noindent \textit{Part (a).}
The results in Theorem 1 imply that 
$$
n^{1/2} (\hat\btheta - \btheta^*) \to 
N \ls {\bf 0}, \J_{\btheta}\rs
$$
in distribution, where 
\begin{equation*}
  \J_{\btheta} = \lb \A_{\bstheta\bstheta} + \A_{\bstheta\bu}\A_{\bu\bu}^{-1}\A_{\bu\bstheta} - \A_{\bstheta\bu}\A_{\bu\bu}^{-1}\A_{\bu\bspsi} \ls\A_{\bspsi\bu}\A_{\bu\bu}^{-1}\A_{\bu\bspsi}\rs^{-1} \A_{\bspsi\bu}\A_{\bu\bu}^{-1}\A_{\bu\bstheta}\rb^{-1}.
\end{equation*}  
From the definitions of $\A_{\bu\bspsi}$ and $\A_{\bu\bu}$, we have 
\begin{equation*}
  \A_{\bu\bspsi}=\ba{c} 0\\ E_0 \lb \frac{\partial \g(\X;\etab^*)}{\partial \bpsi}\rb \ea~~\text{and}~~
  \A_{\bu\bu} = \ba{cc} E_0\lb \frac{\{\omega(X) -1\}^2}{h(X)}\rb & E_0\lb\frac{\{\omega(X) -1\}\g(\X;\etab^*)}{h(X)} \rb\\
  E_0\lb \frac{\{\omega(X) -1\}\g(\X;\etab^*)^\top}{h(X)} \rb & E_0\lb \frac{\g(\X;\etab^*)\g(\X;\etab^*)^\top}{h(X)} \rb \ea.
\end{equation*}
We write
\begin{equation*}
  \A_{\bu\bu}^{-1} = \ba{cc}\A_{\bu\bu}^{11}&\A_{\bu\bu}^{12}\\
  \A_{\bu\bu}^{21}&\A_{\bu\bu}^{22}\ea.
\end{equation*}
When $r =p$, we have
\begin{eqnarray*}
  \ls \A_{\bspsi\bu}\A_{\bu\bu}^{-1}\A_{\bu\bspsi}\rs^{-1} &=& \Lm E_0 \lb \frac{\partial \g(\X;\etab^*)}{\partial \bpsi}\rb^\top \A_{\bu\bu}^{22} E_0 \lb \frac{\partial \g(\X;\etab^*)}{\partial \bpsi}\rb \Rm^{-1}\\
  &=&
   \Lm E_0 \lb \frac{\partial \g(\X;\etab^*)}{\partial \bpsi}\rb^\top\Rm^{-1} \ls\A_{\bu\bu}^{22}\rs^{-1} \Lm E_0 \lb \frac{\partial \g(\X;\etab^*)}{\partial \bpsi}\rb\Rm^{-1}.
\end{eqnarray*}
This leads to  
\begin{eqnarray*}
%&& \A_{\bu\bu}^{-1}\A_{\bu\bstheta} \ls\A_{\bstheta\bu}\A_{\bu\bu}^{-1}\A_{\bu\bstheta}\rs^{-1} \A_{\bstheta\bu}\A_{\bu\bu}^{-1}\\
&&\A_{\bu\bu}^{-1}\A_{\bu\bspsi} \ls\A_{\bspsi\bu}\A_{\bu\bu}^{-1}\A_{\bu\bspsi}\rs^{-1} \A_{\bspsi\bu}\A_{\bu\bu}^{-1} \\
&=& \ba{cc}\A_{\bu\bu}^{12}\ls\A_{\bu\bu}^{22}\rs^{-1}\A_{\bu\bu}^{21}&\A_{\bu\bu}^{12}\\
  \A_{\bu\bu}^{21}&\A_{\bu\bu}^{22}\ea\\
&=&\A_{\bu\bu}^{-1} - 
\ba{cc} \A_{\bu\bu}^{11}- \A_{\bu\bu}^{12}\ls\A_{\bu\bu}^{22}\rs^{-1}\A_{\bu\bu}^{21}&{\bf 0}\\ {\bf 0}&{\bf 0}\ea.
\end{eqnarray*}

It can be verified that $\A_{\bstheta\bu}\e_{\bu} = \{\lambda^*(1-\lambda^*)\}^{-1}\A_{\bstheta\bstheta}\e_{\bstheta}$ and 
\begin{eqnarray*}
% \A_{\bstheta\bu}\e_{\bu} &=& \frac{1}{\lambda^*(1-\lambda^*)}\A_{\bstheta\bstheta}\e_{\bstheta},\\
 \lb \A_{\bu\bu}^{11}- \A_{\bu\bu}^{12}\ls\A_{\bu\bu}^{22}\rs^{-1}\A_{\bu\bu}^{21}\rb^{-1} = E_0\lb \frac{\{\omega(X) -1\}^2}{h(X)}\rb 
 = \frac{1}{\lambda^*(1-\lambda^*)}\lb 1 - \frac{\e_{\bstheta}^\top \A_{\bstheta\bstheta}\e_{\bstheta}}{\lambda^*(1-\lambda^*)} \rb.
\end{eqnarray*}
By the Woodbury matrix identity, the variance matrix $\J_{\btheta}$ can be simplified as
\begin{eqnarray*}
  \J_{\btheta} &=& \lb \A_{\bstheta\bstheta} + \lb\frac{\A_{\bstheta\bstheta}\e_{\bstheta}}{\lambda^*(1-\lambda^*)} \rb  
  \Lm E_0\lb \frac{\{\omega(X) -1\}^2}{h(X)}\rb \Rm^{-1} \lb\frac{\A_{\bstheta\bstheta}\e_{\bstheta}}{\lambda^*(1-\lambda^*)} \rb^\top\rb^{-1}\\
  &=& \A_{\bstheta\bstheta}^{-1} - \frac{\e_{\bstheta}\e_{\bstheta}^\top}{\lambda^*(1-\lambda^*)}.
\end{eqnarray*}
This is the same as the asymptotic variance of $n^{1/2}(\Tilde{\btheta} - \btheta^*)$ shown in Lemma 1 of \cite{qin1997goodness} under Conditions C1--C3. \\

\noindent \textit{Part (b).}
For $r >p$, let $\U_m,\V_m,\J_m$ denote the corresponding $\U,\V,\J$ matrices obtained by using only the first $m$ EEs of $\g(\x;\etab)$. With Theorem 1, to complete the proof of this part it suffices to show that 
\begin{equation*}
  %\label{Jm and Jm-1}
  \J_m \geq \J_{m-1}.
\end{equation*}

From the definition of the matrix $\U$, we notice that $\U_m$ has one more column than $\U_{m-1}$, and we denote this extra column $u_m$. Then we have $\U_m = (\U_{m-1},u_m)$. Following the proof of Corollary 1 of \cite{qin1994empirical}, we have 
\begin{equation}
  \label{coro1_V_inequality} 
  \V_m^{-1} \geq \ba{cc}
\V_{m-1}^{-1}&{\bf 0}\\{\bf 0}&{ 0}
\ea.
\end{equation}
Therefore, 
\begin{equation}
\label{coro1_J_inequality}
  \J_m = \U_m \V_m^{-1} \U^\top_m \geq (\U_{m-1},u_m) \ba{cc}
\V_{m-1}^{-1}&{\bf 0} \\{\bf 0}&{0}
\ea(\U_{m-1},u_m)^\top = \J_{m-1},
\end{equation}
 as required. This completes the proof. 

\section{Proof of Theorem 2}
Recall that the null hypothesis forms a constraint 
\begin{equation*}
  \mathcal{C}_3 = \lb \etab: \H(\etab) = {\bf 0} \rb,
\end{equation*}
and 
the ELR statistic for testing $H_0: \H(\etab) = 0$ is defined as 
\begin{eqnarray*}
  R_n& =& 2\lb \sup_{\bspsi,\bstheta} \ell_n(\bpsi,\btheta) - \sup \limits_{\etab \in \mathcal{C}_3} \ell_n(\bpsi,\btheta) \rb\\
  &=&2\lb\ell_n(\hat\bpsi,\hat\btheta)-\ell_n(\check\bpsi,\check\btheta)\rb,
\end{eqnarray*}
where 
$$
(\check\bpsi,\check\btheta)=\arg\max_{\etab\in \mathcal{C}_3} \ell_n(\bpsi,\btheta) . 
$$
In the following steps, we find the approximations of $\ell_n(\hat\bpsi,\hat\btheta)$ and $\ell_n(\check\bpsi,\check\btheta)$. 

We first derive the approximation of $l_n(\bgamma)$ when $\bgamma$ is in the $n^{-1/2}$ neighborhood of its true value $\bgamma^*$. Applying the second-order Taylor expansion to $l_n(\bgamma)$,  and  using  \eqref{approx_expec_2nd} and Condition C5, we have 
\begin{eqnarray*}
%\label{lr_taylor}
 l_n(\bgamma)
 &=& l_n(\bgamma^*) + \S_{n}^\top (\bgamma - \bgamma^*) - \frac{n}{2}(\bgamma - \bgamma^*)^\top \A (\bgamma - \bgamma^*) + o_p(1)\\
 &=& l_n(\bgamma^*) + \ba{cc} {\bf 0}& \S^\top_{n\btheta}\ea (\etab - \etab^*) + \S_{n\bu}^\top(\bu - \bu_0) -\frac{n}{2}(\etab - \etab^*)^\top \ba{cc}
  {\bf 0}&{\bf 0}\\
  {\bf 0}&\A_{\bstheta\bstheta}
  \ea (\etab - \etab^*)\\
  &&- n (\etab - \etab^*)^\top \ba{c}\A_{\bspsi\bu}\\ \A_{\bstheta\bu}\ea (\bu - \bu^*) + \frac{n}{2}(\bu - \bu^*)^\top\A_{\bu\bu}(\bu - \bu^*)+o_p(1).
\end{eqnarray*}
Setting the derivative of $l_n(\bgamma)$ with respect to $\bu$ equal to zero gives 
\begin{equation*}
  \bu - \bu^* =  \A_{\bu\bu}^{-1} \ba{cc} \A_{\bu\bspsi}&\A_{\bu\bstheta}\\
  \ea(\etab-\etab^*) - \A_{\bu\bu}^{-1} \ls\frac{1}{n}\S_{n\bu}\rs + o_p(n^{-\frac{1}{2}}).
\end{equation*}
Substituting the approximation of $ \bu - \bu^*$ into $ l_n(\bgamma)$ leads
to an approximation of $ \ell_n(\bpsi,\btheta) $: 
\begin{equation}
\label{approx_profile_eta}
  \ell_n(\bpsi,\btheta)= l_n(\bgamma^*) + (\etab-\etab^*)^\top \U\V^{-1}\S^*_{n} - \frac{n}{2}(\etab-\etab^*)^\top\J(\etab-\etab^*) - \frac{1}{2n}\S_{n\bu}^\top\A_{\bu\bu}^{-1}\S_{n\bu} +o_p(1).
\end{equation}
With the approximation of $\hetab$ in \eqref{approx_profile_eta}, we then have 
\begin{eqnarray*}
\ell_n(\hat\bpsi,\hat\btheta)&=& l_n(\bgamma^*) + \frac{1}{2n}\S^{*\top}_n\V^{-1}\U^\top \J^{-1}\U\V^{-1}\S^*_n - \frac{1}{2n}\S_{n\bu}^\top\A_{\bu\bu}^{-1}\S_{n\bu} + o_p(1).
\end{eqnarray*}

Next, we find an approximation for $\check\etab=(\check\bpsi^\top, \check\btheta^\top)^\top$. 
We first define 
\begin{equation*}
%\label{ell_const}
 \ell_n^* (\bpsi, \btheta,\v) = \ell_n(\bpsi, \btheta) + n \v^\top \H(\etab), 
\end{equation*}
where $\v$ is the Lagrange multiplier. 
Then $\check\etab$ and the corresponding Lagrange multiplier $\check\v$ satisfy 
\begin{equation}
\label{elln.star.equation}
\frac{\partial \ell_n^* (\check\bpsi,\check\btheta,\check\v) }{\partial \bpsi}
={\bf 0},~~
\frac{\partial \ell_n^* (\check\bpsi,\check\btheta,\check\v) }{\partial \btheta}
={\bf 0},~~
\frac{\partial \ell_n^* (\check\bpsi,\check\btheta,\check\v) }{\partial \v}
={\bf 0}.
\end{equation}
It is easy to verify that $ \check{\bgamma} = \bgamma^* +O_p(n^{-1/2})$ and $\check{\v}= O_p(n^{-1/2})$ \citep{qin1995estimating,qin2014using}.

Let $\h^* = \partial \H(\etab^*)/\partial \etab$.
When $\etab$ is in the $n^{-1/2}$ neighborhood of the true value $\etab^*$, we approximate $\H(\etab)$
with $\H(\etab) = \h^*(\etab-\etab^*) + o_p(n^{-1/2})$. 
Together with the approximation of $ \ell_n(\bpsi,\btheta)$ in \eqref{approx_profile_eta}, we approximate $ \ell_n^* (\bpsi, \btheta,\v) $ at an $n^{-1/2}$ neighbor of $(\bpsi_0^\top,\btheta_0^\top,{\bf 0}_{1\times q})^\top$ with 
\begin{eqnarray*}
 \ell_n^* (\bpsi, \btheta,\v) &=& l_n(\bgamma^*) + (\etab-\etab^*)^\top \U\V^{-1}\S^*_{n} - \frac{n}{2}(\etab-\etab^*)^\top\J(\etab-\etab^*) \nonumber \\
  \label{approx_check_liklihood}
  && + n\v^\top \h^* (\etab- \etab^*) - \frac{1}{2n}\S_{n\bu}^\top\A_{\bu\bu}^{-1}\S_{n\bu} +o_p(1).
\end{eqnarray*}
Applying the first-order Taylor expansion to (\ref{elln.star.equation}), 
we have 
\begin{equation*}
%\label{1st_condition}
\ba{cc}
 \J& - {\h^*}^\top\\ - \h^*& {\bf 0}
\ea \ba{c}
\check{\etab} - \etab^*\\\check{\v}
\ea = \frac{1}{n}\ba{c}
\U\V^{-1}\S^*_n\\ {\bf 0}
\ea + o_p(n^{-\frac{1}{2}}).
\end{equation*}
Hence, 
\begin{eqnarray}
n^{1/2}(\check{\etab} - \etab^*) &=& 
(\I,{\bf 0})\ba{cc}
 \J& - {\h^*}^\top\\ - {\h^*}^\top&{\bf 0}
\ea^{-1}\ba{c}
 n^{-1/2}\U\V^{-1}\S_n\\{\bf 0}
\ea + o_p(1) \nonumber \\ 
\label{eta_hat_star}
&=&\{ \J^{-1} -\J^{-1}{\h^*}^\top( \h^* \J^{-1} {\h^*}^\top)^{-1} \h^* \J^{-1}\} \U\V^{-1}(n^{-1/2}\S^*_n) + o_p(1),
\end{eqnarray}
where $\I$ is the identity matrix with dimension $p+d+1$.

Substituting the expression of $\check{\etab}$ in \eqref{eta_hat_star} into \eqref{approx_profile_eta} gives  
\begin{eqnarray*}
\ell_n(\check\bpsi,\check\btheta) &=& l_n(\bgamma^*) + \frac{1}{2n}\S^{*\top}_n\V^{-1}\U^\top \{\J^{-1} - \J^{-1}{\h^*}^\top( \h^* \J^{-1} {\h^*}^\top)^{-1} \h^* \J^{-1}\} \U\V^{-1}\S^*_n \\
&&- \frac{1}{2n}\S_{n\bu}^\top\A_{\bu\bu}^{-1}\S_{n\bu} + o_p(1).
\end{eqnarray*}
Hence, the ELR statistic $R_n$ can be written as
\begin{eqnarray*}
 R_n &=& \frac{1}{n} \S^{*\top}_n\V^{-1}\U^\top \J^{-1}{\h^*}^\top( \h^* \J^{-1} {\h^*}^\top)^{-1} \h^* \J^{-1}\U\V^{-1}\S^*_n + o_p(1).
\end{eqnarray*}
We find that 
$\J^{-1/2}{\h^*}^\top( \h^* \J^{-1} {\h^*}^\top)^{-1} \h^* \J^{-1/2}$ is an idempotent matrix with rank $q$. 
Further, as $n\to\infty$, 
$$\J^{-1/2}\U\V^{-1}(n^{-1/2}\S^*_n) \to N(0,\I)$$ in distribution. Therefore, the limiting distribution of $R_n$ is $\chi^2_q$ under $H_0$.

\section{Proofs of Theorem 3 and Corollary 2}
We start with the proof of Theorem 3. 
Recall that 
the ELR statistic for testing the validity of the EEs
is defined as 
\begin{equation*}
  W_n 
  = 2\lb \ell_{nd}(\Tilde{\btheta}) - \ell_n(\hat\bpsi, \hat\btheta) \rb.
\end{equation*}

We first find an approximation of $\ell_{nd}(\Tilde{\btheta})$. 
Applying the second-order Taylor expansion to $\ell_{nd}(\Tilde{\btheta})$ at the true value $\btheta^*$, we have 
\begin{equation*}
  \ell_{nd}(\Tilde{\btheta}) = \ell_{nd}(\btheta^*) + (\Tilde{\btheta} - \btheta^*)^\top\frac{ \partial \ell_{nd}(\btheta^*)}{\partial \btheta}+
  \frac{1}{2}(\Tilde{\btheta} - \btheta^*)^\top\frac{ \partial^2 \ell_{nd}(\btheta^*)}{\partial \btheta\partial \btheta^\top}(\Tilde{\btheta} - \btheta^*) +
  o_p(1).
\end{equation*}
The fact that $\bnu^* = {\bf 0}$ implies $\ell_{nd}(\btheta^*) = l_n(\bgamma^*)$. According to \cite{qin1997goodness}, it is easy to verify that  
\begin{equation*}
  \Tilde{\btheta} - \btheta^* = \frac{1}{n} \A_{\bstheta\bstheta}^{-1}\frac{ \partial \ell_{nd}(\btheta^*)}{\partial \btheta}+o_p(n^{-1/2}), ~~
  \frac{ \partial \ell_{nd}(\btheta^*)}{\partial \btheta} = \S_{n\bstheta}, 
  ~~{\rm and}~~
 \frac{1}{n}
 \frac{ \partial^2 \ell_{nd}(\btheta^*)}{\partial \btheta\partial \btheta^\top}
 =-\A_{\btheta\btheta}+o_p(1) 
  .
\end{equation*}
Then 
\begin{equation*}
  \ell_{nd}(\Tilde{\btheta}) = l_n(\bgamma^*) + \frac{1}{2n}\S_{n\btheta}^\top \A_{\bstheta\bstheta}^{-1}\S_{n\btheta}
  +o_p(1).
\end{equation*}

Hence, the ELR statistic can be written as 
\begin{eqnarray}
\nonumber
  W_n &=& 2\lb \ell_{nd}(\Tilde{\btheta}) - \ell_n(\hat\bpsi, \hat\btheta) \rb\\
\nonumber  &=& \frac{1}{n}\S_{n\btheta}^\top \A_{\bstheta\bstheta}^{-1}\S_{n\btheta} + \frac{1}{n}\S_{n\bu}^\top\A_{\bu\bu}^{-1}\S_{n\bu}-\frac{1}{n}\S^{*\top}_n\V^{-1}\U^\top \J^{-1}\U\V^{-1}\S^*_n  \\
  &=& \frac{1}{n}\S_n^{*\top} \V^{-1} ( \V - \U^\top\J^{-1}\U ) \V^{-1}\S^*_n+ o_p(1). 
  \label{wn.appro}
\end{eqnarray}

Since $\V$ is a positive-definite matrix, we define an inner product on the vector space $\mathbb{R}^{2+d+r}$ as $<\boldsymbol{a},\boldsymbol{b}>_{\V^{-1}} = \boldsymbol{a} ^\top \V^{-1} \boldsymbol{b}$ for any vector $\boldsymbol{a},\boldsymbol{b}$ in the vector space. Recall that  
\begin{equation*}
  \C = \ba{c}
  \A_{\bstheta\bstheta}\e_{\bstheta}\\ 
  -\lambda^*(1-\lambda^*) \A_{\bu\bu}\e_{\bu}
  \ea.
\end{equation*}
The vector $\C$ and each row in $\U$ are linearly independent in the inner product space because $\U\V^{-1}\C = {\bf 0}$. 
Let $\mathcal{V}$ be the inner product space spanned by the vector $\C$ and each row in $\U$. Then there exists an orthogonal complement $\mathcal{B}$ of the subspace $\mathcal{V}$ with the dimension $r-p$. Let the columns of $\C^{*}$ be the basis of the orthogonal complement $\mathcal{B}$. Then $\C^{*}$ satisfies $\C^{* \top}\V^{-1}(\C, \U^\top) = {\bf 0}$. Define $\M^\top = (\C^{*}, \C, \U^\top)$, 
which satisfies 
\begin{equation*}
  \M \V^{-1}\M^\top = \ba{ccc} 
  \C^{* \top} \V^{-1} \C^{*} & {\bf 0}& {\bf 0}\\ 
{\bf 0} & \C^\top\V^{-1}\C &{\bf 0} \\
{\bf 0}& {\bf 0}& \J \ea.
\end{equation*}

With the above construction, $\M$ is a full rank matrix and can be inverted. We can write the inverse of $\M \V^{-1} \M^\top$ as 
\begin{equation*}
(\M^\top)^{-1}\V\M^{-1} = \ba{ccc} 
(\C^{* \top}\V^{-1}\C^{*})^{-1} & {\bf 0}&{\bf 0}\\ 
{\bf 0} &(\C^\top\V^{-1}\C)^{-1} &{\bf 0}\\
{\bf 0}& {\bf 0}&\J^{-1} \ea.
\end{equation*}
Then 
\begin{eqnarray*}
\V &=& \M^\top (\M^\top)^{-1}\V\M^{-1} \M \\
&=&\C^{*} (\C^{* \top}\V^{-1}\C^{*})^{-1}\C^{* \top} +\C (\C^\top\V^{-1}\C)^{-1}\C^\top + \U^\top\J^{-1}\U.
\end{eqnarray*}

Note that  
\begin{eqnarray*}
 \C^\top\V^{-1}\S^*_n &=& \e_{\bstheta}^\top\S_{n\bstheta} - \lambda^*(1-\lambda^*)\e_{\bu}^\top\S_{n\bu}\\
 &=& n_1 -\Sumij h_1(X_{ij}) + \lambda^*(1-\lambda^*) \Sumij \frac{\omega(X_{ij}) -1}{h(X_{ij})}\\
 &=& 0.
\end{eqnarray*}
This helps to simplify $W_n$ as 
\begin{equation*}
 W_n = \frac{1}{n}\S_n^{*\top} \V^{-1}\C^{*} (\C^{* \top}\V^{-1}\C^{*})^{-1}\C^{* \top} \V^{-1}\S^*_n + o_p(1).
\end{equation*}

According to Lemma \ref{2nd_expect}, we have 
\begin{equation*}
  Var\ls n^{-1/2}\S^*_n\rs = \V - \frac{1}{\lambda^*(1-\lambda^*)}\C\C^\top. 
\end{equation*}
Together with $\C^{*\top}\V^{-1}\C = {\bf 0}$ and the fact that 
$\V^{-1/2}\{\C^* (\C^{*\top}\V^{-1}\C^{*})^{-1}\C^{*\top} \}\V^{-1/2}$ is idempotent with rank $r-p$, we have 
\begin{eqnarray*}
  && \{\V^{-1/2} \C^{*} (\C^{* \top}\V^{-1}\C^{*})^{-1}\C^{* \top} \V^{-1}\} Var\ls n^{-1/2}\S^*_n\rs \{\V^{-1} \C^{*} (\C^{* \top}\V^{-1}\C^{*})^{-1}\C^{* \top} \V^{-1/2}\}\\
  &=&\V^{-1/2} \C^{*} (\C^{* \top}\V^{-1/}\C^{*})^{-1}\C^{* \top} \V^{-1/2}.
\end{eqnarray*}
Therefore, $W_n$ asymptotically follows $\chi^2_{r-p}$ under $H_0$ as $n\to\infty$.

We now prove Corollary 2.  
Let $\S^{*}_{n1}$ be the first $ d+r-m+2$ elements of $\S^{*\top}_{n}$, 
$\U_1$ be the first $r-m$ columns of $\U$, 
$\V_1$ be the upper $(d+r-m+2)\times (d+r-m+2)$ matrix of $\V$, 
and $\J_1=\U_1\V_1^{-1}\U_1^\top$. 
Further, let $ \ell_{n1}(\bpsi, \btheta) $
be the profile empirical log-likelihood of $(\bpsi, \btheta)$
using only $\g_1(\x;\etab)$
and  
$$
(\hat\bpsi^*, \hat\btheta^*)
=\arg\max_{\bspsi, \bstheta} \ell_{n1}(\bpsi, \btheta) . 
$$

%We now consider the property of the ELR statistic $LR_{n1}$ when the information coming from $\g_2(\x;\etab)$ is correctly specified.
%By replacing all the $\g(\x;\etab)$ with $\g_1(\x;\etab)$ in related definitions, we similarly define the PEL of $\etab$ in this case as
%\begin{equation*}
%   \ell_{n1}(\etab) = l_n(\bgamma^*) + (\etab-\etab^*)^\top \U_1\V_1^{-1}\S^*_{n1} - \frac{n}{2}(\etab-\etab^*)^\top\J_1(\etab-\etab^*) - \frac{1}{2n}\S_{n\bu_1}^\top\A_{\bu_1\bu_1}^{-1}\S_{n\bu_1} +o_p(1).
%\end{equation*}
%%where we adjust the matrices $\U$, $\V$,$\J$ and vectors $\S_{n} = (\S_{n\bstheta}^\top, \S_{n\bu}^\top)^\top$ according to $\g_1(\x;\etab)$ and denote them as $\U_1$, $\V_1$,$\J_1$ and vectors $\S_{n_1} = (\S_{n\bstheta}^\top, \S_{n\bu_1}^\top)^\top$. 
%Denote $\hetab^*$ as the MELE of the parameter $\etab$. 
Following the techniques used to obtain \eqref{wn.appro}, 
we have 
\begin{eqnarray*}
2\lb \ell_{nd}(\Tilde{\btheta}) - \ell_{n1}(\hat\bpsi^*, \hat\btheta^*) \rb
  =\S^{*\top}_{n1}\V_1^{-1}(\V_1 - \U_1^\top \J_1^{-1}\U_1)\V_1^{-1}\S^*_{n1}  + o_p(1).
\end{eqnarray*}
Then, the ELR statistic $W_n^*$ has the following approximation: 
\begin{eqnarray*}
 W_{n}^*& =&2\{ \ell_{n1}(\hat\bpsi^*, \hat\btheta^*)-\ell_n(\hat\bpsi, \hat\btheta) \}\\
 &=&2\lb \ell_{nd}(\Tilde{\btheta}) - \ell_n(\hat\bpsi, \hat\btheta) \rb -2\lb \ell_{nd}(\Tilde{\btheta}) - \ell_{n1}(\hat\btheta^*,\hat\bbeta^*) \rb\\
 &=& \frac{1}{n} \Lm \S^{*\top}_{n}\V^{-1}(\V - \U^\top \J^{-1}\U)\V^{-1}\S^*_{n} - \S^{*\top}_{n1}\V_1^{-1}( \V_1 - \U_1^\top \J_1^{-1}\U_1)\V_1^{-1}\S^*_{n1}\Rm +o_p(1).
\end{eqnarray*}
With the technique used to prove Corollary 1, we have 
\begin{equation*}
  \V^{-1}(\V - \U^\top \J^{-1}\U)\V^{-1} \geq 
  \ba{cc} \V_1^{-1}\{\V_1 - \U_1^\top \J_1^{-1}\U_1\}\V_1^{-1} &{\bf 0}\\{\bf 0} &{\bf 0}\ea.
\end{equation*}
Then 
$$
 \frac{1}{n} \Lm \S^{*\top}_{n}\V^{-1}(\V - \U^\top \J^{-1}\U)\V^{-1}\S^*_{n} - \S^{*\top}_{n1}\V_1^{-1}( \V_1 - \U_1^\top \J_1^{-1}\U_1)\V_1^{-1}\S^*_{n1}\Rm
 \geq 0.
$$

Recall that as $n\to\infty$, 
$$
\frac{1}{n} \S^{*\top}_{n}\V^{-1}(\V - \U^\top \J^{-1}\U)\V^{-1}\S^*_{n} 
\to \chi_{r-p}^2
$$
in distribution. 
We can similarly prove that as $n\to\infty$, 
$$
\frac{1}{n}\S^{*\top}_{n1}\V_1^{-1}( \V_1 - \U_1^\top \J_1^{-1}\U_1)\V_1^{-1}\S^*_{n1}
\to \chi_{r-m-p}^2
$$
in distribution. 

By the arguments in \cite{qin1994empirical}, 
we conclude that $W_{n}^* \to \chi^2_{(r-p)-(r-m-p)} = \chi^2_m$ in distribution as $n \to \infty$. 

\section{Proof of Theorem 4}
For (a): 
We start with some preparation. 
For any $x$ in the support of $F_0$, let
\begin{eqnarray*}
F_0(x,\bgamma) = \frac{1}{n}\Sumij \frac{I(X_{ij} \leq x)}{1+\lambda \lb \omega(X_{ij};\btheta)-1\rb +\bnu^\top\g(X_{ij};\bpsi, \btheta)},\\
F_1(x,\bgamma) = \frac{1}{n}\Sumij \frac{\omega(X_{ij};\btheta) I(X_{ij} \leq x)}{1+\lambda \lb \omega(X_{ij};\btheta)-1\rb +\bnu^\top\g(X_{ij};\bpsi, \btheta)}.
\end{eqnarray*}
Then
\begin{eqnarray*}
\hat F_0(x) = F_0(x,\hat\bgamma),~~~ F_0(x,\bgamma^*) = \frac{1}{n}\Sumij \frac{I(X_{ij} \leq x)}{h(X_{ij})},\\
\hat F_1(x) = F_1(x,\hat\bgamma),~~~ F_1(x,\bgamma^*) = \frac{1}{n}\Sumij \frac{\omega(X_{ij})I(X_{ij} \leq x)}{h(X_{ij})}.
\end{eqnarray*}

Next, we explore the properties of the first derivatives of $F_0(x,\bgamma)$ and $F_1(x,\bgamma)$ at the true value $\bgamma^*$. Define
\begin{equation*}
  \frac{\partial F_0(x,\bgamma^*)}{\partial \bgamma} = \ba{c}\frac{\partial F_0(x,\bgamma^*)}{\partial \bpsi}\\ \frac{\partial F_0(x,\bgamma^*)}{\partial \btheta}\\ \frac{\partial F_0(x,\bgamma^*)}{\partial \bu}\ea,~~~
\frac{\partial F_1(x,\bgamma^*)}{\partial \bgamma} = \ba{c}\frac{\partial F_1(x,\bgamma^*)}{\partial \bpsi} \\ \frac{\partial F_1(x,\bgamma^*)}{\partial \btheta}\\ \frac{\partial F_1(x,\bgamma^*)}{\partial \bu}\ea, 
\end{equation*}
where 
\begin{eqnarray*}
&&\frac{\partial F_0(x,\bgamma^*)}{\partial \bpsi} = \frac{\partial F_1(x,\bgamma^*)}{\partial \bpsi} = {\bf 0},\\
&&\frac{\partial F_0(x,\bgamma^*)}{\partial \btheta} = -\frac{1}{n}\Sumij h_1(X_{ij})h(X_{ij})\Q(X_{ij})
I(X_{ij} \leq x),\\
&&\frac{\partial F_0(x,\bgamma^*)}{\partial \bu} =-\frac{1}{n}\Sumij \frac{\G(X_{ij})}{\{h(X_{ij})\}^2}I(X_{ij} \leq x),\\
&&\frac{\partial F_1(x,\bgamma^*)}{\partial \btheta}= \frac{1}{n}\Sumij \frac{\omega(X_{ij})}{h(X_{ij})}h_0(X_{ij})\Q(X_{ij})
 I(X_{ij} \leq x),\\
&&\frac{\partial F_1(x,\bgamma^*)}{\partial \bu} = -\frac{1}{n}\Sumij \frac{\omega(X_{ij})}{\{h(X_{ij})\}^2}\G(X_{ij})
 I(X_{ij} \leq x).
\end{eqnarray*}
Applying Lemma \ref{expectation}, we have the following results for $E\left\{ \frac{\partial F_0(x,\bgamma^*)}{\partial \bgamma}\right\} $ and $E\left\{ \frac{\partial F_1(x,\bgamma^*)}{\partial \bgamma}\right\}$.
\begin{lemma}
\label{lemma_cdf}
With the form of $\partial F_0(x,\bgamma^*)/\partial \bgamma$ and $\partial F_1(x,\bgamma^*)/\partial \bgamma$ defined above, we have 
\begin{eqnarray*}
-E\left\{ \frac{\partial F_0(x,\bgamma^*)}{\partial \bgamma}\right\} = \B_0(x) = \ba{c}{\bf 0}\\ \B_{0\btheta}(x)\\\B_{0\bu}(x) \ea = \ba{c}{\bf 0}\\ \B_{0}^*(x) \ea,\\
-E\left\{ \frac{\partial F_1(x,\bgamma^*)}{\partial \bgamma}\right\} = \B_1(x) = \ba{c}{\bf 0}\\ \B_{1\btheta}(x)\\\B_{1\bu}(x) \ea = \ba{c}{\bf 0}\\ \B_{1}^*(x) \ea,
\end{eqnarray*}
where 
\begin{eqnarray*}
 &&\B_{0\btheta}(x) =E_0\lb h_1(X)\Q(X) I(X \leq x) \rb ,~~
 \B_{0\bu}(x) = E_0\lb \frac{\G(X)}{h(X)}I(X\leq x) \rb, \\
 &&\B_{1\btheta}(x) = \frac{\lambda^*-1}{\lambda^*}E_0\lb h_1(X)\Q(X) I(X \leq x) \rb,~~
 \B_{1\bu}(x) = E_0\lb \frac{\omega(X)\G(X)}{h(X)}I(X\leq x) \rb.
\end{eqnarray*}
\end{lemma}

We now move to the joint asymptotic normality of $\hat{F}_l(x)$ and $\hat{F}_s(y)$. 
We first find an approximation for $\hat{F}_l(x)$ for $l=0$ and 1. 
Applying the first-order Taylor expansion to $\hat{F}_l(x)$ and using the results in Lemma \ref{lemma_cdf}, we have 
\begin{eqnarray*}
\hat{F}_l(x) &=& F_l(x,\bgamma^*) - \B^*_l(x)^\top(\hat\bgamma^* -\bgamma^*) + o_p(n^{-1/2})\\
&=& F_l(x,\bgamma^*) - ({\bf 0}, \B_{l\bstheta}(x)^\top)(\hat\etab^* -\etab^*) - \B_{0\bu}(x)^\top(\hat\bu - \bu^*)+o_p(n^{-1/2}).
\end{eqnarray*}
Using the relationship in \eqref{taylor2} and the definitions of the matrices $\U$ and $\V$ in \eqref{def.uvj}, we have 
\begin{eqnarray*}
\hat{F}_l(x) &=& F_l(x,\bgamma^*) - \B^*_l(x)^\top\V^{-1}\U^\top(\hat\etab^* -\etab^*) + \frac{1}{n}\B_{l\bu}(x)^\top\A_{\bu\bu}^{-1}\S_{n\bu}+o_p(n^{-1/2})\\
&=&F_l(x,\bgamma^*) - \B^*_l(x)^\top \lb \V^{-1}\U^\top(\hat\etab^* -\etab^*) - \ba{cc} {\bf 0} & {\bf 0} \\ {\bf 0} &\A_{\bu\bu}^{-1} \ea \ls \frac{1}{n}\S_{n}^*\rs \rb+o_p(n^{-1/2}).
\end{eqnarray*}
Recall that $\hetab - \etab^* = \J^{-1}\U\V^{-1}(n^{-1}\S^*_n) + o_p(n^{-1/2})$. The approximation of $\hat{F}_l(x)$ is then given by 
\begin{equation*}
\hat{F}_l(x) = F_l(x,\bgamma^*) - \frac{1}{n}\B^*_l(x)^\top\W \S^*_n  + o_p(n^{-1/2})
\end{equation*}
with 
\begin{equation*}
  \W = \V^{-1} \U^\top \J^{-1}\U\V^{-1} - \ba{cc} {\bf 0}&{\bf 0}\\ {\bf 0} & \A_{\bu\bu}^{-1}\ea.
\end{equation*}

Note that $F_l(x) = E_0\{F_l(x,\bgamma^*)\}$. Then
\begin{equation*}
 n^{1/2}\{\hat{F}_l(x) - F_l(x) \}= n^{1/2} \{F_l(x,\bgamma^*) - F_l(x)\} - n^{-1/2}\B^*_l(x)^\top\W \S^*_n + o_p(1).
\end{equation*}
The two leading terms are summations of independent random variables and both have mean zero. 
Hence, 
as $n\to\infty$, 
$$
\sqrt{n}
\left(
\begin{array}{c}
\hat{F}_l(x) - F_l(x)\\
\hat{F}_s(y) - F_s(y)\\
\end{array}
\right)
\to 
N\Big({\bf 0}, \bSigma_{ls}(x,y) \Big), 
$$ 
where 
\begin{equation*}
  \bSigma_{ls}(x,y) = \ba{cc} \sigma_{ll}(x,x) &\sigma_{ls}(x,y) \\ \sigma_{sl}(y,x) &\sigma_{ss}(y,y) \ea.
\end{equation*}

To complete the proof of (a), we need to argue that $ \bSigma_{ls}(x,y) $ has the form claimed in the lemma. 
According to the expression of $\hat{F}_l(x) - F_l(x)$, we have 
\begin{eqnarray*}
  \sigma_{ll}(x,x) &=& n Var\lb F_l(x,\bgamma^*) \rb + n^{-1}Var(\B^*_l(x)^\top\W\S^*_n)\\ 
  &&- 2 Cov\lb F_l(x,\bgamma^*),\B^*_l(x)^\top\W\S^*_n\rb;\\
  \sigma_{ss}(y,y) &=& n Var\lb F_s(y,\bgamma^*)\rb + n^{-1}Var(\B^*_s(y)^\top\W\S^*_n)\\ 
  &&-2 Cov\lb F_s(y,\bgamma^*),\B^*_s(y)^\top\W\S^*_n\rb;\\
  \sigma_{ls}(x,y) &=& n Cov \lb F_l(x,\bgamma^*), F_s(y,\bgamma^*)\rb - Cov\lb F_l(x,\bgamma^*),\B^*_s(y)^\top\W\S^*_n\rb \\
&&- Cov \lb F_s(y,\bgamma^*),\B^*_l(x)^\top\W\S^*_n\rb
+\B^*_l(x)^\top \{n^{-1}Var(\W\S^*_n)\}\B^*_s(y);\\
 \sigma_{sl}(y,x) &=& \sigma_{ls}(x,y).
\end{eqnarray*}

Next, we calculate the covariances and variances appearing above. We start with the covariance 
and variance related to $F_l(x,\bgamma^*)$ and $F_s(y,\bgamma^*)$.  
Let $x \wedge y = \min\{x,y\}$. 
Using Lemma \ref{expectation}, we have 
\begin{eqnarray*}
 && nCov\lb F_0(x,\bgamma^*),~F_0(y,\bgamma^*) \rb \\
 &=& (1-\lambda^*)Cov \lb \frac{I(X_{01} \leq x)}{h(X_{01})},~ \frac{I(X_{01} \leq y)}{h(X_{01})} \rb +\lambda^*Cov \lb \frac{I(X_{11} \leq x)}{h(X_{11})},~ \frac{I(X_{11} \leq y)}{h(X_{11})} \rb\\
 &=& E_0\lb\frac{I(X\leq x\wedge y)}{h(X)} \rb -(1-\lambda^*) E_0 \lb \frac{I(X\leq x)}{h(X)} \rb E_0 \lb \frac{I(X\leq y)}{h(X)} \rb \\
  &&- \lambda^* E_0 \lb\frac{\omega(X)I(X\leq x)}{h(X)} \rb E_0 \lb \frac{\omega(X)I(X\leq y)}{h(X)}\rb. 
\end{eqnarray*}

After some algebra, we have that for any $x$ in the support of $F_0$,
\begin{eqnarray*}
 &&\B_{0\bu}(x)^\top\e_{\bu} = E_0 \lb\frac{\omega(X)I(X\leq x)}{h(X)} \rb - E_0 \lb\frac{I(X\leq x)}{h(X)} \rb,\\
 &&F_0(x) = E_0 \lb\frac{I(X\leq x)}{h(X)} \rb + \lambda^*\B_{0\bu}(x)^\top\e_{\bu}.
\end{eqnarray*}
Then 
the covariance $nCov\lb F_0(x,\bgamma^*),~F_0(y,\bgamma^*) \rb$ is simplified as 
\begin{eqnarray*}
&& nCov\lb F_0(x,\bgamma^*),~F_0(y,\bgamma^*) \rb \\
&=& E_0\lb\frac{I(X\leq x\wedge y)}{h(X)} \rb - \lambda^*\B_{0\bu}(x)^\top\e_{\bu}\e_{\bu}^\top\B_{0\bu}(y) - \lambda^*\B_{0\bu}(x)^\top\e_{\bu}E_0 \lb\frac{I(X\leq y)}{h(X)} \rb\\
&&- \lambda^*E_0 \lb\frac{I(X\leq x)}{h(X)}\rb \e_{\bu}^\top\B_{0\bu}(y) -E_0 \lb\frac{I(X\leq x)}{h(X)}\rb E_0 \lb\frac{I(X\leq y)}{h(X)}\rb \\
&=& E_0\lb\frac{I(X\leq x\wedge y)}{h(X)} \rb - \lambda^*\B_{0\bu}(x)^\top\e_{\bu}\e_{\bu}^\top\B_{0\bu}(y) - \lambda^*\B_{0\bu}(x)^\top\e_{\bu}\Lm F_0(y) - \lambda^*\e_{\bu}^\top\B_{0\bu}(y) \Rm\\
&&- E_0 \lb\frac{I(X\leq x)}{h(X)}\rb F_0(y) \\
&=& E_0\lb\frac{I(X\leq x\wedge y)}{h(X)} \rb - F_0(x)F_0(y) - \lambda^*(1-\lambda^*)\B_{0\bu}(x)^\top\e_{\bu}\e_{\bu}^\top\B_{0\bu}(y).
\end{eqnarray*}

The covariances $nCov\lb F_0(x,\bgamma^*),~F_0(y,\bgamma^*) \rb$ and $nCov\lb F_0(x,\bgamma^*),~F_1(y,\bgamma^*) \rb$ can be found in a similar manner. 
For $nCov\lb F_1(x,\bgamma^*),~F_1(y,\bgamma^*) \rb$, 
we have 
\begin{eqnarray*}
&&nCov\lb F_1(x,\bgamma^*),~F_1(y,\bgamma^*) \rb\\
&=&E_0\lb\frac{\omega^2(X)I(X\leq x\wedge y)}{h(X)} \rb -(1-\lambda^*) E_0 \lb \frac{\omega(X)I(X\leq x)}{h(X)} \rb E_0 \lb \frac{\omega(X)I(X\leq y)}{h(X)} \rb \\
&&- \lambda^* E_0 \lb\frac{\omega^2(X)I(X\leq x)}{h(X)} \rb E_0 \lb \frac{\omega^2(X)I(X\leq y)}{h(X)}\rb\\
  &=& E_0\lb\frac{\omega^2(X)I(X\leq x\wedge y)}{h(X)} \rb - F_1(x)F_1(y) - \lambda^*(1-\lambda^*)\B_{1\bu}(x)^\top\e_{\bu}\e_{\bu}^\top\B_{1\bu}(y)
\end{eqnarray*}
and 
\begin{eqnarray*}
&&nCov\lb F_0(x,\bgamma^*),~F_1(y,\bgamma^*) \rb\\
&=&E_0\lb\frac{\omega(X)I(X\leq x\wedge y)}{h(X)} \rb -(1-\lambda^*) E_0 \lb \frac{I(X\leq x)}{h(X)} \rb E_0 \lb \frac{\omega(X)I(X\leq y)}{h(X)} \rb \\
&&- \lambda^* E_0 \lb\frac{\omega(X)I(X\leq x)}{h(X)} \rb E_0 \lb \frac{\omega^2(X)I(X\leq y)}{h(X)}\rb\\
  &=& E_0\lb\frac{\omega(X)I(X\leq x\wedge y)}{h(X)} \rb - F_0(x)F_1(y) - \lambda^*(1-\lambda^*)\B_{0\bu}(x)^\top\e_{\bu}\e_{\bu}^\top\B_{1\bu}(y).
\end{eqnarray*}

In summary, for any $l,s\in\{0,1\}$, we get 
\begin{eqnarray}
\nonumber
 &&nCov\lb F_l(x,\bgamma^*),~F_s(y,\bgamma^*) \rb  \\
 &=& 
  E_0\lb\frac{\omega^{l+s}(X)I(X\leq x\wedge y)}{h(X)} \rb - F_l(x)F_s(y) - \lambda^*(1-\lambda^*)\B_{l\bu}(x)^\top\e_{\bu}\e_{\bu}^\top\B_{s\bu}(y).
  \label{cov.hatF}
\end{eqnarray}
%When $s=l$, we have $nCov\lb F_s(x,\bgamma^*),~F_s(x,\bgamma^*) \rb = nVar\lb F_s(x,\bgamma^*) \rb$.

Next, we consider the cross-terms with $\S^*_n$. We present the calculation of $Cov\lb F_0(x,\bgamma^*),\S^*_n\rb$ as an illustration. Using Lemma \ref{expectation}, we get
\begin{eqnarray*}
&&Cov\lb F_0(x,\bgamma^*),\S_{n\bstheta}\rb\\
&=& \frac{1}{n}Cov\lb \Sumij \frac{I(X_{ij} \leq x)}{h(X_{ij})},~~ \Sumj h_0(X_{1j})\Q(X_{1j})^\top - \sum_{j=1}^{n_0}h_1(X_{0j})\Q(X_{0j})^\top\rb\\
&=& \lambda^*Cov\lb \frac{I(X_{11} \leq x)}{h(X_{11})},~ h_0(X_{11})\Q(X_{11})^\top\rb - (1-\lambda^*) Cov\lb \frac{I(X_{01} \leq x)}{h(X_{01})},~h_1(X_{01})\Q(X_{01})^\top \rb\\
&=& \Lm E_0\{h_0(X)I(X \leq x)\} - \frac{1-\lambda^*}{\lambda^*}E_0\{h_1(X)I(X \leq x)\} \Rm E_0\{h_1(X)\Q(X)^\top\} .
\end{eqnarray*}
It can be checked that 
\begin{eqnarray*}
&&E_0\{h_1(X)\Q(X)\} = \frac{1}{1-\lambda^*} \A_{\bstheta\bstheta}\e_{\bstheta},\\
&&E_0\{h_0(X)I(X \leq x)\} - \frac{1-\lambda^*}{\lambda^*}E_0\{h_1(X)I(X \leq x)\} 
 = -(1-\lambda^*)\B_{0\bu}(x)^\top\e_{\bu}.
\end{eqnarray*}
Then we have 
\begin{equation*}
  Cov\lb F_0(x,\bgamma^*),\S_{n\bstheta}\rb = -\B_{0\bu}(x)^\top\e_{\bu}(\A_{\bstheta\bstheta}\e_{\bstheta})^\top.
\end{equation*}

Similarly, 
\begin{eqnarray*}
&&Cov\lb F_0(x,\bgamma^*),\S_{n\bu}\rb\\
&=& -\frac{1}{n}Cov\lb \Sumij \frac{I(X_{ij} \leq x)}{h(X_{ij})},~~ \Sumij \frac{\G(X_{ij})^\top}{h(X_{ij})}\rb\\
&=& -\lambda^*Cov\lb \frac{I(X_{11} \leq x)}{h(X_{11})},~ \frac{\G(X_{11})^\top}{h(X_{11})}\rb - (1-\lambda^*) Cov\lb \frac{I(X_{01} \leq x)}{h(X_{01})},~\frac{\G(X_{01})^\top}{h(X_{01})} \rb\\
&=& - E_0\lb\frac{I(X\leq x)\G(X)^\top}{h(X)} \rb \\
&&+\frac{1}{1-\lambda^*} \Lm E_0\{h_0(X)I(X \leq x)\} - \frac{1-\lambda^*}{\lambda^*}E_0\{h_1(X)I(X \leq x)\}\Rm E_0\{h_0(X)\G(X)^\top\}\\
&=& - E_0\lb\frac{I(X\leq x)\G(X)^\top}{h(X)} \rb - \B_{0\bu}(x)^\top\e_{\bu}\cdot E_0\{h_0(X)\G(X)^\top\}\\
&=& - \B_{0\bu}(x)^\top + \lambda^*(1-\lambda^*)\B_{0\bu}(x)^\top\e_{\bu} (\A_{\bu\bu}\e_{\bu})^\top,
\end{eqnarray*}
where in the last step we used the facts that 
\begin{equation*}
  \B_{0\bu}(x) = E_0\lb\frac{I(X\leq x)\G(X)}{h(X)}\rb~~\text{and}~~E_0\{h_0(X)\G(X)\} = -\lambda^*(1-\lambda^*)\A_{\bu\bu}\e_{\bu}.
\end{equation*}

Recall that 
\begin{equation*}
  \C = \ba{c}
  \A_{\bstheta\bstheta}\e_{\bstheta}\\ 
  -\lambda^*(1-\lambda^*) \A_{\bu\bu}\e_{\bu}
  \ea.
\end{equation*}
Hence,
\begin{equation*}
  Cov\lb F_0(x,\bgamma^*),\S^*_{n}\rb = -\ba{c} {\bf 0}\\\B_{0\bu}(x) \ea^\top - \B_{0\bu}(x)^{\top}\e_{\bu}\C^\top.
\end{equation*}

The covariance between $F_1(x,\bgamma^*)$ and $\S^*_{n}$ can be found in a similar manner; the details are omitted. We conclude 
that for any $x$ in the support of $F_0$,
\begin{equation*}
  Cov\lb F_l(x,\bgamma^*),\S^*_n\rb
  =  -\ba{c} {\bf 0}\\\B_{l\bu}(x) \ea^\top - \B_{l\bu}(x)^{\top}\e_{\bu}\C^\top,~~l \in \{0,1\}.
\end{equation*}

We now return to the form of $\bSigma(x,y)$. 
Recall that 
\begin{equation*}
  n^{-1}Var(\S_n) = \bGamma = \V - \frac{1}{\lambda^*(1-\lambda^*)}\C\C^\top~~\text{and}~~\U\V^{-1}\C = {\bf 0}.
\end{equation*}
This leads to 
\begin{eqnarray*}
\B^*_l(x)^\top\W\bGamma &=& \B^*_l(x)^\top\V^{-1}\U^\top\J^{-1}\U - \ba{c} {\bf 0}\\\B_{l\bu}(x) \ea^\top - \B_{l\bu}(x)^{\top}\e_{\bu}\C^\top\\
&=& \B^*_l(x)^\top\V^{-1}\U^\top\J^{-1}\U + Cov\lb F_l(x,\bgamma^*),\S^*_n\rb.
\end{eqnarray*}
Consequently, for $l = 0,1$, the summation of the last two terms in $\sigma_{ll}(x,x) $ is 
\begin{eqnarray}
\nonumber&&n^{-1}Var(\B^*_l(x)^\top\W\S^*_n)- 2 Cov\lb F_l(x,\bgamma^*),\B^*_l(x)^\top\W \S^*_n\rb\\
\nonumber&=&\Lm \B^*_l(x)^\top\W \bGamma -2 Cov\lb F_l(x,\bgamma^*),\S^*_n\rb\Rm \W \B^*_l(x)\\
\nonumber&=& \Lm \B^*_l(x)^\top\V^{-1}\U^\top\J^{-1}\U + \ba{c} {\bf 0}\\\B_{l\bu}(x) \ea^\top + \B_{l\bu}(x)^{\top}\e_{\bu}\C^\top \Rm \W \B^*_l(x)\\
&=& \B^*_l(x)^\top\W\B^*_l(x) + \lambda^*(1-\lambda^*)\B_{l\bu}(x)^{\top}\e_{\bu}\e_{\bu}^\top\B_{l\bu}(x).
\label{sigmall.last.two}
\end{eqnarray}
Combining (\ref{cov.hatF}) and (\ref{sigmall.last.two}) leads to 
\begin{eqnarray}
 \sigma_{ll}(x,x)=
  E_0\lb\frac{\omega^{2l}(X)I(X\leq x)}{h(X)} \rb - F_l(x)^2 + \B^*_l(x)^\top\W\B^*_l(x)
  .
  \label{sigmall.xx}
\end{eqnarray}

Using similar steps to derive \eqref{sigmall.last.two}, we find that the summation of the last three terms in $\sigma_{ls}(x,y)$ is
\begin{eqnarray}
\nonumber&&\B^{*}_l(x)^\top\W\bGamma\W\B^*_s(y)- Cov\lb F_l(x,\bgamma^*),\S^*_n\rb\W\B^*_s(y) - \B^{*}_l(x)^\top\W Cov \lb \S^*_n, F_s(y,\bgamma^*)\rb\\
\nonumber&=&\B^*_l(x)^\top\V^{-1}\U\J^{-1}\U^\top\W\B^*_s(y)- \B^*_l(x)^\top\W Cov \lb \S^*_n, F_s(y,\bgamma^*)\rb\\
&=&\B^*_l(x)^\top\W\B^*_s(y) + \lambda^*(1-\lambda^*)\B_{l\bu}(x)^{\top}\e_{\bu}\e_{\bu}^\top\B_{s\bu}(y).\label{sigmals.last.three}
\end{eqnarray}
Combining (\ref{cov.hatF}) and (\ref{sigmals.last.three}) gives 
\begin{eqnarray}
 \sigma_{ls}(x,y )=
  E_0\lb\frac{\omega^{l+s}(X)I(X\leq x\wedge y)}{h(X)} \rb - F_l(x)F_s(y) + \B^*_l(x)^\top\W\B^*_s(y)
  .
  \label{sigmals.xy}
\end{eqnarray}

Summarizing \eqref{sigmall.xx} and \eqref{sigmals.xy}, we conclude that for any $i,j \in \{l,s\}$
\begin{equation}
\label{sigmaij}
     \sigma_{ij}(x,y) = E_0\lb\frac{\omega ^{i+j} (X)I(X\leq x \wedge y)}{h(X)} \rb - F_i(x)F_j(y)+ \B^*_i(x)^\top \W \B^*_j(y),
   \end{equation}
   which is as claimed in the lemma. 
This completes the proof of (a). 
%\begin{eqnarray*}
 %  \sigma_{ll}(x,x) &=& E_0\lb\frac{\delta(X;\bbeta_0)^{2l}I(X\leq x)}{h(X;\bomega_0)} \rb - F_l(x)^2+ \B^{*\top}_l(x) \W \B^*_l(x);\\
 %  \sigma_{ss}(y,y) &=& E_0\lb\frac{\delta(X;\bbeta_0)^{2s}I(X\leq y)}{h(X;\bomega_0)} \rb - F_s(y)^2+ \B^{*\top}_s(y)^\top \H \B^*_s(y);\\
 % \sigma_{ls}(x,y) &=& (1+\rho)d_{ls}(x \wedge y)-F_l(x)F_s(y) + \B_l(x)^\top\H\B_s(y) ;\\
  % \sigma_{ls}(x,y) &=& (1+\rho)d_{sl}(x \wedge y)-F_s(y)F_l(x) + \B_s(y)^\top\H\B_l(x).
%\end{eqnarray*}

For (b): We prove that the claim in (b) is correct for $l=0$ and $s = 1$. The proofs for the other cases are similar and are omitted. 

We first simplify the matrix $\W$. 
Let $\M_q^\top = ( \C, \U^\top)$. Then $\M_q$ is full rank and therefore invertible. 
Note that 
\begin{equation*}
  \V = \M_q^\top (\M_q^\top)^{-1}\V \M_q^{-1} \M_q = \M_q^\top(\M_q \V^{-1}\M_q^\top ) ^{-1} \M_q.
\end{equation*}
Recall that $\U\V^{-1}\C={\bf 0}$ and $\J=\U\V^{-1}\U^\top$. 
Then 
$$
\M_q \V^{-1}\M_q^\top=\left(
\begin{array}{cc}
\C^\top\V^{-1}\C&{\bf 0}\\
{\bf 0}&\J\\
\end{array}
\right)
$$
and 
$$
 \V 
= \C (\C^\top\V^{-1}\C)^{-1}\C^\top + \U^\top\J^{-1}\U.
$$
Note that 
\begin{eqnarray*}
\C^\top\V^{-1}\C &=& \e_{\btheta}^\top\A_{\btheta\btheta}\e_{\btheta} + \{\lambda^*(1-\lambda^*)\}^2\e_{\bu}^\top\A_{\bu\bu}e_{\bu}\\
&=& (1-\lambda^*)E_0\{h_1(X)\}+\{\lambda^*(1-\lambda^*)\}^2E_0\left[ \frac{\{\omega(X)-1\}^2}{h(X)} \right]\\
&=& \lambda^*(1-\lambda^*),
\end{eqnarray*}
where we use the fact that 
\begin{equation*}
  \lambda^*E_0\left[ \frac{\{\omega(X)-1\}^2}{h(X)} \right] + E_0\lb \frac{\omega(X)-1}{h(X)} \rb = 0
\end{equation*}
in the last step. 
The matrix $\V$ is expressed as 
\begin{equation*}
  \V = \{\lambda^*(1-\lambda^*)\}^{-1}\C\C^\top + \U^\top\J^{-1}\U.
\end{equation*}
This expression helps us to simplify $\W$ as  
\begin{eqnarray*}
  \W &= &\V^{-1} \U^\top \J^{-1}\U\V^{-1} - \ba{cc} {\bf 0}&{\bf 0}\\ {\bf 0} & \A_{\bu\bu}^{-1}\ea\\
&=& \V^{-1}\{\U^\top \J^{-1}\U - \V\} \V^{-1} + \ba{cc}\A_{\btheta\btheta}^{-1}&{\bf 0}\\{\bf 0}&{\bf 0}\ea\\
  &=&\ba{cc}\A_{\btheta\btheta}^{-1}&{\bf 0}\\{\bf 0}&{\bf 0}\ea - \{\lambda^*(1-\lambda^*)\}^{-1}\V^{-1}\C\C^\top\V^{-1}\\
  &=&\ba{cc}\A_{\btheta\btheta}^{-1}&{\bf 0}\\{\bf 0}&{\bf 0}\ea - \{\lambda^*(1-\lambda^*)\}^{-1}\ba{c} \e_{\btheta}\\- \lambda^*(1-\lambda^*)\e_{\bu} \ea \ba{c} \e_{\btheta}\\- \lambda^*(1-\lambda^*)\e_{\bu} \ea^\top.
\end{eqnarray*}

Substituting $\W$ into \eqref{sigmaij} and using the fact that
\begin{eqnarray*}
&&\B^*_0(x)^\top\ba{c} \e_{\btheta}\\- \lambda^*(1-\lambda^*)\e_{\bu} \ea = \lambda^*F_0(x),\\
&&\B^*_1(x)^\top\ba{c} \e_{\btheta}\\- \lambda^*(1-\lambda^*)\e_{\bu} \ea = -(1-\lambda^*)F_1(x),
\end{eqnarray*} 
we find that for any $i,j \in \{l,s\}$
\begin{equation*}
\sigma_{ij}(x,y) = E_0\lb\frac{\omega^{i+j}(X)I(X\leq x \wedge y)}{h(X)} \rb + \B_{i\btheta}(x)^{\top} \A_{\btheta\btheta}^{-1} \B_{j\btheta}(y) - \delta_{ij}F_i(x)F_j(y),
\end{equation*}
where 
\begin{equation*}
  \delta_{ij} = \begin{cases} {(1-\lambda^*)}^{-1},& i =j =0\\ (\lambda^*)^{-1},& i =j = 1\\
  0,& i \neq j\end{cases}.
\end{equation*}

This form is the same as that in \cite{chen2013quantile} for the two-sample case, which completes the proof of (b).

For (c): Recall that $\U_m,\V_m,$ and $\J_m$ denote the corresponding $\U,\V,$ and $\J$ matrices obtained by using only the first $m$ EEs of $\g(\x;\etab)$. We further define $\bSigma^{(m)}_{ls}(x,y) = \{\sigma^{(m)}_{ij}(x,y)\}_{i,j\in\{l,s\}}$ and $\B^{*(m)}_l(x)$ to denote the corresponding matrix $\bSigma _{ls}(x,y)$ and vector $\B_l(x)$ obtained by using the first $m$ EEs.

From the definitions of these matrices and vectors, we notice the following relationships:
\begin{equation*}
  \U_m = (\U_{m-1},u_m);~~~ \V_m = \ba{cc} \V_{m-1}& \vartheta_{m-1,m}\\
  \vartheta_{m,m-1}&\vartheta_{m,m}\ea;~~~
  \B^{*(m)}_l(x) = \ba{c} \B^{*(m-1)}_l(x)\\ b_{lm}(x)\ea,
\end{equation*}
where $u_m$, $\vartheta_{m-1,m}$, $\vartheta_{m,m}$, and $b_{lm}(x)$ are the extra terms coming from the $m$th dimension of the EEs.

With the fact that 
\begin{equation*}
   \W = \V^{-1}( \U^\top \J^{-1}\U - \V ) \V^{-1} + \ba{cc}\A_{\btheta\btheta}^{-1}&{\bf 0}\\{\bf 0}&{\bf 0}\ea,
\end{equation*}
the entry in the covariance matrix $\bSigma^{(m)}_{ls}(x,y)$ can be written as 
\begin{eqnarray*}
\sigma^{(m)}_{ij}(x,y) &=& E_0\lb\frac{\omega^{i+j}(X)I(X\leq x \wedge y)}{h(X)} \rb - F_i(x)F_j(y)+ \B^{*(m)}_i(x)^\top \W \B^{*(m)}_j(y)\\
&=& E_0\lb\frac{\omega^{i+j}(X)I(X\leq x \wedge y)}{h(X)} \rb - F_i(x)F_j(y)+ \B_{i\btheta}(x)^\top \A_{\btheta\btheta}^{-1} \B_{j\btheta}(y)\\
&& - \B_i^{*(m)}(x)^\top
 \V_m^{-1} ( \V_{m} - \U_m^\top\J_m^{-1}\U_m ) \V_m^{-1}
 \B_j^{*(m)}(x)
\end{eqnarray*}
for any $i,j \in \{l,s\}$.

Therefore, 
\begin{eqnarray*}
&&\bSigma^{(m-1)}_{ls}(x,y)
- 
\bSigma^{(m)}_{ls}(x,y)\\
&=&
\left(
\begin{array}{c}
\B_l^{*(m)}(x)\\
\B_s^{*(m)}(y)\\
\end{array}
\right)^\top( \V_{m} - \U_m^\top\J_m^{-1}\U_m ) \V_m^{-1}
\left(
\begin{array}{c}
\B_l^{*(m)}(x)\\
\B_s^{*(m)}(y)\\
\end{array}
\right) \\
&&-
 \left(
\begin{array}{c}
\B_l^{*(m-1)}(x)\\
\B_s^{*(m-1)}(y)\\
\end{array}
\right)^\top( \V_{m-1} - \U_{m-1}^\top\J_{m-1}^{-1}\U_{m-1} ) \V_{m-1}^{-1}
\left(
\begin{array}{c}
\B_l^{*(m-1)}(x)\\
\B_s^{*(m-1)}(y)\\
\end{array}
\right). 
\end{eqnarray*}

Using the results in \eqref{coro1_V_inequality} and \eqref{coro1_J_inequality}, we have 
\begin{eqnarray*}
&& \V_m^{-1}\{\V_{m} - \U_m^\top\J_m^{-1}\U_m\}\V_m^{-1}\\
 &\geq& \ba{cc}
\V_{m-1}^{-1}&{\bf 0}\\{\bf 0}&{0}
\ea  
\left\{ \ba{cc} \V_{m-1}& \vartheta_{m-1,m}\\
\vartheta_{m,m-1}&\vartheta_{m,m}\ea -\ba{c}\U_{m-1}^\top\\u_m^\top\ea \J_{m-1}^{-1}(\U_{m-1},u_m)
\right\}
\ba{cc}
\V_{m-1}^{-1}&{\bf 0}\\{\bf 0}&{0}
\ea\\
&\geq&
\ba{cc}
\V_{m-1}^{-1}\{\V_{m-1} - \U_{m-1}^\top\J_{m-1}^{-1}\U_{m-1}\}\V_{m-1}^{-1}&{\bf 0}\\{\bf 0}&{0}
\ea .
\end{eqnarray*}
This implies that 
\begin{eqnarray*}
&&\bSigma^{(m-1)}_{ls}(x,y)
- 
\bSigma^{(m)}_{ls}(x,y)\\
&\geq&
\left(
\begin{array}{c}
\B_l^{*(m)}(x)\\
\B_s^{*(m)}(y)\\
\end{array}
\right)^\top \ba{cc}
\V_{m-1}^{-1}\{\V_{m-1} - \U_{m-1}^\top\J_{m-1}^{-1}\U_{m-1}\}\V_{m-1}^{-1}&{\bf 0}\\{\bf 0}&{0}
\ea 
\left(
\begin{array}{c}
\B_l^{*(m)}(x)\\
\B_s^{*(m)}(y)\\
\end{array}
\right) \\
&&-
 \left(
\begin{array}{c}
\B_l^{*(m-1)}(x)\\
\B_s^{*(m-1)}(y)\\
\end{array}
\right)^\top( \V_{m-1} - \U_{m-1}^\top\J_{m-1}^{-1}\U_{m-1} ) \V_{m-1}^{-1}
\left(
\begin{array}{c}
\B_l^{*(m-1)}(x)\\
\B_s^{*(m-1)}(y)\\
\end{array}
\right)\\
&=&{\bf 0}.
\end{eqnarray*}
This completes the proof of (c). 

\section{Proof of Theorem 5}

We first introduce two lemmas that will be helpful in the proof of Theorem 5.
The following lemma establishes the convergence rate of $\hat{\xi}_{i,\tau}$.  

\begin{lemma}
\label{lemma1.quantile}
Assume the conditions of Theorem 5 are satisfied. For each fixed $\tau\in(0,1)$ and $i = 0,1$, we have 
$$\hat{\xi}_{i,\tau} - \xi_{i,\tau} = O_p(n^{-1/2}).$$ 
\end{lemma}

\begin{proof}
We concentrate on the case $i=0$; the case $i=1$ can be proved similarly. 
Let $ \Delta_{n} = \sup_x \vert\hat{F}_0(x) - F_0(x)\vert $. 
It suffices to show that \citep{chen2013quantile, chen2019composite}
\begin{equation}
\label{deltan}
  \Delta_{n}= O_p(n^{-1/2}).
\end{equation}

Define
$$
\bar F_0(x)= \frac{1}{n}\Sumij \frac{I(X_{ij} \leq x)}
{1+\lambda^* \left[ \exp\{{\hat\btheta}^\top \Q(X_{ij})\}-1\right] }.
$$
Then 
\begin{eqnarray*}
 \Delta_n &=& \sup_x \vert\hat{F}_0(x) - F_0(x)\vert
 \leq \sup_x \vert\hat{F}_0(x)- \bar F_0(x)\vert + \sup_x\vert \bar F_0(x) - F_i(x)\vert= \Delta_{n1}+\Delta_{n2}, \end{eqnarray*}
where 
$$
\Delta_{n1}= \sup_x \vert\hat{F}_0(x)- \bar F_0(x)\vert
$$
and 
$$
\Delta_{n2}=\sup_x\vert \bar F_0(x) - F_0(x)\vert. 
$$

Following the proof of Theorem 3.1 in \cite{chen2013quantile} and Lemma 1 in \cite{chen2019composite}, we can verify that 
$$\Delta_{n2} = O_p(n^{-1/2}). $$
With this result, the claim \eqref{deltan} is proved if 
$\Delta_{n1} = O_p(n^{-1/2})$.

As preparation, we argue that 
\begin{equation}
\label{lower.denom}
(n\hat p_{ij})^{-1}
=
 1+\hat \lambda[\exp\{{\hat\btheta}^\top \Q(X_{ij})\}-1] + \hat{\bnu}^\top\g(X_{ij};\hat{\bpsi},\hat{\btheta}) \geq 1-\lambda^*+o_p(1)
\end{equation}
or equivalently $\hat p_{ij}\leq n^{-1}\{1-\lambda^*+o_p(1)\}^{-1}=O_p(1/n) $.
Note that 
\begin{eqnarray*}
 (n\hat p_{ij})^{-1}&\geq&1-\hat\lambda+\hat{\bnu}^\top\g(X_{ij};\hat{\bpsi},\hat{\btheta})\geq 1-\hat\lambda -
 \lVert \hat \bnu \rVert  \max_{ij} \lVert \g(X_{ij};\hat{\bpsi},\hat{\btheta})\rVert . 
\end{eqnarray*}
By Condition C5, 
$$
 \max_{ij} \lVert \g(X_{ij};\hat{\bpsi},\hat{\btheta})\rVert\leq \max_{ij} R^{1/3}(X_{ij})=o_p(n^{1/2}),
$$
which, together with $\hat\bgamma-\bgamma^*=O_p(n^{-1/2})$, implies that \eqref{lower.denom} is valid. 

We now return to argue that $\Delta_{n1}=O_p(n^{-1/2})$. 
After some algebra, we have 
\begin{eqnarray*}
 && \hat{F}_0(x)- \bar F_0(x)\\
 & =&\Sumij\hat p_{ij}\frac{
 (\lambda^*-\hat\lambda) \left[ \exp\{{\hat\btheta}^\top \Q(X_{ij})\}-1\right]
 -\hat{\bnu}^\top\g(X_{ij};\hat{\bpsi},\hat{\btheta}) }
 { 1+\lambda^* \left[ \exp\{{\hat\btheta}^\top \Q(X_{ij})\}-1\right] }I(X_{ij}\leq x). 
 \end{eqnarray*}
 Using \eqref{lower.denom}, we have 
\begin{eqnarray}
\nonumber
 |\hat{F}_0(x)- \bar F_0(x)|
 &\leq& O_p(1/n) \Sumij 
\frac{
| \hat \lambda-\lambda^*|  \left[ \exp\{{\hat\btheta}^\top \Q(X_{ij})\}+1\right]
}
 { 1+\lambda^* \left[ \exp\{{\hat\btheta}^\top \Q(X_{ij})\}-1\right] }I(X_{ij}\leq x)\\
 \nonumber &&+O_p(1/n) \Sumij \frac{
 |\hat{\bnu}^\top\g(X_{ij};\hat{\bpsi},\hat{\btheta})|
}
 { 1+\lambda^* \left[ \exp\{{\hat\btheta}^\top \Q(X_{ij})\}-1\right] }I(X_{ij}\leq x)
\\
\nonumber&\leq& O_p(1/n) \Sumij 
\frac{
| \hat \lambda-\lambda^*|  
}
 { \lambda^* (1-\lambda^*) }I(X_{ij}\leq x)\\
 &&+O_p(1/n) \Sumij \frac{
 |\hat{\bnu}^\top\g(X_{ij};\hat{\bpsi},\hat{\btheta})|
}
 { 1-\lambda^* }I(X_{ij}\leq x).\label{deltan1.part1}
 \end{eqnarray}
By Condition C5, 
$$
\Delta_{n1}
=\sup_{x} |\hat{F}_0(x)- \bar F_0(x)|
\leq O_p(1) | \hat \lambda-\lambda^*|
+O_p(1) 
 \frac{1}{n}\Sumij\left\{\lVert \hat{\bnu}\rVert R^{1/3}(X_{ij}) \right\},
$$
which, together 
with $\hat\bgamma-\bgamma^*=O_p(n^{-1/2})$, implies that 
\begin{equation*}
  \Delta_{n1} = O_p(n^{-1/2}).
\end{equation*}
This completes the proof. 
\end{proof}

\begin{lemma}
\label{lemma2.quantile}
Under the regularity conditions, for any $c>0$ and $i = 0,1$, we have 
\begin{equation*}
  \sup \limits_{x:~\vert x-\xi_{i,\tau}\vert<c n^{-1/2}} \vert \{\hat{F}_0(x) - \hat{F}_0(\xi_{i,\tau})\}-\{F_0(x) - F_i(\xi_{i,\tau})\}\vert
  =O_p(n^{-3/4}(\log(n))^{1/2}).
\end{equation*}
\end{lemma}

\begin{proof}
We prove this lemma for $i=0$; the case $i=1$ is equivalent. 
Without loss of generality we assume $x \geq \xi_{0,\tau}$. Note that 
\begin{eqnarray}
\nonumber
  && \vert \{\hat{F}_0(x) - \hat{F}_0(\xi_{0,\tau})\}-\{F_0(x) - F_0(\xi_{0,\tau})\}\vert \\\nonumber
  &\leq& \vert \{\hat{F}_0(x) - \hat{F}_0(\xi_{0,\tau})\}-\{\bar F_0(x) - \bar F_0(\xi_{0,\tau})\}\vert\\ 
  \label{suphatF}
  &&+ \vert \{\bar F_0(x) - \bar F_0(\xi_{0,\tau})\}-\{F_0(x) - F_0(\xi_{0,\tau})\}\vert.
\end{eqnarray}

Following the proof of Lemma A.2 in 
\cite{chen2013quantile}, we can verify that 
\begin{equation*}
  \sup \limits_{x:~0\leq x-\xi_{0,\tau}<cn^{-1/2}}\vert \{ \bar F_0(x) - \bar F_0(\xi_{0,\tau})\}-\{F_0(x) - F_0(\xi_{0,\tau}) \}\vert = O_p(n^{-3/4}(\log(n))^{1/2}).
\end{equation*}
Consequently, we need to show only that the first term in \eqref{suphatF} has a higher order than $n^{-3/4}(\log(n))^{1/2}$ uniformly in $0\leq x-\xi_{0,\tau}<cn^{-1/2}$. 

With the technique used to obtain \eqref{deltan1.part1}, 
we have 
\begin{eqnarray*}
&&\vert \{\hat{F}_0(x) - \hat{F}_0(\xi_{0,\tau})\}-\{\bar F_0(x) - \bar F_0(\xi_{0,\tau})\} \vert\\ 
&\leq& O_p(1/n) \Sumij 
\frac{
| \hat \lambda-\lambda^*|  
}
 { \lambda^* (1-\lambda^*) }I(\xi_{0,\tau}< X_{ij}\leq x)\\
 &&+O_p(1/n) \Sumij \frac{
\lVert \hat{\bnu}\rVert R^{1/3}(X_{ij}) 
}
 { 1-\lambda^* }I(\xi_{0,\tau}< X_{ij}\leq x)\\
 &=&O_p(n^{-1/2}) \frac{1}{n } \Sumij \{1+R^{1/3}(X_{ij})  \}I(\xi_{0,\tau}< X_{ij}\leq x). 
\end{eqnarray*}

By Condition C5, 
$$
E_0\{1+R^{1/3}(X)  \}<\infty ~~
\mbox{and}~~
E_1 \{1+R^{1/3}(X)  \}<\infty,
$$ 
then uniformly in $x$
$$
E_0[\{1+R^{1/3}(X)  \} I(\xi_{0,\tau}< X_{ij}\leq x)]=O_{p}(n^{-1/2})
$$
and 
$$
E_1[\{1+R^{1/3}(X)  \} I(\xi_{0,\tau}< X_{ij}\leq x)]=O_{p}(n^{-1/2}).
$$
Therefore, 
\begin{eqnarray*}
\sup \limits_{x:~0\leq x-\xi_{0,\tau}<c n^{-1/2}} \vert \{\hat{F}_0(x) - \hat{F}_0(\xi_{0,\tau})\}-\{\bar F_0(x) - \bar F_0(\xi_{0,\tau})\} \vert =O_p(n^{-1}). 
\end{eqnarray*}
This completes the proof. 
\end{proof}

We are now ready to prove Theorem 5.
By Lemma \ref{lemma1.quantile}, for $i = 0,1$,
\begin{eqnarray}
\nonumber
F_i(\hat{\xi}_{i,\tau}) - F_i(\xi_{i,\tau})
&=& f_i(\xi_{i,\tau})(\hat{\xi}_{i,\tau} - \xi_{i,\tau})+ O_p((\hat{\xi}_{i,\tau} - \xi_{i,\tau})^2)\\
\label{taylorHatF}
&=&f_i(\xi_{i,\tau})(\hat{\xi}_{i,\tau} - \xi_{i,\tau})+ O_p(n^{-1}).
\end{eqnarray}
Note that $\hat{F}_i(\hat{\xi}_{i,\tau})=\tau + O(n^{-1})$. Replacing $x$ by $\hat{\xi}_{i,\tau}$ in Lemma \ref{lemma2.quantile} and using \eqref{taylorHatF} yields
\begin{equation*}
  \tau - \hat{F}_i(\xi_{i,\tau}) = 
  f_i(\xi_{i,\tau})(\hat{\xi}_{i,\tau} - \xi_{i,\tau}) + 
  O_p (n^{-3/4}(\log(n))^{1/2}). 
\end{equation*}
This completes the proof.

\section{Proof of Theorem 6}
The results in (a) and (b) are direct consequences of Theorems 4 and 5. 

For (c): We note that 
$$
  \bOmega_{ls}
  =\left(
  \begin{array}{cc}
  1/f_l(\xi_{l,\tau_l})&0\\
  0&1/f_s(\xi_{s,\tau_s})\\
  \end{array}
  \right)
  \bSigma_{ls}(\xi_{l,\tau_l}, \xi_{s,\tau_s})
  \left(
  \begin{array}{cc}
  1/f_l(\xi_{l,\tau_l})&0\\
  0&1/f_s(\xi_{s,\tau_s})
  \end{array}
  \right). 
$$
Then Theorem 4(c) implies the results in (c). 
This completes the proof. 

\medskip %adjust for the distance between chapter and Bibliography
%\bibliographystyle{agsm}
%\bibliography{reference}

\end{document}